\documentclass[11pt]{amsart}
 
\usepackage{tikz}
\usetikzlibrary{cd}

\usepackage{fullpage}
\usepackage{amsfonts,amssymb}

\usepackage{pinlabel,url}
\usepackage[all]{xy}
\usepackage{tikz}
\usepackage{graphicx}
\usepackage{hyperref}
\usepackage{subfig}

\usepackage[nameinlink]{cleveref}

\theoremstyle{theorem}
\newtheorem{theorem}{Theorem}[section]

\newtheorem{lemma}[theorem]{Lemma}
\newtheorem{proposition}[theorem]{Proposition}
\newtheorem{corollary}[theorem]{Corollary}
\newtheorem{claim}[theorem]{Claim}

\theoremstyle{definition}

\newtheorem{defn}[theorem]{Definition}
\newtheorem*{remark}{Remark}

\newtheorem{example}[theorem]{Example}
\newtheorem{definition}[theorem]{Definition}

\newtheorem{convention}[theorem]{Convention}
\newtheorem{construction}[theorem]{Construction}

\numberwithin{equation}{section}

\newcommand{\Vol}{{\rm Vol}}
\newcommand{\Map}{{\rm Map}}
\newcommand{\PMap}{{\rm PMap}}
\newcommand{\C}{\mathcal{C}}
\newcommand{\B}{\mathcal{B}}
\newcommand{\intt}{{\rm{Int}}}
\newcommand{\PSL}{{\rm PSL}}
\newcommand{\iVol}{{\rm {V\underline{ol}}}}

\newcommand{\pare}{\mathfrak{p}}
\newcommand{\diam}{{\rm diam}}
\newcommand{\supp}{{\rm supp}}
\newcommand{\oct}{{\rm oct}}
\newcommand{\Area}{{\rm Area}}
\newcommand{\tet}{{\rm tet}}

\title{End-periodic homeomorphisms and volumes of mapping tori}

\author[E. Field, H. Kim, C. Leininger, and M. Loving]{Elizabeth Field, Heejoung Kim, Christopher Leininger, and Marissa Loving}

\date{\today}

\begin{document}

\maketitle

\begin{abstract}
Given an irreducible, end-periodic homeomorphism $f: S \to S$ of a surface with finitely many ends, all accumulated by genus, the mapping torus, $M_f$, is the interior of a compact, irreducible, atoroidal 3-manifold $\overline M_f$ with incompressible boundary. Our main result is an upper bound on the infimal hyperbolic volume of $\overline M_f$ in terms of the translation length of $f$ on the pants graph of $S$. This builds on work of Brock and Agol in the finite-type setting. We also construct a broad class of examples of  irreducible, end-periodic homeomorphisms and use them to show that our bound is asymptotically sharp.
\end{abstract}

\section{Introduction}

Infinite-type surfaces arise naturally as leaves of foliations of 3-manifolds. In the case of taut, depth-one foliations of closed, hyperbolic $3$--manifolds, any noncompact leaf is a surface $S$ of infinite genus with finitely many ends (all accumulated by genus).  The leaf $S$ is dense in an open submanifold which is the mapping torus $M_f$ of an {\em end-periodic} homeomorphism $f \colon S \to S$.
As such, the study of homeomorphisms of infinite-type surfaces is intimately related to the study of foliated 3-manifolds. 
The goal of this paper is to relate the hyperbolic geometry of $M_f$ to the geometry and dynamics of the monodromy $f$ in a way that mirrors the connections between properties of pseudo-Anosov homeomorphisms of finite-type surfaces and the geometry of their mapping tori.

The motivating result from the finite-type setting is due to Brock \cite{Brock-mappingtorus-vol}, who related the translation length, $\tau(f)$, of a pseudo-Anosov homeomorphism $f$ on the pants graph with the hyperbolic volume of its mapping torus.
In particular, he proved that there is some $K >0$ so that
\[\frac{1}{K}\tau(f)\leq \Vol(M_f) \leq K \tau(f),\]
where $K$ depends only on the surface.  Agol \cite{Agol-oct} improved the upper bound on volume, giving an explicit value of the constant, and moreover proved that it is sharp.  

For an end-periodic homeomorphism $f \colon S \to S$, the mapping torus $M_f$ is the interior of a compact, irreducible $3$--manifold $\overline M_f$ with incompressible boundary (\Cref{proposition:interior}). Handel and Miller proved a Nielsen-Thurston type classification theorem in the end-periodic setting, with the notion of {\em irreducibility} serving as an analogue of being pseudo-Anosov; see \cite{CC-book} or \Cref{S:end-periodic defs}.  We prove that when $f$ is irreducible, $\overline M_f$ is atoroidal (\Cref{P:its hyperbolic}), and further define a notion of {\em strong irreducibility} which ensures $\overline M_f$ is also acylindrical (\Cref{L:annuli analysis}).  Thurston's Hyperbolization Theorem (see \Cref{Theorem:Thurston.Geometrization}) then reinforces the analogy with the Nielsen-Thurston classification in the finite-type case, providing a \emph{convex} hyperbolic metric on $\overline M_f$ when $f$ is irreducible, and one with totally geodesic boundary when $f$ is strongly irreducible. Note that \emph{every} 3-manifold with boundary admits a non-convex hyperbolic metric so it is important to make this distinction. We write $\iVol(\overline M_f)$ to denote the infimum of volumes of all such hyperbolic metrics on $\overline M_f$. Our first theorem is an analogue of Brock's upper bound on volume and our approach follows Agol's proof in \cite{Agol-oct}.


\newcommand{\upperboundtheorem}{For any irreducible, end-periodic homeomorphism $f \colon S \to S$ of a surface with finitely many ends, all accumulate by genus, we have $\iVol(\overline M_f) \leq V_\oct \tau(f)$.}
\begin{theorem}\label{T:upper} 
\upperboundtheorem
\end{theorem}

Here, $V_\oct$ denotes the volume of a regular ideal octahedron, and $\tau(f)$ denotes the asymptotic translation length of $f$ on the pants graph $\mathcal P(S)$.  
Throughout this paper, we assume that $S$ has finitely many ends, all accumulated by genus.  Although $\mathcal P(S)$ is disconnected for such surfaces, the end-periodic assumption implies that there is a non-empty subgraph, $\mathcal P_f(S) \subset \mathcal P(S)$, for which each component is invariant by a positive power of $f$, and hence $\tau(f)$ is well-defined (see \Cref{S:associated-graphs} where we also discuss the path metric on the components of $\mathcal P(S)$ we consider).

For any component $\Omega \subset \mathcal P_f(S)$ which is invariant by $f$, we will construct a pants decomposition $P_\Omega$ of $\partial \overline M_f$ for which $\overline M_f - P_\Omega$ is acylindrical, and $P_\Omega$ depends only on the component $\Omega$; see \Cref{P:block decomposition}. Denoting the volume of the complete hyperbolic metric with totally geodesic boundary on this manifold by $\Vol(\overline M_f -P_\Omega)$, which by a result of Storm \cite{Storm2} is equal to the infimum of volumes of all convex hyperbolic metrics on $\overline M_f - P_\Omega$, and letting $\tau(f, \Omega)$ denote the translation distance of $f$ on $\Omega$, we prove the following.
\newcommand{\Omegaupperboundtheorem}{For any $f$--invariant component $\Omega \subset \mathcal P_f(S)$,\[ \Vol(\overline M_f-P_\Omega) \leq V_\oct\tau(f,\Omega).\]}

\begin{theorem} \label{T:component volume}
\Omegaupperboundtheorem
\end{theorem}
The translation length of $f$ on various components of $\mathcal P(S)$ is quite mysterious as the following corollary of \Cref{T:component volume} shows.
\newcommand{\largedistance}{For any irreducible, end-periodic homeomorphism $f \colon S \to S$ and $R >0$, there exists an $f$--invariant component $\Omega \subset \mathcal P_f(S)$ such that $\tau(f,\Omega) \geq R$.}
\begin{corollary} \label{C:large distance}
\largedistance
\end{corollary}

\Cref{T:component volume} also provides uniform lower bounds on asymptotic translation length. To state the result in this case, we let $V_\tet$ denote the volume of a regular ideal tetrahedron in $\mathbb H^3$ (which is also the maximal volume of a tetrahedron in $\mathbb H^3$).

\newcommand{\lowertranslation}{Given an irreducible, end-periodic homeomorphism $f$, we have
\[ \tau(f) \geq \frac{V_\tet \xi(\partial \overline M_f)}{2 V_\oct}.\]}
\begin{corollary} \label{C:lower bound translation}
\lowertranslation
\end{corollary}

Here $\xi(\partial \overline M_f)$ is the complexity of the boundary, see \Cref{S:surfaces-and-mcg}.
For a description of the lower bound which is more intrinsic to the surface $S$ and homeomorphism $f$, see
\Cref{C:intrinsic lower translation}. 

We end by describing a fairly robust construction for examples of both strongly irreducible homeomorphisms and irreducible, but not strongly irreducible, homeomorphisms.

\newcommand{\tblah}{Given any pure end-periodic homeomorphism $\phi \colon S \to S$ there exist both irreducible and strongly irreducible homeomorphisms $f,f' \colon S \to S$, respectively, so that $f$ and $f'$ agree with $\phi$ on the complement of a compact set.}
\begin{theorem} \label{T:blah}
\tblah
\end{theorem}

The construction is sufficiently robust to prove the following. Note that when an end-periodic homeomorphism $f$ is strongly irreducible,  $\overline M_f$ is acylindrical (see \Cref{proposition:interior}), and so the infimal volume is realized by $\Vol(\overline M_f)$, which denotes the volume of the complete hyperbolic metric with totally geodesic boundary.

\newcommand{\tblahblah}{There is a sequence of strongly irreducible end-periodic homeomorphisms $f_k \colon S \to S$ so that $\Vol(\overline M_{f_k}) \to \infty$ and $\frac{\Vol(\overline M_{f_k})}{\tau(f_k)} \to V_\oct$ as $k \to \infty$. In fact,  
$|\Vol(\overline M_{f_k}) - V_\oct\tau(f_k)|$ is uniformly bounded, independent of $k$.}
\begin{theorem} \label{T:blah blah}
\tblahblah
\end{theorem}

The examples in this theorem can also be taken to agree with any given end-periodic homeomorphism $\phi \colon S \to S$ outside some compact set.

\subsection*{Outline of the paper}

In \Cref{Section:Preliminaries}, we provide background on mapping class groups for finite-type and infinite-type surfaces, end-periodic homeomorphisms, pants graphs, decomposition spaces, and 3-manifolds. The mapping torus of an irreducible end-periodic homeomorphism is shown to be the interior of a compact, irreducible, atoroidal 3-manifold in \Cref{section:Compactification}. Furthermore, for a strongly irreducible end-periodic homeomorphism, we show that the compactified mapping torus is also acylindrical. 
In \Cref{Section:quotient construction}, we prove \Cref{T:component volume}, \Cref{T:upper}, and \Cref{C:lower bound translation} as applications of \Cref{P:block decomposition}, which is proved in \Cref{Section: Proof of Pro4.3}. Finally, we construct examples of irreducible (and strongly irreducible) end-periodic homeomorphisms and prove \Cref{T:blah blah} and \Cref{C:large distance} in \Cref{S:example}.

\subsection{Related and future work}
Brock in \cite{Brock-convex-vol} proved that the pants graph is quasi-isometric to the Teichm\"uller space with its Weil-Petersson metric, which allows one to compare volumes of mapping tori with Weil-Petersson translation lengths. A direct proof of such a comparison was later given by Brock--Bromberg \cite{BrockBromberg} and Kojima--McShane \cite{KojimaMcShane}.  Pants distance has also been related to volumes of hyperbolic 3-manifolds in another setting by  Cremaschi,  Rodr\'{\i}guez-Migueles and Yarmola \cite{CR-MY}, see also \cite{Rodriguez-Migueles,CremRodMig}. In forthcoming work, Landry, Minsky, and Taylor study stretch factors of end-periodic homeomorphisms arising from depth one foliations  \cite{LandryMinskyTaylor2021}. Finally, with Autumn Kent, the authors are currently investigating lower bounds on volumes in terms of translation lengths on the pants graph.

\subsection*{Acknowledgements}
The authors would like to thank Sergio Fenley for many useful discussions and suggestions, and for pointing out an error in an earlier argument which ultimately led to the definition of strong irreducibility. They would also like to thank Ty Ghaswala for comments on an earlier draft and some interesting and useful observations. The authors acknowledge support from NSF grants DMS-1840190 (Field), DMS-1811518 (Leininger), and DMS-1902729 (Loving).

\section{Preliminaries}\label{Section:Preliminaries}

\subsection{Surfaces and mapping class groups} \label{S:surfaces-and-mcg}

Let $Y$ be any surface. We will consider the \emph{mapping class group}, $\Map(Y)$, of $Y$, which is the group of orientation preserving homeomorphisms of $Y$ up to isotopy. We do not distinguish the notation for a homeomorphism and the mapping class it defines, as the context will make clear the intended meaning.

Throughout this paper, we let $S$ denote a boundaryless, infinite-type surface with $2 \leq n< \infty$ ends, all accumulated by genus (in particular, we assume $S$ has no planar ends). The mapping class group of an infinite-type surface is an uncountably infinite group and is often called a \emph{big mapping class group}. Important subgroups of big mapping class groups include the \textit{pure mapping class group}, $\PMap(S)$, which is the subgroup of $\Map(S)$ consisting of mapping classes which fix each end of $S$, as well as the \emph{compactly supported mapping class group}, $\Map_c(S)$. For more on infinite-type surfaces and their associated big mapping class groups, see the survey article on these topics by Aramayona and Vlamis \cite{AramayonaVlamis}.

We define the \emph{complexity} of a finite-type surface $\Sigma = \Sigma_{g,n}$ of genus $g$ and with $n$ punctures to be $\xi(\Sigma_{g,n}) = 3g - 3 +n$. We then let $\overline \Sigma = \overline \Sigma_{g,n}$ denote the compact surface of genus $g$ with $n$ boundary components, which we declare to have the same complexity as $\Sigma$. For example, we will call $\overline \Sigma_{0,3}$ a \emph{pair of pants} (compact) and $\Sigma_{0,3}$ a \emph{thrice-punctured sphere} (non-compact).

By a finite-type subsurface of $S$, we mean a connected, finite-complexity, incompressible (i.e.~$\pi_1$--injective) open subsurface $\Sigma \subset S$ with $\xi(\Sigma) \neq 0$. Such a surface $\Sigma$ is homeomorphic to the interior of a compact surface $\overline \Sigma$ with boundary, and we assume that the inclusion extends to a locally injective map of $\overline \Sigma$. We abuse notation and write $\overline \Sigma \subset S$, even though it is only assumed to be embedded on its interior.

Let $Y$ be a surface, either compact or non-compact, possibly of infinite type. By a {\em curve} in $Y$, we mean the homotopy class of an essential (i.e.~non-null-homotopic and non-peripheral) simple closed curve in $Y$. A {\em line} in $Y$ is the proper homotopy class of a proper embedding of $\mathbb R$ into $Y$ which is essential (i.e.~not homotopic via a proper homotopy into an arbitrary neighborhood of an end). By \emph{proper homotopy} we mean that the homotopy itself is a proper map this implies that it is also a homotopy through proper maps. An {\em arc} in $Y$ is the relative homotopy class of an essential arc $(I,\partial I) \to (Y,\partial Y)$ (i.e.~an embedding not homotopic into $\partial Y$ by a relative homotopy $h_t \colon (I,\partial I) \to (Y,\partial Y)$).
We often confuse curves, arcs, and lines with representatives of their (relative/proper) homotopy classes. 
A {\em multicurve}, {\em multiarc}, or {\em multiline} in $Y$ is the union of a collection of curves, arcs, or lines, respectively, in $Y$, having pairwise disjoint representatives.

Given a curve $\alpha$ in $Y$, we let $T_\alpha$ denote the left {\em Dehn twist} about the curve $\alpha$. A {\em pseudo-Anosov mapping class} of a finite-type surface $Y$ is a mapping class with no periodic curves in $Y$. A {\em partial pseudo-Anosov} mapping class of any surface $Y$ is a mapping class that has a representative supported on a (finite-type) subsurface $\Sigma \subset Y$ which is pseudo-Anosov on $\Sigma$. We call $\Sigma$ the {\em support} of such a partial pseudo-Anosov. In general, we say that a mapping class is supported on a subsurface $\Sigma$ if it has a representative which is the identity outside of $\Sigma$. See \cite{Farb.Margalit} for more on mapping class groups of finite-type surfaces and the Nielsen--Thurston classification.

\subsection{End-periodic homeomorphims} \label{S:end-periodic defs}
We refer the reader to \cite{FenleyCT,Fenley-depth-one,CC-book,Gabai-genera} for a more detailed discussion of the theory of end-periodic homeomorphisms, examples, and their relationship to depth-one foliations. We note that the exposition in these references discusses end-periodic homeomorphisms of more general types of surfaces than what we consider here.

\begin{definition} [End-periodic homeomorphisms] \label{D:end-periodic} 
An \emph{end-periodic homeomorphism} of $S$ is a homeomorphism $f$ of $S$ satisfying the following. There exists $m > 0$ such that for each end $E$ of $S$, there is a neighborhood $U_E$ of $E$ so that either
\begin{itemize}
    \item[(i)] $f^m(U_E) \subsetneq U_E$ and the sets $\{f^{nm}(U_E)\}_{n >0}$ form a neighborhood basis of $E$; or
    \item[(ii)] $f^{-m}(U_E) \subsetneq U_E$ and the sets $\{f^{-nm}(U_E)\}_{n > 0}$ form a neighborhood basis of $E$.
\end{itemize}
In the first case, $E$ is said to be an \textit{attracting end} for $f$, and in the second case, $E$ is said to be a \textit{repelling end}. We call such a neighborhood $U_E$ a {\em nesting neighborhood} for $E$, and note that by definition, we may always assume (as we will) that the closures of the nesting neighborhoods of the ends are pairwise disjoint.
\end{definition}

There are many examples of end-periodic homeomorphisms given by Cantwell--Conlon in \cite{CC-examples}. It is also straightforward to build more examples using ``handle shifts", which were introduced by Patel--Vlamis in \cite{PatelVlamis2018}.

\begin{definition}
Take a bi-infinite strip $\mathbb R \times [0, 1]$ and remove disks of radius $\tfrac14$ centered at each point of $\mathbb Z \times \{\tfrac12\}$. To the boundary of each disk (now removed) glue the boundary of a one-holed torus. Call this bi-infinite strip of handles $H$. Let $h$ be the homeomorphism of $H$ given by shifting each handle over by one while tapering to the identity in a neighborhood of the boundary (which consists of two copies of $\mathbb R$). We will call the pair $(H, h)$ a \emph{handle strip}, as shown in \Cref{fig:handleshift}. Note that $H$ has two ends accumulated by genus: one towards which $h$ is shifting, called the \textit{attracting end}, and one from which $h$ shifts away, called the \textit{repelling end}. Now embed $H$ into any surface $S$ with at least one end accumulated by genus and extend $h$ by the identity to a homeomorphism, $\rho$, of $S$. The homeomorphism $\rho$ of $S$ is called a \emph{handle shift}.
\end{definition}

\begin{figure}
    \centering
    \includegraphics[width = 8cm]{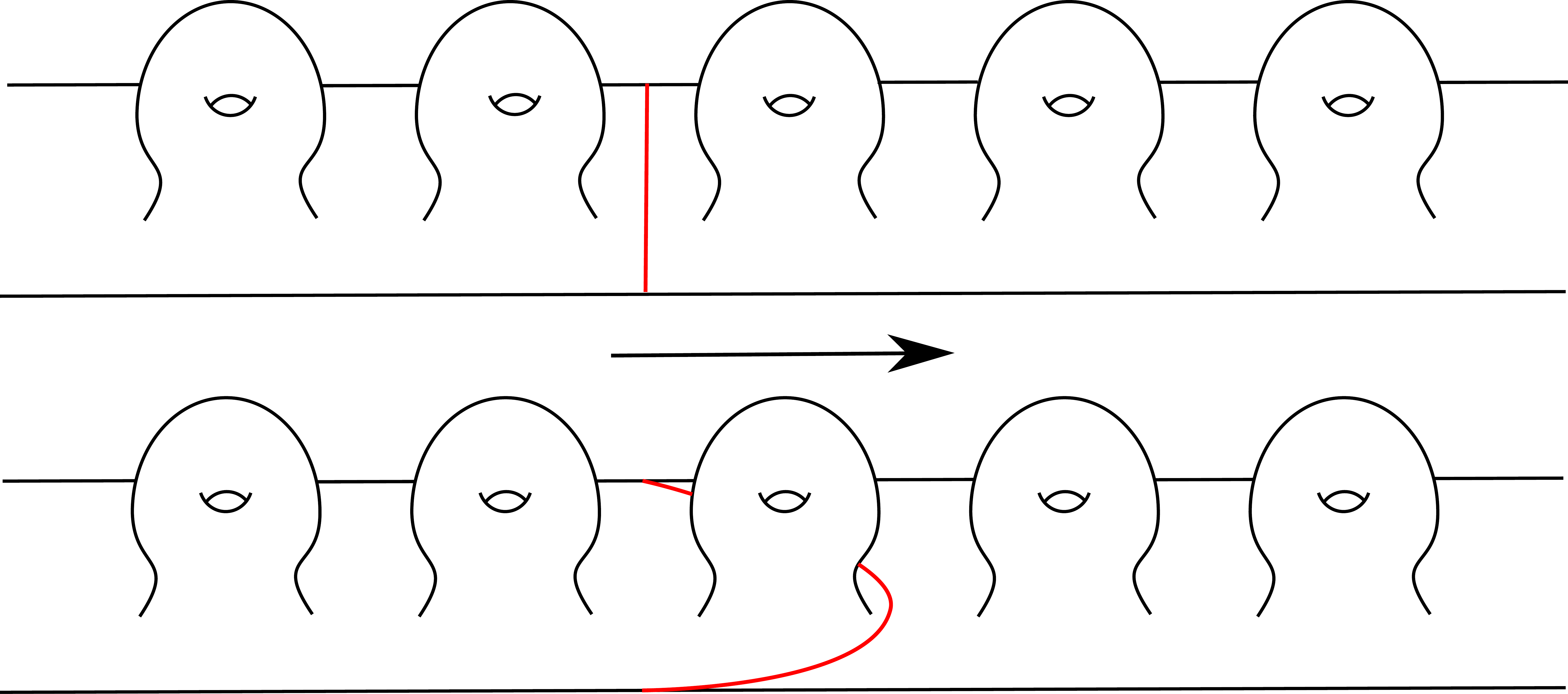}
    \caption{A handle strip comprised of a bi-infinite strip of handles $H$ and a shift map $h$. The left end of the strip is repelling, while the right end is attracting.}
    \label{fig:handleshift}
\end{figure}

\begin{example} \label{Ex:covering example} Consider a finite number of handle strips $(H_1, h_1),\ldots, (H_k, h_k)$ embedded in a ladder surface $L$ (that is an infinite-type surface with exactly two ends, both accumulated by genus) so that the complement of $\bigcup_{i=1}^k H_i$ in $L$ consists of $k$ bi-infinite strips and so that the attracting ends of $H_i$ are sent into a single end of $L$ and the repelling ends are sent into the other end of $L$. See \Cref{fig:product-handleshift} for the case of 2 handle strips. Let $\rho_1, \ldots,  \rho_k$ be the corresponding handle shifts. We obtain an end-periodic homeomorphism $\rho$ by an isotopy of the product of $\rho_1, \ldots, \rho_k$. This is shown for 2 handle strips in \Cref{fig:product-handleshift}. See \Cref{construction} below for more examples constructed using handle shifts. 
\end{example} 

\begin{figure}[htb!]
    \centering
    \includegraphics[width = 6in]{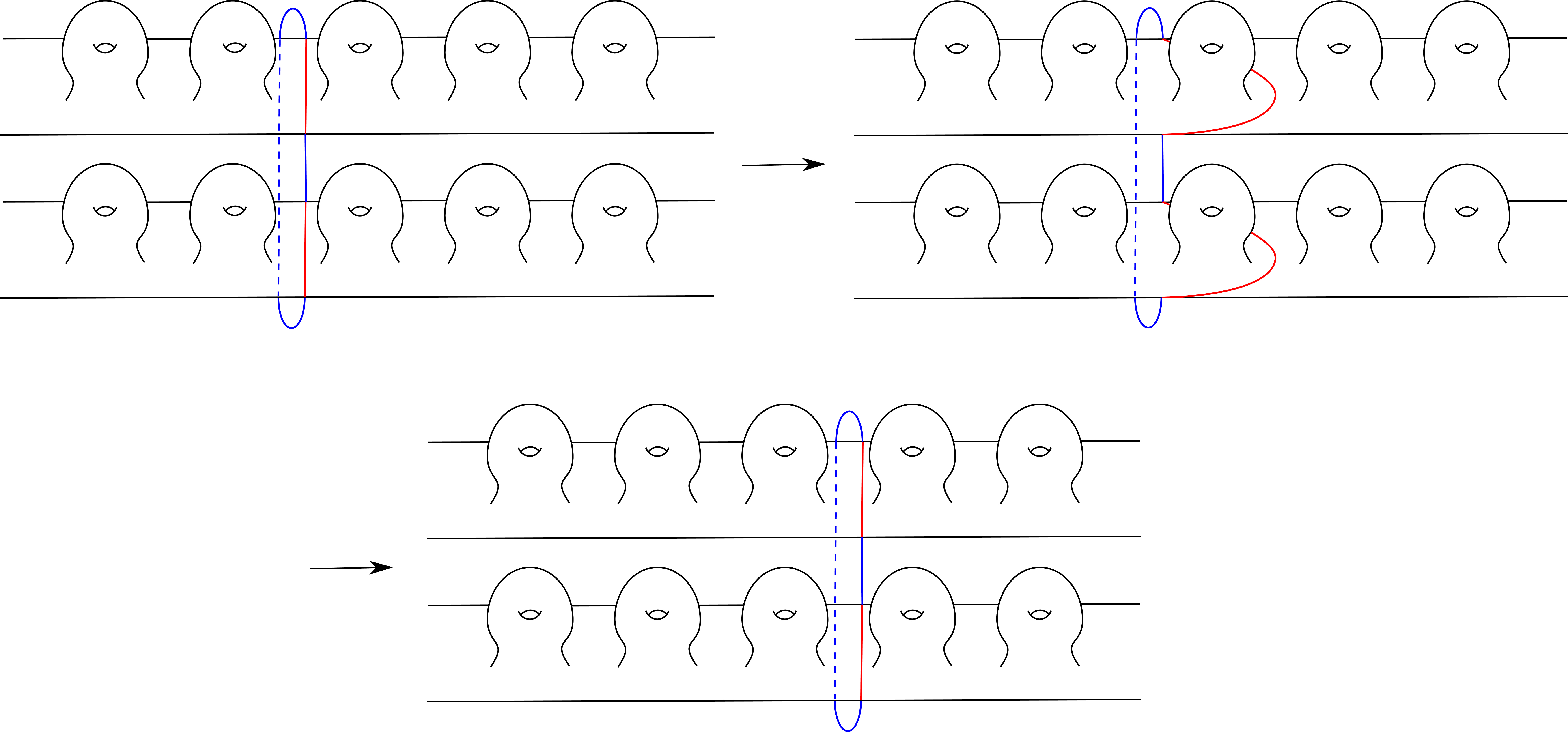}
    \caption{Products of pairwise commuting handle shifts as end-periodic homeomorphisms.}
    \label{fig:product-handleshift}
\end{figure}

In the example of the ladder surface just described, the end-periodic homeomorphism generates a covering group acting with closed surface quotient (of genus one more than the number of handle shifts involved). We now explain the generalization of this behavior to arbitrary end-periodic homeomorphisms of infinite-type surfaces.  

Given an end-periodic homeomorphism $f \colon S \to S$, choose nesting neighborhoods for each end and let $U_+$ be the union of nesting neighborhoods of the attracting ends, and likewise $U_-$ the union of nesting neighborhoods of repelling ends. Define 
\[ \mathcal U_+ = \bigcup_{n \geq 0} f^{- n}(U_+) \quad \mbox{ and } \quad \mathcal U_- = \bigcup_{n \geq 0} f^{n}(U_-) \]
We call $\mathcal U_+$ the {\em positive escaping set for $f$} and $\mathcal U_-$ the {\em negative escaping set for $f$}. From \Cref{D:end-periodic}, any choice of nesting neighborhoods will give rise to the same sets $\mathcal U_\pm$. Our standing assumption that $S$ has only finitely many ends, each accumulated by genus, together with the definition of nesting neighborhoods leads to the following properties of the escaping sets.  

\begin{lemma} \label{L:escaping sets are covering spaces}
Given an end-periodic homeomorphism $f \colon S \to S$, the escaping sets $\mathcal U_+$ and $\mathcal U_-$ are each $f$--invariant, and the infinite cyclic group $\langle f\rangle$ acts freely, properly discontinuously, and cocompactly on each. Consequently, the quotients $\mathcal U_\pm \to S_\pm = \mathcal U_\pm/\langle f \rangle$ are closed, orientable surfaces.
\end{lemma}

\begin{proof} The $f$--invariance of each of $\mathcal U_+$ and $\mathcal U_-$ is immediate from their definition.

Choose the nesting neighborhoods for the attracting ends so that their union $U_+$ has the property that $f(U_+) \subsetneq U_+$. Then $K_+ = \overline{U_+-f(U_+)}$ is a compact subset of $\mathcal U_+$ and
\[ \mathcal U_+ = \bigcup_{n \in \mathbb Z} f^n(K_+).\]
From the fact that positive powers applied to nesting neighborhoods give a neighborhood basis of the ends in \Cref{D:end-periodic}, we see that the $f^n$--translates of $K_+$ form a locally finite cover of $\mathcal U_+$. It follows that $\langle f \rangle$ acts properly discontinuously and cocompactly on $\mathcal U_+$. Since $\langle f \rangle$ is torsion free, the action is free. The same proof is valid for the repelling ends.
\end{proof}

Now note that $\mathcal U_\pm \to S_\pm$ is an infinite cyclic covering, but there may be multiple components of $\mathcal U_\pm$ that project to a single component of $S_\pm$. Fixing any component $\mathcal U$ of $\mathcal U_\pm$ defining an end $E$ of $S$ and projecting to a component $S_\mathcal{U} \subset S_\pm$, then $\mathcal U \to S_\mathcal{U}$ is a connected, infinite cyclic cover of the connected, closed, orientable surface $S_\mathcal{U}$. As such, there is a nonseparating simple closed curve $\alpha \subset S_\mathcal{U}$ that lifts to $\mathcal{U}$ so that any lift $\alpha_0 \subset \mathcal{U}$ separates $\mathcal{U}$ into neighborhoods of the two ends of $\mathcal{U}$. One of these is a nesting neighborhood $U_E$ of an end $E$ of $S$. Moreover, if $p_E$ is the period of $E$, then $\overline{f^{p_E}(U_E)} \subset U_E$: indeed, $\langle f^{p_E} \rangle$ generates the covering group of $\mathcal{U} \to S_\mathcal{U}$, and $\alpha_0 \cup f^{p_E}(\alpha_0)$ bounds a subsurface that serves as a fundamental domain. Note that 
\begin{equation} \label{E:good nesting neighborhoods} 
U_{E_i} = f^i(U_E) \quad \mbox{ for each } \quad i=0,\ldots,p_E-1.
\end{equation}
are nesting neighborhoods for all the attracting ends $E=E_0,\ldots,E_{p_E-1}$ in the $f$--orbit of $E$ with the same properties that $U_E$ has for $E$.

When the nesting neighborhood of an attracting end $E$ is chosen in this way, and the neighborhoods of the other ends in the $f$--orbit of $E$ are given by \eqref{E:good nesting neighborhoods}, then we will call these {\em good nesting neighborhoods} for the attracting ends. We similarly define good nesting neighborhoods for the repelling ends.

Using the unions, $U_+$ and $U_-$, of good nesting neighborhoods for the attracting and repelling ends, respectively, we can choose a hyperbolic metric on $S$ so that the boundaries of these neighborhoods are geodesics, and furthermore $f^{\pm 1}|_{U_+}$ and $f^{\pm 1}|_{U_-}$ are isometries onto their images; see \cite{Fenley-depth-one}.

\begin{corollary} \label{C:end quotient surfaces}
The number of components of $S_+$ and $S_-$ is equal to the number of $f$--orbits of components of $\mathcal U_+$ and $\mathcal U_-$, respectively. Moreover, $\xi(S_+) = \xi(S_-)$. \end{corollary}
\begin{proof} The first statement is clear since the preimage of any component of $S_\pm$ is an $f$--orbit of components of $\mathcal U_\pm$.

For the second claim, choose the unions $U_+$ and $U_-$ of good nesting neighborhoods and a hyperbolic metric as above. Note that by our assumption on $U_+$ and $U_-$, $\overline{U_-} \cap \overline{U_+} = \emptyset$. 
In particular,
\[ \Sigma = \overline{S  -  (U_+ \cup U_-)}\]
is a subsurface with geodesic boundary, as is $f(\Sigma)$. We will let $K_{\pm}  = \overline{U_{\pm} - f^{\pm}(U_{\pm})}$. 

Now observe that $S_+$ admits a hyperbolic metric so that the restriction of the covering projection $\mathcal U_+ \to S_+$ to $K_+ \subset \mathcal U_+$ is a surjective local isometry, injective on the interior. Combining this with the Gauss-Bonnet Theorem we have 
\[ \tfrac{4\pi}3\xi(S_+) = \Area(S_+) = \Area(K_+ ). \]
Likewise, there is a hyperbolic metric on $S_-$ so that 
\[ \tfrac{4\pi}3\xi(S_-) = \Area(S_-) = \Area(K_- ). \]
On the other hand,
\begin{equation} \label{E:area 1} \Area(\Sigma \cup K_-) = \Area(\Sigma) + \Area(K_-),  \end{equation}
since $\Sigma$ and $K_-$ intersect along a finite union of simple closed geodesics, which has area $0$. Similarly,
\begin{equation} \label{E:area 2} \Area(\Sigma \cup K_+) = \Area(\Sigma) + \Area(K_+).  \end{equation}
Now observe that $f$ maps $\Sigma \cup K_-$ to $f(\Sigma) \cup f(K_-) = \Sigma \cup K_+$, and so by Gauss-Bonnet the areas on the left hand sides of \eqref{E:area 1} and \eqref{E:area 2} are equal. Thus, $\Area(K_-) = \Area(K_+)$, and so $\xi(S_+) = \xi(S_-)$, as required.
\end{proof}


In the following, recall that the terms ``curve'' and ``line'' actually mean a (proper) homotopy class of such.
\begin{definition}[Irreducibility/Strong irreducibility]\label{Irreducibility/Strong irreducibility} An end-periodic homeomorphism, $f: S\to S$, is \emph{irreducible} if it has 
\begin{itemize}
    \item[1)] no {\em periodic curves};
    \item[2)] no {\em AR-periodic lines}, i.e.~a line $\ell$ with one end in an attracting end of $S$ and the other in a repelling end of $S$ so that there exists $k \in \mathbb Z$ with $f^k(\ell) = \ell$;~and
    \item[3)] no {\em reducing curves}, i.e.~a curve $\gamma$ so that there exists $m, n \in \mathbb Z$ with $m < n$ and such that $f^n(\gamma)$ is contained in a nesting neighborhood of an attracting end and $f^m(\gamma)$ is contained in a nesting neighborhood of a repelling end. 
\end{itemize}  
An end-periodic homeomorphism $f$ is \emph{strongly irreducible} if it is irreducible and contains no \emph{periodic lines}, i.e. a line $\ell$ in $S$ such that $f^k(\ell) = \ell$ for some $k\in \mathbb{Z}$ (with no constraints on the ends of the line).
\end{definition}

\begin{remark}
Note that the language used above to define periodic and AR-periodic lines differs from that used in the literature. See, for example, \cite{Fenley-depth-one, FenleyCT, CC-book}. 
\end{remark}

In their unpublished work, Handel and Miller developed an analogue of the Nielsen-Thurston classification/reduction theory for end-periodic homeomorphisms.  In that theory, irreducible end-periodic homeomorphisms play the role of pseudo-Anosov homeomorphisms due to the analogous properties their invariant laminations enjoy.
See Cantwell--Conlon--Fenley \cite{CC-book} for a detailed exposition and expansion of this theory.
The concept of strong irreducibility is new, and arose in the current paper from geometric considerations of the mapping torus; see \Cref{proposition:interior}. 

In his thesis \cite{FenleyThesis1989}, Fenley produces the irreducible end-periodic homeomorphism shown in \Cref{fig:fenley-irreducible-example}. His proof of irreducibility relies on building and analyzing the corresponding weighted train tracks. We will provide a general construction of both irreducible and strongly irreducible end-periodic homeomorphisms in \Cref{S:example}. However, our proofs of irreducibility rely on subsurface projection and curve complex geometry rather than weighted train tracks.

\begin{figure}[htb!]
\centering
\def\svgwidth{4in}
\begingroup%
  \makeatletter%
  \providecommand\color[2][]{%
    \errmessage{(Inkscape) Color is used for the text in Inkscape, but the package 'color.sty' is not loaded}%
    \renewcommand\color[2][]{}%
  }%
  \providecommand\transparent[1]{%
    \errmessage{(Inkscape) Transparency is used (non-zero) for the text in Inkscape, but the package 'transparent.sty' is not loaded}%
    \renewcommand\transparent[1]{}%
  }%
  \providecommand\rotatebox[2]{#2}%
  \newcommand*\fsize{\dimexpr\f@size pt\relax}%
  \newcommand*\lineheight[1]{\fontsize{\fsize}{#1\fsize}\selectfont}%
  \ifx\svgwidth\undefined%
    \setlength{\unitlength}{1835.87534457bp}%
    \ifx\svgscale\undefined%
      \relax%
    \else%
      \setlength{\unitlength}{\unitlength * \real{\svgscale}}%
    \fi%
  \else%
    \setlength{\unitlength}{\svgwidth}%
  \fi%
  \global\let\svgwidth\undefined%
  \global\let\svgscale\undefined%
  \makeatother%
  \begin{picture}(1,0.36844717)%
    \lineheight{1}%
    \setlength\tabcolsep{0pt}%
    \put(0,0){\includegraphics[width=\unitlength,page=1]{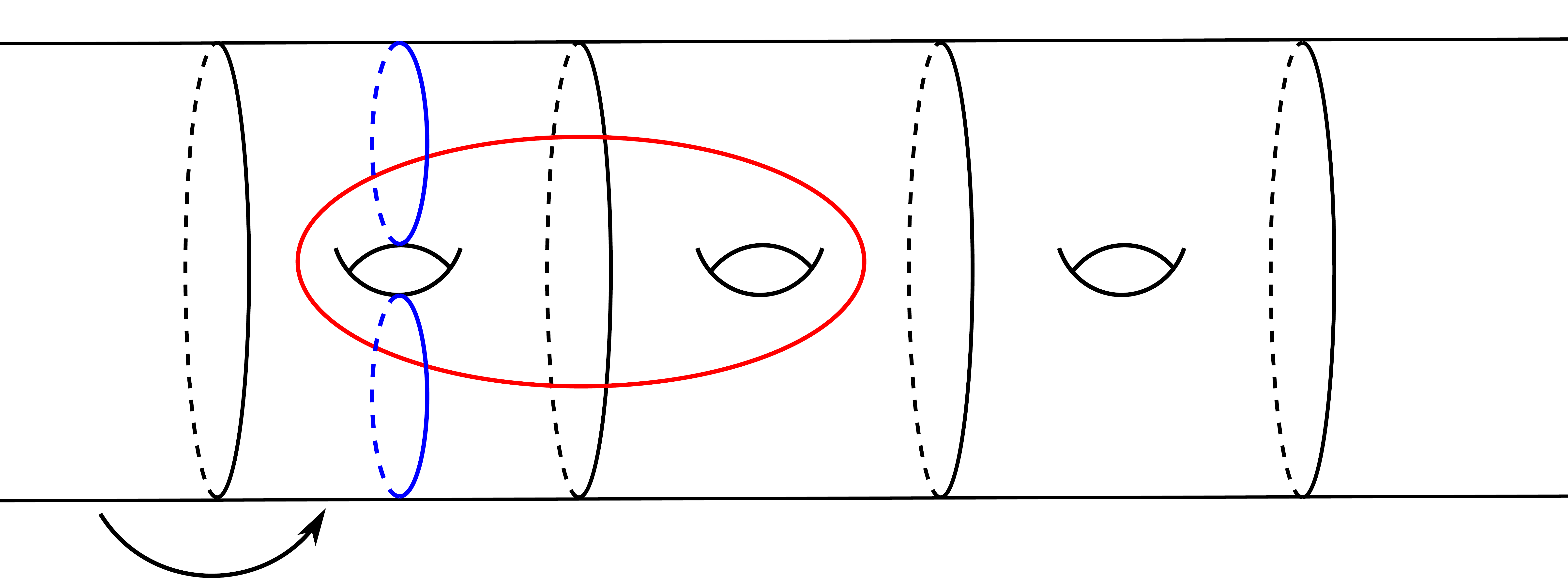}}%
    \put(0.25191462,0.02151016){\color[rgb]{0,0,0}\makebox(0,0)[t]{\lineheight{1.25}\smash{\begin{tabular}[t]{c}$\alpha$\end{tabular}}}}%
    \put(0.25191462,0.35587228){\color[rgb]{0,0,0}\makebox(0,0)[t]{\lineheight{1.25}\smash{\begin{tabular}[t]{c}$\beta$\end{tabular}}}}%
    \put(0.48772235,0.27288072){\color[rgb]{0,0,0}\makebox(0,0)[lt]{\lineheight{1.25}\smash{\begin{tabular}[t]{l}$\gamma$\end{tabular}}}}%
    \put(0,0){\includegraphics[width=\unitlength,page=2]{fenley-irreducible-example.pdf}}%
  \end{picture}%
\endgroup%

\caption{The homeomorphism $f=T_{\alpha}T_{\beta}T_{\gamma}\rho$, where $\rho$ is the end-periodic homeomorphism induced by a single handle strip, is an irreducible end-periodic homeomorphism by work of Fenley.}
\label{fig:fenley-irreducible-example}
\end{figure}

\subsection{End behavior}
\label{sub:endbehavior}

Given a homeomorphism $f \colon S \to S$ (or its associated mapping class $f \in \Map(S)$), its ``end behavior" describes the way in which $f$ acts on neighborhoods of the ends. It is convenient to make this precise in two different ways as follows.

First, we will say that two elements $f,f' \in \Map(S)$ have the same {\em fine end behavior} if $f^{-1}f' \in \Map_c(S) = \PMap_c(S)$: that is, if $f^{-1}f'$ is compactly supported. Thus, for any end of $S$, $f$ and $f'$, agree on some neighborhood of that end. In \Cref{S:example}, we will show that for any pure, end-periodic homeomorphism $f \colon S \to S$, there are both irreducible and strongly irreducible end-periodic homeomorphisms with the same fine end behavior as $f$; see \Cref{T:blah}. 

To construct examples in this level of generality, it is convenient to have a notion of ``coarse end behavior". This coarse end behavior will also provide a convenient description of the surfaces $S_\pm$ associated to an end-periodic homeomorphism as described above. To formalize this, we will first introduce a few more definitions and state a result of Aramayona--Patel--Vlamis \cite{aramayona2020first}. 

Let $H_1^{sep}(S,\mathbb Z)$ be the subgroup of $H_1(S,\mathbb Z)$ spanned by homology classes of separating curves, and let $H^1_{sep}(S, \mathbb Z) \leq H^1(S, \mathbb Z)$ be the subgroup naturally identified with $\text{Hom}(H_1^{sep}(S, \mathbb Z), \mathbb Z)$ by the isomorphism $H^1(S, \mathbb Z) \to \text{Hom}(H_1(S, \mathbb Z), \mathbb Z)$ arising from the Universal Coefficient Theorem.

Given an end $E$ of $S$, let $\gamma$ be an oriented curve that separates $E$ from the other ends (with orientation such that the component of $S - \gamma$ to the right of $\gamma$ is a neighborhood of $E$). Recall that we require that $S$ is boundaryless and has at least two ends accumulated by genus. Then, $\gamma$ defines a nonzero element $[\gamma] \in H_1^{sep}(S,\mathbb Z)$, canonically associated to $E$. When convenient, we denote this class $v_{E} = [\gamma] \in H_1^{sep}(S,\mathbb Z)$. Given $f \in \PMap(S)$, $[\gamma] = [f(\gamma)] = f_*([\gamma])$ and Aramayona-Patel-Vlamis define an integer $\varphi_{[\gamma]}(f)$ which can be viewed as the ``signed genus" between $\gamma$ and $f(\gamma)$. If, for example, $\gamma$ and $f(\gamma)$ are disjoint, then $|\varphi_{[\gamma]}(f)|$ is the genus of the subsurface bounded by $\gamma$ and $f(\gamma)$ (with a negative sign if $f(\gamma)$ is to the left of $\gamma$); see \cite[Section~3]{aramayona2020first} for a more precise description. Furthermore, they show that $\varphi_{[\gamma]} \colon \PMap(S) \to \mathbb Z$ is a well-defined homomorphism that depends only on the homology class $[\gamma]$; see \cite[Proposition~3.3]{aramayona2020first}. 

According to \cite[Proposition~4.4]{aramayona2020first}, this induces a surjective homomorphism
\[ \Phi^* \colon \PMap(S) \to H^1_{sep}(S,\mathbb Z),\]
by $\Phi^*(f)([\gamma]) = \varphi_{[\gamma]}(f)$. In other words, $\Phi^*(f)$ describes ``how much genus has been shifted" on each end by $f$. It follows from \cite[Theorem~5]{aramayona2020first} and its proof that $\Phi^*$ is a surjective homomorphism whose kernel is precisely $\overline{\Map_c(S)}$, the closure of the compactly supported mapping class subgroup. We define the {\em coarse end behavior} of $f \in \PMap(S)$ to be the cohomology class $\Phi^*(f) \in H^1_{sep}(S,\mathbb Z)$.

\begin{remark} We note that $\Phi^*(f) = \Phi^*(f')$ if and only if $f^{-1}f' \in \overline{\Map_c(S)}$ (thanks to Ghaswala for pointing this out).  In particular, if $f$ and $f'$ have the same fine end behavior, then they have the same coarse end behavior (since $\Map_c(S) < \overline{\Map_c(S)}$). Aramayona, Patel, and Vlamis construct a section of $\Phi^*$ onto an abelian group generated by commuting handle shifts, thus giving $\PMap(S)$ the structure of a semidirect product; see \cite[Theorem~3, Corollary~6]{aramayona2020first}. Recall that we require that $S$ is boundaryless and has at least two ends accumulated by genus. Although we will use handle shifts in \Cref{construction}, we will be using a different collection of handle shifts. Since we will not appeal to this structure, we do not discuss it further.
\end{remark}

\begin{lemma} \label{L:genus of component}
Suppose $f \colon S \to S$ is a pure, end-periodic homeomorphism, and let $S_\pm$ be the surfaces obtained as quotients of the positive and negative escaping sets. If $E$ is an end of $S$ with associated homology class $v_{E} \in H_1^{sep}(S,\mathbb Z)$,
then the genus of the corresponding component of $S_\pm$ is $1 + |\Phi^*(f)(v_{E})|$. 
\end{lemma}
\begin{proof}
Since $f$ is pure, we may choose a simple closed curve $\gamma$ such that $[\gamma]= v_{E}$, so that $\gamma$ cuts off a neighborhood $U$ of, say, an attracting end $E$ and $f(\overline U) \subset U$. 
In particular, $\gamma$ and $f(\gamma)$ are disjoint and bound a subsurface of genus $|\Phi^*(f)(v_{E})|$. This subsurface serves as a fundamental domain for the action of $\langle f \rangle$ on the corresponding component of $\mathcal U_+$, and hence the genus of the component of $S_+$ is one more than $|\Phi^*(f)(v_{E})|$, as required. The case of a repelling end is similar. \end{proof}

Given any $f \in \PMap(S)$, let $E_1,\ldots, E_n$ denote the ends of $S$, $v_i = v_{E_i}$, for $i=1,\ldots,n$, and set
\begin{equation} \label{E:end complexity}
|\Phi^*(f)| = \sum_{i=1}^n |\Phi^*(f)(v_i)|.
\end{equation}
With this definition, we have the following.
\begin{corollary} \label{C:complexity of S_pm}
If $f \colon S \to S$ is an end-periodic homeomorphism with associated quotient surfaces $S_\pm$ from the escaping sets, then
\[ \xi(S_+) = \xi(S_-) = \tfrac32 |\Phi^*(f)|. \]
\end{corollary}
\begin{proof}
By the \Cref{L:genus of component}, the genus of the component of $S_\pm$ corresponding to $E_i$ is $|\Phi^*(f)(v_i)|+1$. Therefore,
\[ \xi(S_+) + \xi(S_-) = \sum_{i=1}^n \left( 3(1+|\Phi^*(f)(v_i)|)-3\right) = 3 \sum_{i=1}^n |\Phi^*(f)(v_i)|. \]
By \Cref{C:end quotient surfaces}, $\xi(S_+) = \xi(S_-) = \tfrac12(\xi(S_+) + \xi(S_-))$, completing the proof.
\end{proof} 

To construct end-periodic homeomorphisms with a specified fine end behavior, the next corollary shows that it suffices to construct an end-periodic homeomorphism with the same coarse end behavior.
\begin{corollary} \label{C:coarse conjugates to fine}
Let $f_1, f_2: S\to S$ be pure, end-periodic homeomorphisms with the same coarse end behavior. Then, there exists a mapping class $f_1'$ with the same fine end behavior as $f_2$ such that $f_1$ is conjugate to $f_1'$ by an element of $\PMap(S)$.
\end{corollary} 

\begin{proof}
For each $i=1,2$, choose good nesting neighborhoods $U_\pm^i$ for $f_i$ so that the neighborhoods $U_E^i$ of the end $E$ is bounded by a single curve $v_E^i$ that projects to a nonseparating curve in the quotient surface $S_{\pm}^i$. We further assume that the neighborhood are chosen so that there is an orientation preserving homeomorphism
\[ h_0 \colon \overline{S - (U_+^1 \cup U_-^1)} \to \overline{S - (U_+^2 \cup U_-^2)}.\] 
The existence of such a homeomorphism follows from the classification of surfaces, since we can increase the genus of either $\overline{S-(U^i_+ \cup U^i_-)}$, $i=1,2$ by choosing smaller good nesting neighborhoods.  Moreover, we may assume $h_0(v_E^1) = v_E^2$ for each end $E$. Since $f_1$ and $f_2$ have the same coarse end behavior, by \Cref{L:genus of component}, for each end $E$, the image $S_E^1 \subset S_\pm^1$ of $U_E^1$ is a closed surface of the same genus as the image $S_E^2 \subset S_\pm^2$ of $U_E^2$. By the classification of surfaces again there is an orientation preserving homeomorphism $S_E^1 \to S_E^2$ that sends the image of $v_E^1$ to the image of $v_E^2$. This homeomorphism lifts to a homeomorphism $h_E \colon U_E^1 \to U_E^2$ which is equivariant with respect to the restriction of the semigroup generated by $f_1$ and $f_2$ if $E$ is an attracting end, and equivariant with respect to $f_1^{-1}$ and $f_2^{-1}$ if $E$ is a repelling end. That is, on $U_E$ we have we have that $f_2 = h_E f_1 h_E^{-1}$ or $f_2^{-1} = h_E f_1^{-1}h_E^{-1}$ if $E$ is attracting or repelling, respectively. Observe that in the repelling case, on $f_2^{-1}(U_E)$ we have $f_2 = h_E f_1^{-1} h_E^{-1}$. 
Because all of these homeomorphisms are orientation preserving, adjusting $h_0$ if necessary, we may assume that $h_E$ agrees with $h_0$ on $v_E^1$. We may therefore glue these together into a single pure homeomorphism $h \colon S \to S$. Setting $f_1' = h f_1^{-1} h^{-1}$ we see that $f_1'$ and $f_2$ agree on
\[ U_+^2 \cup f_2^{-1}(U_-^2),\]
the complement of which has compact closure. Thus, $f_1'$ and $f_2$ have the same fine end behavior. 
\end{proof}

Note that for $v_1,\ldots,v_n$ as above (e.g.~in \eqref{E:end complexity}), we have $\sum_i v_i = 0$ in $H_1^{sep}(S,\mathbb Z)$, and any $n-1$ elements of this set forms a basis. Given any $\omega \in H^1_{sep}(S,\mathbb Z)$, we thus have
\[\sum_{i=1}^n \omega(v_i) = 0, \]
 and moreover, for any integers $w_1,\ldots,w_n$ with $\sum w_i=0$, there exists a unique $\omega \in H^1_{sep}(S,\mathbb Z)$ so that $\omega(v_i) = w_i$ for all $i = 1, \ldots, n$. Thus, specifying the coarse end behavior is equivalent to specifying such a collection of integers $w_1,\ldots,w_n$ with sum zero.

For an end-periodic homeomorphism on a surface with $n$ ends, the integers $w_1,\ldots,w_n$ describing the coarse end behavior (as above, summing to zero) are all nonzero.
Conversely, for any prescribed such coarse end behavior, in the following construction, we explain how to build examples of end-periodic homeomorphisms with that coarse end behavior. Moreover, in \Cref{S:example}, we will show that many of the end-periodic homeomorphisms arising from \Cref{construction} are irreducible and strongly irreducible.

\begin{construction}\label{construction} As above, let $v_1,\ldots,v_n$ be separating curves bounding pairwise disjoint  neighborhoods $U_1,\ldots,U_n$ of the ends $E_1,\ldots,E_n$ of $S$, respectively, and fix any nonzero integers $w_1, \ldots, w_n$ with $\sum_{i=1}^n w_i = 0$. Next, properly embed pairwise-disjoint handle strips $\{(H_j,h_j)\}_1^m$ into $S$ so that for each $w_i > 0$, there are exactly $w_i$ ends of the handle strips sent to $E_i$, all of which are attracting, while if $w_i <0$, then there are $-w_i$ ends of the handle strips sent to $E_i$, all of which are repelling.
We assume the embeddings are such that each $v_i$ meets each of the handle strips exiting the end $E_i$ in a single separating arc in the handle strip like the red arc in the top of \Cref{fig:handleshift}, and that in $U_i$ the complement of the handle strips is a union of properly embedded half strips, $[0,1] \times [0,\infty) \subset \overline U_i$; see \Cref{fig:handle-strip-intersection-section-2}.
As in \Cref{Ex:covering example}, the associated handle shifts $\rho_1,\ldots,\rho_m$ pairwise-commute, and the map $\rho = \Pi_1^m \rho_i$ is isotopic to an end-periodic homeomorphism with good nesting neighborhoods $U_1,\ldots,U_n$.
By construction, $\Phi^*(\rho)(v_i) = w_i$.  In \Cref{sharpness}, we will construct examples of irreducible and strongly irreducible end periodic homeomorphisms with the same fine end behavior as $\rho$ by composing $\rho$ with a compactly supported homeomorphism.
\end{construction}

\begin{figure}[htb!]
    \centering
    \def\svgwidth{4in}
    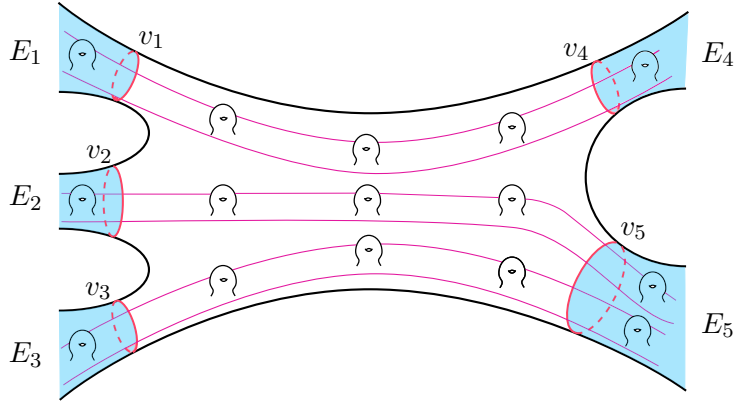
    \caption{The curves $v_1,\ldots,v_5$ cut off neighborhoods $U_1,\ldots,U_5$ (in blue) of the ends $E_1, \ldots, E_5$, respectively.}
    \label{fig:handle-strip-intersection-section-2}
\end{figure}

\begin{remark} 
The notion of coarse end behavior can be extended to an arbitrary element $f \in \Map(S)$ by recording the permutation of the ends, together with the cohomology class $\Phi^*(f^m)$, where $m\geq 1$ is the order of the permutation. The purpose of the discussion of end behaviors is to provide a framework for illustrating the generality of the examples described above and elaborated upon in \Cref{S:example}. Expanding to non-pure end-periodic homeomorphisms only serves to complicate the notation and obfuscate the ideas. For this reason, we do not elaborate on, or further pursue end behaviors for non-pure elements.
\end{remark}

\subsection{Graphs associated to surfaces} \label{S:associated-graphs}

Graphs built from topological data associated to curves and arcs on a surface have become an essential tool in studying surfaces and their mapping class groups.  Most notable among these is the curve graph, $\C(Y)$, of a surface $Y$, whose vertices are curves on $Y$ with edges connecting any two vertices that have disjoint representatives on $Y$.

When $Y$ has finite type,
a germinal result of Masur--Minsky \cite{MasurMinsky.1999} with numerous implications is that $\C(Y)$ is infinite diameter and hyperbolic.
Unfortunately, it is not hard to see that $\C(Y)$ has diameter 2 when $Y$ has infinite type. The actions of big mapping class groups on various other graphs of curves, simplicial complexes, and metric spaces have been studied extensively. See, for just a few examples, Aramayona--Valdez \cite{AramayonaValdez2016}, Bavard \cite{Bavard}, Durham--Fanoni--Vlamis \cite{DFV}, Lanier--Loving \cite{LanierLoving2021}, and Mann--Rafi \cite{MannRafi2019}.

Another useful graph is the \textit{arc and curve graph}, $\mathcal{AC}(Y)$: the graph whose vertices are curves and arcs on $Y$, and whose edges correspond to pairs of curves and/or arcs having disjoint representatives. When $Y$ is an annulus, the vertices of $\mathcal{AC}(Y)$ are also arcs (which recall means homotopy classes of essential arcs), though in this setting we require the homotopies to {\em fix} the endpoints (otherwise there would be only one vertex).  In this case, edges connect vertices with representatives having disjoint interiors.  

Given $\Sigma \subset Y$ a non-annular, essential subsurface, the \emph{subsurface projection} \[ \pi_{\Sigma} : \mathcal{C}(Y) \longrightarrow \mathcal{P}(\mathcal{AC}(\overline \Sigma))\]
takes a vertex $\alpha \in \mathcal{C}(Y)$ to the collection of curves and arcs in $\Sigma$ obtained by taking the union of all distinct homotopy classes occurring in the intersection of $\alpha$ with $\overline \Sigma$, after $\overline \Sigma$ and $\alpha$ are put in minimal position (we send points in an edge to the union of the images of the endpoints of the edge).  Note that our definition differs slightly from the projections defined by Masur and Minsky \cite{MasurMinsky.2000}. 

When $\Sigma$ is an annulus, we first consider the cover $Z_{\Sigma}$ of $Y$ corresponding to $\pi_1\Sigma$. This cover is naturally the interior of a compact annulus, $\overline Z_\Sigma$, (obtained by adding the ideal boundary) and we let $\pi_{\Sigma}(\alpha)$ be the union of the closures of those components of the preimage of $\alpha$ in $Z_\Sigma$ that limit to points on both boundary components of $\overline Z_\Sigma$. 
 
We define the $\Sigma$-\textit{subsurface distance} as \[d_{\Sigma}(\alpha,\beta) = \diam_{\mathcal{AC}(\overline \Sigma)}(\pi_\Sigma(\alpha) \cup \pi_\Sigma(\beta)),\]
whenever $\pi_\Sigma(\alpha),\pi_\Sigma(\beta) \neq \emptyset$. Observe that in this case we have
\begin{equation} \label{E:at least 2}
    d_\Sigma(\alpha,\beta) \geq 2 \Leftrightarrow i(\alpha,\beta) \neq 0.
\end{equation} 
If $\pi_\Sigma(\alpha),\pi_\Sigma(\beta),\pi_\Sigma(\gamma) \neq \emptyset$, then this distance satisfies the triangle inequality
\[d_\Sigma(\alpha,\beta) \leq d_\Sigma(\alpha,\gamma)+ d_\Sigma(\gamma,\beta).\]

\begin{figure}
    \centering
    \includegraphics[width = 8 cm]{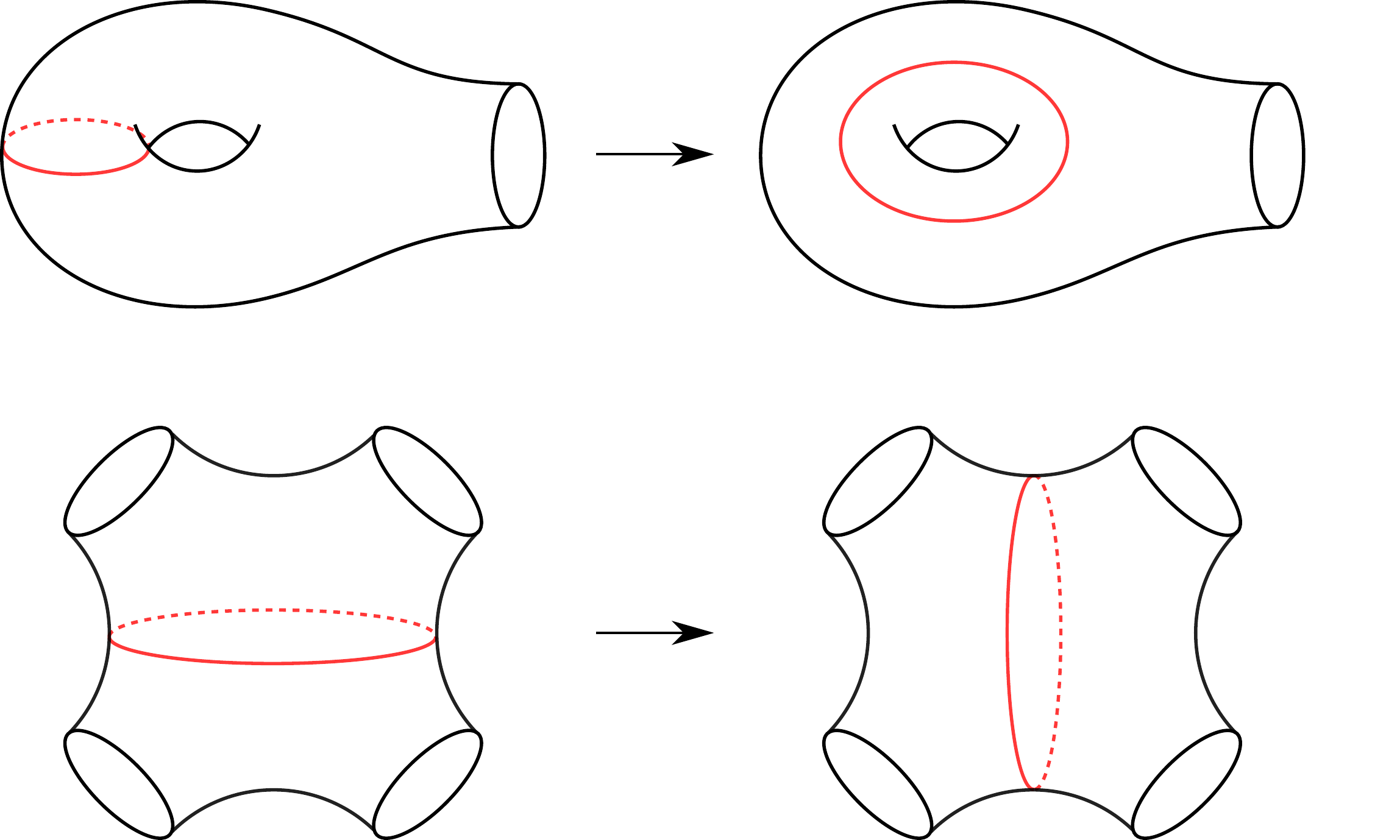}
    \caption{Elementary moves on pants decompositions.}
    \label{fig:pants-moves}
\end{figure}

The graph of central interest in this paper is the pants graph of an infinite-type surface $S$, which we now describe. First, a {\em pants decomposition} on $S$ is a multicurve in $S$ that cuts $S$ into pairs of pants. An {\em elementary move} on a pants decomposition $P$ replaces a single curve in $P$ with a different one intersecting it a minimal number of times, producing a new pants decomposition $P'$. If $P'$ is obtained from $P$ by an elementary move and $Q = P \cap P'$ is the maximal multicurve in common, then $S  -  Q$ has one component $\Sigma$ which has complexity one. There are two types of elementary moves, depending on the homeomorphism type of $\Sigma$, and these are illustrated in \Cref{fig:pants-moves}.

\begin{definition} \label{def:pants-graph}
The {\em pants graph}, $\mathcal P(S)$, is the graph whose vertices are (isotopy classes of) pants decompositions on $S$, with edges between pants decompositions that differ by an elementary move.  We define a path metric on (the components of) $\mathcal P(S)$ so that and edge corresponding to an elementary move that occurs on a one-holed torus has length $1$, and an edge corresponding to an elementary move that occurs on a four-holed sphere has length $2$. The group $\Map(S)$ still acts isometrically on $\mathcal P(S)$ with this metric.
\end{definition} 

Brock proved in \cite{Brock-mappingtorus-vol} that for finite-type surfaces $Y$, $\mathcal P(Y)$ is a coarse model for the Teichm\"uller space of $Y$ with the Weil--Petersson metric. For infinite-type surfaces, the pants graph is not connected, and thus the distance between pants decompositions may be infinite. See \cite{Branman} for an alternative topology on this graph and other facts about it.

\begin{definition}
Given any end-periodic homeomorphism $f \colon S \to S$, we define the \textit{asymptotic translation distance} of $f$ on $\mathcal{P}(S)$ to be $$\tau(f) = \inf_{P\in\mathcal{P}(S)} \liminf_{n \to \infty} \frac{d(P, f^n(P))}n,$$ 
where this infimum is over all pants decompositions $P \in \mathcal P(S)$. Observe that $\tau(f^n) = n \tau(f)$ for all $n > 0$.
\end{definition}

For a pants decomposition $P$ of $S$, it may be the case that for all $n > 0$, $P$ and $f^n(P)$ are not in the same component, and thus $d(P,f^n(P)) = \infty$ for all $n >0$. If this is the case, then the same is necessarily true for all $P'$ in the same component as $P$. For the purposes of studying $\tau(f)$, such components are useless, and thus we let $\mathcal P_f(S)$ be the subgraph of $\mathcal P(S)$ spanned by all {\em $f$--asymptotic pants decompositions}; that is, those pants decompositions $P$ for which $P$ and $f^n(P)$ differ by a finite number of elementary moves for some $n > 0$ (and are thus a finite distance in $\mathcal P(S)$). So, in the definition of $\tau(f)$, it suffices to consider only pants decompositions $P$ taken from $\mathcal P_f(S)$. Note that, since $\mathcal P_f(S)$ is non-empty, $\tau(f)$ is finite. 

\begin{remark}
The behavior of end periodic homeomorphisms on the pants graph is reminicent of the work of Funar and Kapoudjian in \cite{FunKap}.  Indeed, their {\em asymptotic mapping class group of infinite genus} is a subgroup of the mapping class group of the {\em blooming Cantor tree} that preserves a certain component of its pants graph.
 Thanks to Ghaswala for pointing out this reference.
\end{remark}

Given a component $\Omega \subset \mathcal P_f(S)$, we define $\tau(f,\Omega)$ by the same formula as $\tau(f)$, except the infimum is taken over all $P \in \Omega$ (rather than over all $P \in \mathcal P(S)$). As such, $\tau(f,\Omega) \geq \tau(f)$.

\subsection{Decomposition spaces}
Our proof of \Cref{T:upper} will utilize some decomposition theory and a result of Armentrout \cite{Arm69}. 

Let $M$ be a 3-manifold. A \textit{decomposition} $G$ of $M$ is a partition of $M$ into disjoint sets. Given such a decomposition, there is a natural map $f: M\to M / G$ obtained by collapsing each element $X\in G$ to a point and endowing $M / G$ with the quotient topology. If $\pi: M\to N$ is any quotient map, then there is a naturally associated decomposition of $M$ which comes from the preimages of points in $N$ under $\pi$. As such, a quotient map may also be referred to as a decomposition. For more details about decompositions; see \cite{Dav86}.

\begin{defn}[Upper semicontinuous]
A decomposition $G$ of $M$ is said to be \textit{upper semicontinuous} if each element $X\in G$ is compact, and if for each $X\in G$ and each open subset $U\subseteq M$ with $X$ contained in $U$, there exists a further open set $V\subseteq M$ containing $X$ such that for each $X'\in G$ with $X'\cap V\neq \emptyset$, we have that $X'\subseteq U$. A quotient map $\pi: M\to N$ is said to be \textit{upper semicontinuous} if the associated decomposition space is upper semicontinuous.
\end{defn}

The following result gives alternative characterizations of an upper semicontinuous decomposition; see \cite[Proposition I.1.1]{Dav86}. 

\begin{proposition}\label{P:cellular characterization}
Let $G$ be a decomposition of a 3-manifold $M$. The following are equivalent:

\begin{enumerate}
    \item The decomposition $G$ is upper semicontinuous.
    \item For each open set $U$ in $M$, let $U^*$ denote the union of all sets in $G$ which are contained in $U$. Then, the set $U^*$ is open in $M$.
    \item The associated quotient map $f: M\to M / G$ is closed.
\end{enumerate}
\end{proposition}

\begin{defn}[Cellular decomposition]\label{definition:cellular}
A subset $X$ of a 3-manifold $M$ is said to be \textit{cellular} if there exists a sequence $\{B_i\}$ of 3-cells in $M$ such that for each $i\geq 1$, $B_{i+1}\subset \mathrm{Int}(B_i)$, and $X = \cap_{i=1}^\infty B_i$. Equivalently, a subset $X$ of a 3-manifold $M$ is cellular if and only if $X$ is compact and has arbitrarily small neighborhoods which are homeomorphic to Euclidean 3-space. A decomposition $G$ of a 3-manifold $M$ is a \textit{cellular decomposition} if $G$ is an upper semicontinuous decomposition of $M$ and if each element $X\in G$ is a cellular subset of $M$.
\end{defn}

The following result of Armentrout \cite[Theorem 2]{Arm69} gives a sufficient condition for when the quotient space is homeomorphic to the original space. We apply this result to the double of $\overline{M}_f$ in the proof of \Cref{P:block decomposition}.

\begin{theorem} \label{thm:armentrout} Suppose $M$ is a 3-manifold and $G$ is a cellular decomposition of $M$ such that $M / G$ is a 3-manifold. Then, $M \to M / G$ is homotopic to a homeomorphism.\end{theorem}

\subsection{Surfaces in 3-manifolds and foliations}
Suppose $\overline{M}$ is a compact, orientable, $3$-manifold with (possibly empty) boundary. We say $\overline{M}$ is {\em irreducible} if every smoothly embedded $2$--sphere in $\overline{M}$ bounds a $3$-ball in $\overline{M}$. An {\em incompressible surface} in $\overline M$ is the image of a compact, connected, orientable surface by an embedding $(\overline{\Sigma},\partial \overline{\Sigma}) \to (\overline{M},\partial \overline{M})$ such that $\overline \Sigma$ is not a sphere or disk, and such that the induced map on fundamental groups is injective. We often write $\overline \Sigma \subset \overline M$ for an incompressible surface, though it will always come from an embedding of pairs.

An incompressible annulus $(A,\partial A) \subset (\overline M,\partial \overline M)$ is {\em boundary parallel} if there is an embedded solid torus $V \subset \overline M$ so that $\partial V$ is the union of $A$ with an annulus $A' = V \cap \partial \overline M$ intersecting each other along their common boundary $\partial A = \partial A'$. If $\overline M$ is irreducible, then by \cite[Lemma~5.3]{Waldhausen}, an incompressible annulus is boundary parallel if and only if the inclusion is homotopic by a homotopy of pairs $h_t \colon (A,\partial A) \to (\overline M, \partial \overline M)$ to a map of $A$ into $\partial \overline M$. An {\em essential annulus} in $\overline M$ is an incompressible annulus which is not boundary parallel.
If no essential annulus exists, we say that $\overline M$ is {\em acylindrical}. 
The manifold $\overline{M}$ is {\em atoroidal} if for every incompressible torus, the inclusion is homotopic into $\partial \overline{M}$.

Let $\mathcal{F}$ be a transversely oriented, codimension-one foliation  of a compact 3-manifold $\overline{M}$ for which the boundary is a union of leaves. We say such a foliation is \textit{taut} if for every leaf $\lambda$ of $\mathcal{F}$ there is a closed curve or embedded arc $(\gamma_\lambda,\partial \gamma_\lambda) \subset (\overline M,\partial \overline M)$ which is transverse to $\mathcal{F}$ and intersects $\lambda$.

We will need the following ``normal form" theorem for surfaces in taut foliated $3$-manifolds due to Roussarie \cite[Th\'eor\`em 1]{Rou74} (see also \cite[Theorem~9.5.5]{CandConII}). A strengthening of an absolute version of this theorem is due to Thurston \cite{Thu72} (see also \cite{Has86}), but we will use the following relative version.

\begin{theorem} \label{theorem:taut}
Suppose $\mathcal{F}$ is a transversely oriented, taut foliation of a 3-manifold with boundary, $\overline M$, for which $\partial \overline{M}$ is a union of leaves. Suppose $\overline{\Sigma} \subset \overline{M}$ is a properly embedded, incompressible torus or annulus. After an isotopy of the embedding which is the identity on the boundary (in the annulus case) we have:
\begin{enumerate}
\item $\overline{\Sigma}$ is a torus and is either transverse to $\mathcal F$ or is a leaf of $\mathcal F$; or
\item $\overline{\Sigma}$ is an annulus and is transverse to $\mathcal F$, except at finitely many circle tangencies occurring in the interior of $\overline{M}$.
\end{enumerate}
\end{theorem}

\subsection{Dehn filling}
Given a $3$-manifold $\overline{M}$ with a torus boundary component $T$, we can \emph{Dehn fill} $\overline{M}$ along $T$ by gluing a solid torus $D \times S^1$ to $\overline{M}$ by a homeomorphism $\varphi: D \times S^1 \to T$. The homeomorphism type of the Dehn filling is determined by the isotopy class of $\beta \subset T$ to which $\partial D \times \{x\}$ is identified. We call $\beta$ the {\em Dehn filling coefficient}, and write $\overline{M}(\beta)$ to denote the {\em $\beta$--Dehn filled manifold}. Choosing a basis $\mu,\lambda$ for $\pi_1(T) \cong \mathbb Z^2$, the isotopy class of $\beta$ on $T$ is given by $\beta = \pm(p \lambda + q \mu)$, with $p$ and $q$ relatively prime. So, we may also refer to this Dehn filling coefficient as $(p,q)$, writing $\overline M(\beta) = \overline M(p,q)$. Additionally, we allow a Dehn filling coefficient of ``$\infty$'', which means that we do not glue in a solid torus at all (i.e.,~we leave the boundary component unfilled). Any sequence of distinct slopes is said to converge to $\infty$. 

If $\overline{M}$ is a $3$-manifold with $k$ torus boundary components and $\beta = (\beta_1,\ldots,\beta_k)$ is a $k$--tuple of Dehn filling coefficients (some or all of which may be $\infty$), we may perform $\beta_i$--Dehn filling on each boundary component, which we denote by $\overline{M}(\beta_1,\ldots,\beta_k)$ or simply $\overline{M}(\beta)$. Finally, we also consider Dehn filling on torus ends of a noncompact manifold. These are ends which have neighborhoods homeomorphic to $T^2 \times (0,1)$, where $T^2$ is a torus. We perform Dehn filling on such ends by first deleting  a product open neighborhood of the end and then Dehn filling the resultant boundary torus.

A link in a $3$--manifold,  $L \subset \overline{M}$, is the image of an embedding of a disjoint union of circles (a link with one component is a knot). Given a link $L \subset \overline M$, we let $\overline{M} - L$ denote the exterior of $L$ in $\overline{M}$ obtained by removing an open tubular neighborhood of $L$ from $\overline{M}$. 

It is straightforward to see that $(1,n)$--Dehn filling the exterior of a knot which is a simple closed curve in a surface fiber of a fibered $3$-manifold (for an appropriate basis of the fundamental group of the torus boundary) has the same effect as changing the monodromy by the $n^{th}$ power of a Dehn twist; see e.g.~\cite{StallingsFiber}.

\begin{proposition}\label{P:Stallings}
Let $\gamma$ be a curve in $Y$ and $K$ the associated knot in the mapping torus $M_f$ for $f \in \Map(Y)$. Then $(M_f-K)(1,n) \cong M_{f \circ T_{\gamma}^n}$.

More generally, if $L$ is a link consisting of curves $\gamma_1,\ldots,\gamma_k$ with $\gamma_i$ lying on fiber $Y \times \{\frac{i}{k+1}\} \subset Y \times (0,1) \subset M_f$, for each $i$, then \[ (M_f-L)((1,n_1),\ldots,(1,n_k)) \cong M_{f  \circ T_{\gamma_k}^{n_k} \circ  \cdots \circ T_{\gamma_1}^{n_1}}. \]
\end{proposition}

\subsection{Pared 3-manifolds} \label{S:pared}
The following definitions provide the appropriate topological framework for stating Thurston's Hyperbolization Theorem (see \Cref{Theorem:Thurston.Geometrization}), which ensures that the manifolds we are interested in admit hyperbolic structures. Given a $3$-manifold $\overline{M}$ and closed subset $\pare \subset \partial \overline{M}$, we say that $(\overline{M},\pare)$ is a {\em pared $3$-manifold} if:
\begin{enumerate}
\item \label{I:pared basic} $\overline{M}$ is compact, connected, orientable, irreducible and $\partial \overline M \neq \emptyset$;
\item \label{I:not v abelian} $\pi_1(\overline M)$ is not virtually abelian;
\item \label{I:what pare is} $\pare$ consists of incompressible annuli and tori, and contains all tori in $\partial \overline M$;
\item \label{I:annuli}  any incompressible annulus $(S^1 \times I,S^1 \times \partial I) \to (\overline{M},\pare)$ is properly homotopic into $\pare$ (i.e.,~homotopic as a map of pairs); and
\item \label{I:non cyclic} any non-cyclic abelian subgroup of $\pi_1(\overline M)$ is conjugate into the fundamental group of some component of $\pare$.
\end{enumerate} 

A pared manifold $(\overline{M},\pare)$ has {\em incompressible boundary} if $\partial \overline{M} - \pare$ is incompressible (when $\pare = \emptyset$, this condition simply means that $\partial \overline{M}$ is incompressible). An {\em essential annulus} in $(\overline{M},\pare)$ is an incompressible annulus $(S^1 \times I,S^1 \times \partial I) \to (\overline{M},\partial \overline{M} - \pare)$ which is not homotopic into $\partial \overline{M}$. One says that a pared manifold $(\overline{M},\pare)$ is {\em acylindrical} if it has incompressible boundary and contains no essential annulus. If $\pare = \emptyset$, this agrees with the definition of acylindrical given above.

Observe that if $(\overline M,\pare)$ is a pared manifold with incompressible boundary and $\pare$ meets each boundary component of $\overline M$ of genus at least $2$ in annuli whose core curves define a pants decomposition of these boundary components, then $(\overline M,\pare)$ is acylindrical. To see this, observe that if there were an incompressible annulus with boundary in $\partial \overline M - \pare$, then since this surface is a union of thrice-punctured spheres, the boundary of the annulus may be isotoped into $\pare$, and hence the entire annulus is properly homotopic relative to the boundary into $\pare$ by (\ref{I:annuli}).

Suppose $\overline M$ is a compact, connected, orientable, irreducible $3$-manifold with incompressible, nonempty boundary.  If all components of $\partial \overline M$ have genus at least $2$, then $(\overline M,\emptyset)$ is a pared manifold if $\overline M$ is atoroidal: conditions (\ref{I:pared basic})-(\ref{I:not v abelian}) are immediate, (\ref{I:what pare is})-(\ref{I:annuli}) are vacuous, and (\ref{I:non cyclic}) follows from the Torus Theorem; see e.g.~\cite[Corollary~IV.4.3]{JaSh}. Similarly, if $\pare \neq \emptyset$ and contains all tori in $\partial \overline M$, then $(\overline M,\pare)$ is pared if $\overline M$ is atoroidal and not  Seifert fibered, see \cite[Theorem~IV.4.1]{JaSh}.   

\begin{remark}
It is sometimes convenient to replace some or all of the annuli in $\pare$ by their core curves, so that $\pare$ is a union of tori, annuli, and simple closed curves. Note that up to homeomorphism, this does not affect the manifold $\overline{M}-\pare$ or its boundary $\partial \overline{M} - \pare$. 
\end{remark}

A ``pants block'' is a pared manifold built as a quotient of $\overline \Sigma \times I$, where $\overline \Sigma$ is a complexity-one surface (so $\overline \Sigma \cong \overline \Sigma_{1,1}$ or $\overline \Sigma \cong \overline \Sigma_{0,4}$) and $I$ an interval, say $I = [0,1]$, together with some additional decoration. We make this precise with the following.

\begin{defn}[Pants block] \label{D:pants block}
A {\bf pants block} is a pair $({\mathcal B},\partial_1 \mathcal B)$, where $\mathcal B$ is a handlebody of genus $2$ or $3$, and $\partial_1 \mathcal B \subset \partial \mathcal B$ is union of simple closed curves defining a pants decomposition of the boundary of $\mathcal B$. We further assume that $(\mathcal B,\partial_1 \mathcal B)$ admits a quotient map of pairs
\[ \varphi \colon (\overline \Sigma \times I,\alpha_0 \cup \alpha_1 \cup (\partial \overline \Sigma \times I)) \to (\mathcal B,\partial_1 \mathcal B) \] 
such that
\begin{enumerate}
    \item[(a)] $\overline \Sigma$ is a compact, complexity-one surface;
    \item[(b)] $\alpha_i \subset \overline \Sigma \times \{i\}$, for $i=0,1$ are curves of the form $\alpha_i = \alpha^i \times \{i\}$, with $\alpha^0,\alpha^1 \subset \overline \Sigma$ curves intersecting minimally (thus intersecting once in $\overline \Sigma_{1,1}$ and twice in $\overline \Sigma_{0,4}$); and
    \item[(c)] $\varphi$ is 1-to-1, except on $\partial \overline \Sigma \times I$, where $\varphi^{-1}(\varphi(x,t)) = \{x \} \times I$, for all $(x,t) \in \partial \overline \Sigma \times I$.
\end{enumerate}
We sometimes refer to a pants block simply as a {\bf block}.
\end{defn}

Up to homeomorphism (ignoring the pared locus), there are exactly two distinct pants blocks, as shown in \Cref{fig:pants-blocks}. 
It is straightforward to see that pants blocks are acylindrical pared manifolds using the product structure, since
\[ \mathcal B - \partial_1 \mathcal B \cong \overline \Sigma \times I  -  (\alpha_0 \cup \alpha_1 \cup (\partial \overline \Sigma \times I)). \] 

\begin{figure}
    \centering
    \includegraphics[width=13cm]{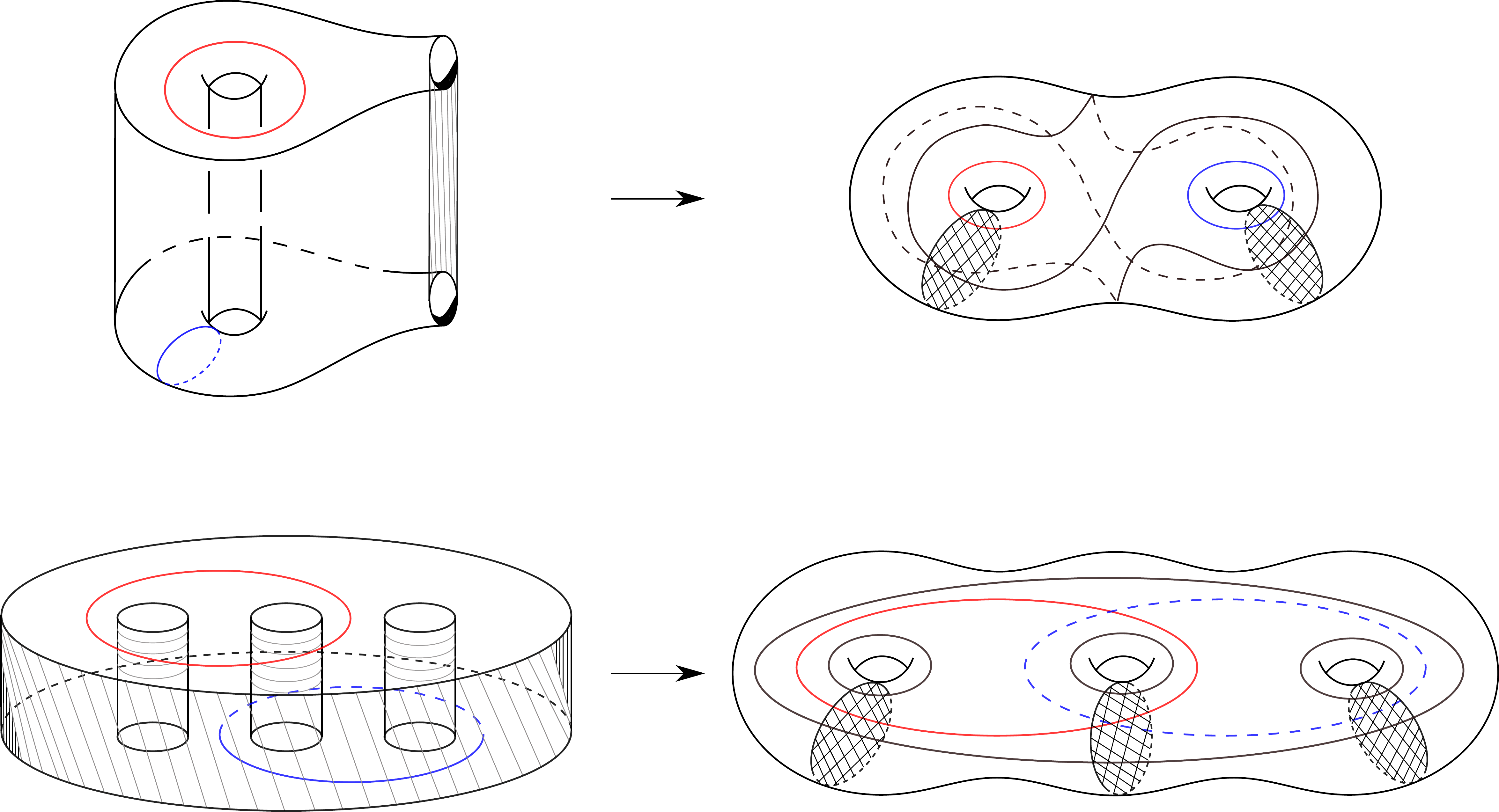}
    \caption{Pants blocks corresponding to the one-holed torus and 4-holed sphere. The curve $\alpha_0$ is shown in blue, the curve $\alpha_1$ is shown in red, and the image (under $\varphi$) of $\partial \overline \Sigma$ is shown in grey. The meridian disks for the handlebodies are drawn for illustrative purposes only.}
    \label{fig:pants-blocks}
\end{figure}

It will be convenient to keep track of where the components of $\partial_1 \mathcal B$ came from. As such, we use the following notation for the components of $\partial_1 \mathcal B$: 
\[ \partial_1^- \mathcal B = \varphi(\alpha_0) \quad \partial_1^+ \mathcal B = \varphi(\alpha_1) \quad \partial_1^v \mathcal B = \varphi(\partial \overline \Sigma \times I).\]
We also write $\partial_2 \mathcal B = \partial \mathcal B - \partial_1 \mathcal B$, which is a disjoint union of thrice-punctured spheres, $\Sigma_{0,3}$, (two for $\overline \Sigma_{1,1}$ and four for $\overline \Sigma_{0,4}$), and set \[ \partial_2^- \B = \partial_2 \B \cap \varphi(\overline \Sigma \times\{0\}) \mbox{ and } \partial_2^+ \B = \partial_2 \cap \varphi(\overline \Sigma \times \{1\}).\]

We often denote a pants block simply as $\mathcal B$, suppressing the decoration and the quotient map, though they are in fact part of the structure. 
When we want to distinguish between the two types of pants blocks built from $\overline \Sigma_{0,4}$ and $\overline \Sigma_{1,1}$, we write $\mathcal B^S$ and $\mathcal B^T$, respectively.

\begin{remark}
We note that $(\mathcal B,\partial_1 \mathcal B)$ is a pared $3$--manifold: everything except condition \eqref{I:annuli} in the definition is immediate. Condition \eqref{I:annuli} follows from an identification of the fundamental group as a free group, in which all components of $\partial_1 \mathcal B$ are distinct homotopy classes.  The product structure implies $\partial_2 \mathcal B$ is incompressible, and by the comments above, $(\mathcal B,\partial_1 \mathcal B)$ is acylindrical.  
\end{remark}

\subsection{Geometry of hyperbolic 3-manifolds}
By a {\em convex hyperbolic $3$-manifold}, we mean a $3$-manifold which is the quotient of a closed, convex subset of $\mathbb H^3$ by a discrete, torsion-free subgroup $\Gamma < \PSL_2(\mathbb C)$. A {\em convex hyperbolic metric} on a $3$-manifold, $M$, is a metric which makes $M$ isometric to a convex hyperbolic $3$-manifold. A special case of Thurston's Geometrization Theorem for Haken manifolds is the following; see \cite{Thurston.Geometrization.1, mcmullen1992riemann,Morgan}.

\begin{theorem} \label{Theorem:Thurston.Geometrization} Suppose $(\overline{M},\pare)$ is a pared $3$-manifold and $\partial \overline{M}-\pare \neq \emptyset$ is incompressible. Then $\overline{M}-\pare$ admits a convex hyperbolic metric $\sigma$. If $(\overline{M},\pare)$ is additionally assumed to be acylindrical and {\em non-degenerate}, then there is a unique (up to isometry) convex hyperbolic structure $\sigma_{min}$ on $\overline{M}-\pare$ for which $\partial \overline{M}-\pare$ is totally geodesic.
\end{theorem}

The {\em non-degeneracy} assumption on $(\overline{M},\pare)$ rules out the case that the structure degenerates to a lower-dimensional manifold. Because of acylindricity, this simply means that $(\overline{M},\pare)$ is not a pair of pants times an interval, with $\pare$ being the three annuli which are the boundary curves of the pants times the interval. We also note that $\pare$ is allowed to be empty in the theorem above, though $\partial \overline M$ is not.  The uniqueness of $\sigma_{min}$ in \Cref{Theorem:Thurston.Geometrization} is a consequence of Mostow-Prasad Rigidity by doubling $\overline {M}-\pare$ over the boundary.

We will write $\Vol(\overline M-\pare,\sigma)$ for the total volume of the manifold $\overline M-\pare$ with respect to the metric $\sigma$.  In the special case that $(\overline M,\pare)$ is acylindrical, we will write $\Vol(\overline M-\pare)$ to mean $\Vol(\overline M-\pare,\sigma_{min})$, where $\sigma_{min}$ is the metric given by \Cref{Theorem:Thurston.Geometrization}.

A pants block $(\mathcal B,\partial_1 \mathcal B)$ is an acylindrical, pared manifold, as described in the remark above; thus \Cref{Theorem:Thurston.Geometrization} gives a unique hyperbolic structure on $\mathring{\B} = \B - \partial_1 \B$ with totally geodesic boundary. In fact, this can be constructed explicitly from  regular ideal octahedra: one octahedron in the case of $\mathcal B^T$ and two in the case of $\mathcal B^S$; see e.g.~Agol \cite[Lemma 2.3]{Agol-oct}. 

\begin{proposition} \label{P:geometric block}
The manifold $\mathring{\B}$ has a convex hyperbolic metric $\sigma_{\B}$ with totally geodesic boundary consisting of a union of thrice-punctured spheres. The metric $\sigma_\B$ is obtained by gluing together half the faces of a regular ideal octahedron for $\B = \B^T$ and two regular ideal octahedra for $\B = \B^S$. Consequently,
\[ \Vol(\mathring{\B}^T,\sigma_{\B^T}) = V_\oct \quad \mbox{ and } \quad \Vol(\mathring{\B}^S,\sigma_{\B^S}) = 2V_\oct, \]
where $V_\oct$ is the volume of a regular ideal octahedron.
\end{proposition} 

The notation $\sigma_{min}$ is justified by the following theorem due to Storm \cite[Theorem~3.1]{Storm2}; see also Storm \cite{Storm} and Agol--Storm--Thurston \cite{AgolStormThurston}.

\begin{theorem} \label{T:Storm}
Given a pared manifold $(\overline{M},\pare)$ which is acylindrical and non-degenerate with $\partial \overline M - \pare \neq \emptyset$ and given any convex hyperbolic metric $\sigma$ on $\overline M-\pare$, we have $\Vol(\overline M-\pare,\sigma_{min}) \leq \Vol(\overline M - \pare,\sigma)$ with equality precisely when $\sigma_{min} = \sigma$.
\end{theorem}

Suppose $(\overline{M},\pare)$ is a pared $3$-manifold and $\partial \overline M - \pare \neq \emptyset$ is incompressible. In this situation, we define
\begin{equation}\label{E:V inf pared defined} \iVol(\overline M - \pare) = \inf\{\Vol(\overline M -\pare,\sigma) \mid \sigma \mbox{ a convex hyperbolic metric on } \overline{M}-\pare \}.
\end{equation} 
If $\pare = \emptyset$, we also write $\iVol(\overline M) = \iVol(\overline M - \emptyset)$.

By \Cref{T:Storm}, $\iVol(\overline{M} - \pare)$ is realized as $\Vol(\overline M - \pare,\sigma_{min})$ when $(\overline M,\pare)$ is acylindrical. The main fact we will need about the general case is the second half of the result of Storm \cite[Theorem~3.1]{Storm2} referenced above.

\begin{theorem}\label{T:Storm 2}
If $(\overline M,\pare)$ is a pared $3$-manifold with $\partial \overline M-\pare \neq \emptyset$ incompressible, then $\iVol(\overline M - \pare)$ is equal to half of the simplicial volume of the double of $\overline M$ over $\overline{\partial \overline M-\pare}$. In particular, for any degree $k < \infty$ covering, $(\overline M',\pare') \to (\overline M,\pare)$, one has $\iVol(\overline M'- \pare') = k \iVol(\overline M - \pare)$.
\end{theorem}

It will also be useful to relate the infimal volume of a pared manifold to the infimal volume of the manifold after forgetting the paring locus; see Storm \cite[Corollary~1.2, Corollary 2.10, and Theorem~3.2]{Storm2}.
\begin{theorem} \label{T:Storm3}
Suppose $(\overline M,\pare)$ is a pared acylindrical $3$-manifold and that $\partial \overline M$ is incompressible, nonempty, and all components have genus at least $2$. Then $\iVol(\overline M) \leq \iVol(\overline M - \pare)$.
\end{theorem}

Since $(\overline M,\pare)$ is pared and $\partial \overline M$ contains no tori we see that $(\overline M,\emptyset)$ is also pared and hence admits a convex hyperbolic structure by \Cref{Theorem:Thurston.Geometrization} so $\iVol(\overline M)$ is well-defined.  

Storm's proof of \Cref{T:Storm3}  exploits the behavior of volume (and simplicial volume) under Dehn filling. The behavior of the hyperbolic and simplicial volumes is part of Thurston's Dehn Surgery Theorem; see Thurston \cite[Theorem~5.8.2, Theorem 6.5.4, and Theorem 6.5.6]{TNotes}, as well as \cite[Theorem~1A]{NeumannZagier} and \cite[15.4.1]{Martelli2016}. The version of Thurston's theorem, which we state here, follows by doubling over the totally geodesic boundary and appealing to \cite{MenascoReid}. The notion of a sequence of slopes going to $\infty$ is natural from the perspective of Thurston's Hyperbolic Dehn Surgery Theorem. 

\begin{theorem} \label{T:filling volume}
Let $(M,\sigma)$ be a finite volume hyperbolic $3$-manifold with totally geodesic boundary and $k$ torus cusps. Suppose $\beta^n = (\beta_1^n,\ldots,\beta_k^n)$ are a sequence of Dehn filling coefficients so that for each $i$, either $\beta_i^n = \infty$ for all $n$, or $\beta_i^n \to \infty$. Then for $n$ sufficiently large, $M(\beta^n)$ admits a complete hyperbolic structure $\sigma(\beta^n)$ with totally geodesic boundary and
\[ \Vol(M(\beta^n),\sigma(\beta^n)) \to \Vol(M,\sigma).\]
Moreover, 
\[\Vol(M(\beta^n),\sigma(\beta^n)) \leq \Vol(M,\sigma) \]
with equality if and only if $\beta^n = (\infty,\ldots,\infty)$.
\end{theorem}

\section{Compactified mapping tori}\label{section:Compactification}

Given an end-periodic homeomorphism $f \colon S\to S$, we let
\[M_f =S\times[0,1]/(x,1) \sim (f(x),0) \]
denote the mapping torus of $f$.  We will write $(\psi_s)$ to denote the {\em suspension flow} on $M_f$, defined as the descent to $M_f$ of the local flow on $S \times [0,1]$ given by $\psi_s(x,t) = (x,t+s)$, when $0 \leq t \leq t+s \leq 1$. 
The goal of this section is to prove the following result.

\begin{proposition}\label{proposition:interior}
Let $f \colon S \to S$ be an irreducible, end-periodic homeomorphism of a surface with finitely many ends, all accumulated by genus.
Then the mapping torus, $M_f$, is the interior of a compact, irreducible, atoroidal $3$-manifold, $\overline{M}_f$, with incompressible boundary. If $f$ is strongly irreducible, then $\overline{M}_f$ is also acylindrical.
\end{proposition}

The fact that $M_f$ is the interior of a compact, irreducible manifold with boundary is well-known: the boundary points are obtained by adding ``endpoints at infinity'' to each positive (respectively, negative) ray of the suspension flow that passes through $\mathcal U_+ \times \{0\}$ (respectively, $\mathcal U_- \times \{0\}$); see \cite{Fenley-depth-one}. 
As we will make explicit use of this compactification in our discussion, we discuss this in more detail. 

Suppose $f$ is any end-periodic homeomorphism of $S$. We begin by describing the compactification of $M_f$ in a way that is well-suited to our discussion. As a starting point, we identify the infinite cyclic cover of $M_f$ dual to the fibration over $S^1$ by
\[ p \colon S \times (-\infty,\infty) \to M_f \]
together with the action of the covering group isomorphic to $\mathbb Z$, generated by the homeomorphism $F$ given by
\[ F(x,t) = (f(x),t-1).\]
This is a covering action with quotient $M_f$. Indeed, $S \times [0,1]$ is a fundamental domain, and the only points that are nontrivially identified are of the form $(x,1) \sim F(x,1) = (f(x),0)$.

Recall that $\mathcal U_+,\mathcal U_- \subset S$ are the positive and negative escaping sets, respectively, as defined in \Cref{S:end-periodic defs}. 
We define a partial compactification of $S \times (-\infty,\infty)$ inside $S \times [-\infty,\infty]$ by 
\[ \widetilde M_\infty = S \times (-\infty,\infty)\, \, \sqcup \,\, \mathcal U_+ \times \{\infty\} \,\, \sqcup \,\, \mathcal U_- \times \{-\infty\}.\]
Since $\mathcal U_+$ and $\mathcal U_-$ are $f$--invariant, we may extend the action of $\langle F \rangle$ by defining
\[ F(x,\pm \infty) = (f(x),\pm \infty)\]
whenever $x \in \mathcal U_{\pm}$.
By \Cref{L:escaping sets are covering spaces}, $\langle f\rangle|_{\mathcal U_\pm}$ acts cocompactly, with quotient closed (not necessarily connected) surfaces $S_+ = \mathcal U_+/\langle f \rangle$ and $S_- = \mathcal U_-/\langle f \rangle$.

\begin{lemma} \label{L:prop discont} Suppose $f$ is end-periodic and $\langle F \rangle$ acts on $S \times(-\infty,\infty)$ as above. Then the action of $\langle F \rangle \cong \mathbb Z$ extends to a properly discontinuous, cocompact action on $\widetilde M_\infty$ by $F(x,\pm \infty) = (f(x),\pm \infty)$. The quotient, $\overline M_f$, is thus a compact manifold with $\partial \overline M_f \cong S_+ \sqcup S_-$.
\end{lemma}
\begin{proof} Recall that
\[ \mathcal U_\pm = \bigcup_{k=1}^\infty f^{\mp k}U_\pm \]
where $U_\pm$ are good nesting neighborhoods. 
Since $\mathcal U_\pm$ is invariant by $f$, it follows that the action of $F$ extends to an action by the given formula. 

By considering an appropriate open covering of $\widetilde M_\infty$, 
we note that for any compact set $K \subset \widetilde M_\infty$, there is some $n \in \mathbb N$ and a decomposition $K = K_+ \cup K_- \cup K_0$ where
\[ K_+ \subset (f^{-n}(U_+) -  f^n(U_+)) \times [0,\infty],\]
\[ K_- \subset (f^n(U_-)  -  f^{-n}(U_-)) \times [-\infty,0],\]
and
\[ K_0 \subset (S  -  (f^n(U_+) \cup f^{-n}(U_-))) \times [-n,n].\]

Note that only finitely many $\langle F \rangle$--translates of each of $K_0$, $K_+$, or $K_-$ can intersect themselves (the case of $K_0$ by inspection of the second coordinate, and $K_\pm$ by inspection of the first coordinate). 
Therefore, we must show that only finitely many $\langle F \rangle$--translates of any one of these sets can nontrivially intersect another one. And for this, it suffices to consider the $\langle F \rangle$--translates of $K_0$ with either $K_+$ or $K_-$ and the $\langle F \rangle$--translates of $K_+$ with $K_-$.

In the first case, we observe that for any integer $m > n$, $F^{m}(K_0) \cap K_+ = \emptyset$, by considering the second coordinates. On the other hand, for any integer $m < -2n$, inspection of the first coordinate shows that $F^{m}(K_0) \cap K_+ = \emptyset$. Thus $F^m(K_0) \cap K_+ = \emptyset$ for all but at most $3n+1$ values of $m \in \mathbb Z$. Similar reasoning applies to show that at most $3n+1$ $\langle F \rangle$--translates of $K_0$ nontrivially intersect $K_-$.

Next observe that for all integers $m < 0$, $F^m(K_+) \cap K_- = \emptyset$, by considering the second coordinates. Finally, note that for all $k>n$, $f^k(U_+) \cap f^n(U_-) = \emptyset$, and so if $m > 2n$, then $F^m(K_+) \cap K_- = \emptyset$.
Consequently, there are at most $2n+1$ values of $m \in \mathbb Z$ so that $F^m(K_+) \cap K_- \neq \emptyset$. Therefore, the action is properly discontinuous, and hence $\overline M_f = \widetilde M_\infty/\langle F \rangle$ is a $3$-manifold with boundary $S \pm$.

All that remains is to prove that the action is cocompact. For convenience, define
\[ C_n = S  -  (f^n(U_+) \cup f^n(U_-)) \quad \mbox{ and } \quad \Delta_\pm = \overline U_\pm  -  f^{\pm 1}(U_\pm). \]
We observe that $C_n$ is compact for all $n > 0$, and $\Delta_\pm$ is a compact fundamental domain for $\langle f \rangle$ on $\mathcal U_\pm$.

Now we suppose $\{z_k\}_{k=1}^\infty \subset \overline M_f$ is any sequence, and show that after passing to a subsequence, it converges. Since $\partial \overline M_f = S_+ \cup S_-$ is compact, without loss of generality we may assume $z_k \in M_f$, for all $k$.

For each $k \geq 1$, let $(x_k,t_k)$ be the unique point in $p^{-1}(z_k) \cap S \times [0,1)$. Since each $C_n$ is compact, we may assume that for all $n$, only finitely many $x_k$ lie in $C_n$. 
After passing to a subsequence, we may assume that for any $n$, all but finitely many $x_k$ lie in $f^n(U_+)$ or in $f^{-n}(U_-)$. Assuming the first case (the second is similar) and passing to a further subsequence, we can choose an increasing sequence $\{n_k\}_{k=1}^\infty$ so that $x_k \in f^{n_k}(\Delta_+)$. Now observe that
\[ F^{-n_k}(x_k,t_k) = (f^{-n_k}(x_k),t_k+n_k) \in \Delta_+ \times [0,\infty].\]
The set on the right is compact, and so the sequence converges. It follows that $\{z_k\}$ converges, as required.
\end{proof}

Note that the foliation of $\widetilde M_\infty$ by fibers of the projection $\widetilde M_\infty \to [-\infty,\infty]$ is invariant by $\langle F \rangle$, and so descends to a transversely oriented foliation $\mathcal F$ of $\overline M_f$ for which $\partial \overline M_f$ is a union of (compact) leaves.

\begin{lemma} \label{L:Taut double} Suppose $f$ is an end-periodic homeomorphism of $S$ and $\mathcal F$ the foliation on $\overline M_f$ just defined. Then, $\mathcal F$ is a taut foliation. 
\end{lemma}
\begin{proof} Any transversal of the foliation of $\widetilde M_\infty$ connecting a point of $\mathcal U_- \times \{-\infty\}$ to a point of $\mathcal U_+ \times \{\infty\}$, projects to a transversal of $\mathcal F$ on $\overline M_f$.
Doing this for each component of $\mathcal U_+$ and $\mathcal U_-$, produces enough transversals to intersect every leaf of $\mathcal F$, as required.
\end{proof}

We let $D \overline M_f$ be the double of $\overline M_f$ over its boundary. This induces a foliation $\mathcal F_D$ on $D\overline M_f$ which is also taut.

\begin{proposition} \label{P:its hyperbolic}
Suppose $f$ is an irreducible, end-periodic homeomorphism. Then $\overline M_f$ is a compact, irreducible, atoroidal $3$-manifold with incompressible boundary.
\end{proposition}
\begin{proof}
By \Cref{L:prop discont}, we know that $\overline M_f$ is a compact $3$-manifold with boundary. Since $M_f$ is a mapping torus, it follows that $M_f$ (and thus $\overline{M}_f$) is irreducible. Since each component of $\partial \overline M_f$ is covered by a component of the $\pi_1$--injective subsurface $\mathcal U_\pm \times \{\pm \infty\} \subset \widetilde M_\infty$, it follows that $\partial \overline M_f$ is incompressible.

To prove that $\overline M_f$ is atoroidal, we suppose that $T \subset \overline M_f$ is an incompressible torus and derive a contradiction. By \Cref{theorem:taut}, after an isotopy of the inclusion, we can assume that $T$ is transverse to $\mathcal F$ (since $\mathcal F$ has no torus leaves).  

Cutting $M_f$ open along a fiber, we obtain a product $S \times [0,1]$ and the torus is cut into a union of annuli connecting a finite, disjoint union of curves in $S \times \{0\}$ to a finite disjoint union of curves in $S \times \{1\}$ (the torus cannot be disjoint from any fiber since $S \times [0,1]$ contains no incompressible tori). Since $M_f$ is obtained by gluing $S \times \{1\}$ to $S \times \{0\}$ by $f$, this produces a finite $f$--invariant union of curves, and thus an $f$--periodic curve. This is a contradiction to the irreducibility assumption, and thus $\overline M_f$ is atoroidal.
\end{proof}

Suppose $f$ is an irreducible, end-periodic homeomorphism. Since the boundary components of $\overline{M}_f$ are closed surfaces of genus at least $2$, \Cref{P:its hyperbolic} and \Cref{Theorem:Thurston.Geometrization} imply that $\overline{M}_f$ admits a convex hyperbolic metric, since $(\overline{M}_f,\emptyset)$ is pared; see the discussion following \Cref{Theorem:Thurston.Geometrization}.

\begin{lemma} \label{L:annuli analysis}
Suppose $f$ is an irreducible, end-periodic homeomorphism and $(A,\partial A) \subset 
(\overline M_f, \partial \overline M_f)$ is an essential annulus. Then both boundary components of $A$ are in $S_+$ or both are in $S_-$, and their preimages in $\widetilde M_{\infty}$ are lines in the surface $\mathcal U_+ \times \{\infty\}$ or $\mathcal U_- \times \{- \infty\}$, respectively.

If $f$ is strongly irreducible, then $\overline M_f$ is acylindrical.
\end{lemma}
\begin{proof} An application of \Cref{theorem:taut} implies that, after an isotopy of the inclusion which is the identity on $\partial A$, we may assume $A$ is transverse to $\mathcal F$, except possibly at finitely many circle tangencies occurring in $M_f \cap A$. In particular, $\mathcal F$ induces a foliation $\mathcal F_A$ on $A$ whose leaves are components of intersection with the leaves of $\mathcal F$ (consequently, each component of $\partial A$ is a leaf of $\mathcal F_A$).

If there is a circle leaf $\alpha$ of $\mathcal F_A$ in the interior of $A$ (this happens, for example, if there is a circle tangency) then $\alpha$ is also a curve in some fiber $S$ of $M_f$. Since $\alpha$, viewed as a loop on $A$, generates $\pi_1(A)$ and lies on $S$, it follows that the inclusion of $A$ lifts to an embedding into $\widetilde M_\infty$. In particular, the preimage of $A$ in this covering space is a disjoint union of copies of $A$ permuted by the covering transformation $F$.

Suppose one boundary component of $A$ lies on $S_-$ and the other on $S_+$. Then any lift of $A$ has one boundary component in $\mathcal U_- \times \{-\infty\}$ and the other in $\mathcal U_+ \times \{\infty\}$. This means there is an $f$-reducing curve which corresponds to $(S\times \{0\}) \cap A$, contradicting irreducibility.


Therefore, both boundary components of $A$ lie in $S_+$ or both lie in $S_-$. In $\widetilde M_\infty \subset S \times [-\infty,\infty]$, any lift of $A$ is therefore boundary parallel, implying the same for $A$ in $\overline M_f$, contradicting the fact that $A$ is essential.

From the contradictions above, we conclude that $A$ has no closed leaves except its boundary curves.  In this case, all leaves of $\mathcal F_A$ spiral in toward the boundary components of $A$.  
Let $\widetilde A \subset \widetilde M_\infty$ be any component of the preimage of $A$. Note that the foliation $\mathcal F_A$ lifts to a foliation $\widetilde{\mathcal F}_A$ on $\widetilde A$ whose non-boundary leaves are intersections with fibers $S \times \{t \}$. 

\begin{claim} $\widetilde A$ is the universal covering of $A$ and some power of $F$ preserves $\widetilde A$ and generates the covering group.
\end{claim}

\begin{proof} Either some positive power of $F$ preserves $\widetilde A$, or no positive power does. In the former case, the smallest such positive power generates an infinite order covering action on $\widetilde A$ with quotient $A$. In this case, $\widetilde A$ must be the universal cover and we are done. If there is no positive power that preserves $\widetilde A$, then $\widetilde A$ must be a lift of $A$. In this case, the boundary of $\widetilde A$ is a union of curves in $\mathcal U_+ \times \{\infty\} \cup\, \mathcal U_- \times \{-\infty\}$. 
Since the non-boundary leaves of $\widetilde{\mathcal F}_A$ limit onto the boundary curves of $\widetilde A$, we see that some fiber $S \times \{t\}$ contains a path (the leaf) that accumulates on $\mathcal U_\pm \times \{\pm \infty\}$, which is impossible. This contradiction means we are in the first case that $\widetilde A$ is the universal cover, as required.
\end{proof}

Let $F^m$ be the smallest positive power of $F$ that preserves $\widetilde A$ and thus generates the covering group of $\widetilde A \to A$. Let $\ell \subset (S \times \{t\}) \cap \widetilde A$ be a leaf of $\widetilde{\mathcal F}_A$ and consider the strip in $\widetilde A$ between $\ell$ and $F^m(\ell)$, which is contained in $S \times [t,t+m]$. This is a union of leaves of $\widetilde{\mathcal F}_A$ which are slices of a product structure on the strip. 
Note that in the covering to $A$ (or $\overline M_f$), the restriction to each leaf in the strip injects to a leaf of $\mathcal F_A$ and spirals in toward $\partial A$. In particular, the two ends of the strip must exit ends of $S \times [t,t+m]$ (since the only sequences in $S \times [t,t+m]$ which when projected to $\overline M_f$ accumulate on $\partial \overline M_f$ are those that exit the ends). Since $F(x,s) = (f(x),s-1)$, it follows that after projecting out the $\mathbb R$--direction, $\ell$ projects to a line $\ell_0$ in $S$ and the strip projects to a proper homotopy from $\ell_0$ to $f^m(\ell_0)$. That is, $\ell_0$ is an $f$--periodic line. We note that $\ell_0$ is in fact an essential line. For, if not, then it would bound an infinite monogon (topological half-plane) in $S$, and hence each leaf of $\widetilde{\mathcal F}_A$ would also bound a monogon in the surface slice $S \times \{s\}$. The union of these would project to the interior of a solid torus in $\overline M_f$ whose boundary is the union of $A$ and an annulus in $\partial \overline M_f$, contradicting the fact that $A$ is essential.
If $f$ is strongly irreducible, this is a contradiction. Therefore, in this case, there is no essential annulus and $\overline M_f$ is acylindrical. 

In general, the boundary of $\widetilde A$ consists of two lines in $\mathcal U_\pm \times \{\pm \infty\}$. If one of these lines lies in $\mathcal U_- \times \{-\infty\}$ and the other lies in $\mathcal U_+ \times \{\infty\}$, then $\ell_0$ is an AR-periodic line with one end in an attracting end of $S$ and the other in a repelling end, a contradiction. It follows that both components of the boundary of $\widetilde A$ are contained in $\mathcal U_+ \times \{ \infty \}$ or both are contained in $\mathcal U_- \times \{-\infty\}$. This proves the first claim of the lemma, and since we have already proved the second claim, we are done.
\end{proof}

\begin{proof}[Proof of \Cref{proposition:interior}]
This is immediate from \Cref{P:its hyperbolic} and \Cref{L:annuli analysis}.
\end{proof}

Finally, we note that the suspension flow $(\psi_s)$ on $M_f$ can be reparameterized to a local flow so that the flow lines that exit every compact set limit to a point on $\partial \overline M_f$ in finite time. We do this so that it extends to a local flow $(\bar \psi_s)$ on $\overline M_f$. This lifts to a local flow on $\widetilde M_\infty \subset S \times [-\infty,\infty]$ and has flow lines given by $\{x\} \times [-\infty,\infty]$, with either or both endpoints missing if $x$ is not in one or both of $\mathcal U_\pm \times \{\pm \infty\}$. Moreover, we can assume that $(\bar \psi_s)$ on $\overline M_f$, and its negative on the reflected copy of $\overline M_f$ in the double $D \overline M_f$, glue together to a well-defined flow $(\psi_s^D)$ on $D \overline M_f$.

\section{A quotient from a path of pants decompositions} \label{Section:quotient construction}

In this section, we begin by describing a method for decomposing $3$-manifolds that will be useful in proving bounds on volumes (see the remark below for other instances of such decompositions in the literature). We will use this idea to prove Theorems~\ref{T:upper} and \ref{T:component volume} as well as the corollaries described in the introduction. In our situation, the decomposition will arise naturally from a quotient map with various properties; see \Cref{P:block decomposition}.

\subsection{Block decompositions} First, recall from \Cref{D:pants block} that a pants block is a compact $3$-manifold, $\B$, with pared boundary $\partial_1 \B \subset \partial \B$, 
where $\partial_1 \B$ is a disjoint union of three or six embedded circles, depending on whether the block comes from $\Sigma_{1,1}$ or $\Sigma_{0,4}$, respectively. We view $\B$ as a stratified space (for bookkeeping purposes only), with strata being $\partial_1 \B$, $\partial_2 \B = \partial \B - \partial_1 \B$, and the interior $\intt(\B)$. We also recall that $\mathring{\B} = \B - \partial_1 \B$ admits a complete hyperbolic structure with totally geodesic, thrice-punctured sphere boundary components.

Suppose $N$ is a compact $3$-manifold, possibly with boundary, and $L \subset N$ a {\em link}. This is the image of a disjoint union of finitely many circles by a (smooth) embedding. We require each component of $L$ to be contained in the interior of $N$ or the boundary of $N$ (that is, if a component intersects the boundary, then it is contained in it).

\begin{defn}
A {\em block map} to $N$ relative to $L$ is a map of pairs $(\B,\partial_1 \B) \to (N,L)$ which is injective on each component of each stratum of $\B$, and so that the image of each component of $\partial_2\B$ is either disjoint from $\partial N$ or contained in it (but distinct components may have the same image). We also require that the image of $\B$ meets $L$ precisely in the image of $\partial_1 \B$. Given a block map, it is convenient to identify the block with its image $\B \subset N$, with the map itself implicit, thus allowing us to identify the components of each stratum in $N$ (some may be equal).

A {\em block decomposition of $N$ relative to $L$} (or a {\em block decomposition of $(N,L)$}) is a finite set of block maps $\{\B_i\}_{i=1}^n$ to $N$ relative to $L$, so that (after identifying the blocks with their images) we have
\begin{enumerate}
    \item the interiors of the $\B_i$ are pairwise disjoint;
    \item the union of the blocks is all of $N$; and
    \item whenever two blocks intersect, they do so in a union of components of strata.
\end{enumerate}

Thus, we can view $N$ as being built from a finite number of blocks. The blocks are glued together in pairs along closures of three-punctured spheres in their boundaries. 
The link $L$ is the union of the images of the $1$--boundaries, $\bigcup \partial_1 \B_i$. 
\end{defn}

Note that if $(N,L)$ admits a block decomposition, then $L \cap \partial N$ defines a pants decomposition of $\partial N$ (which may be empty).

\begin{remark} Using block decompositions to build $3$-manifolds is very similar to the construction of Minsky's ``model manifold" defined by Minsky \cite{ELC1} and used by Brock--Canary--Minsky \cite{ELC2} in the proof the Ending Lamination Conjecture. Our particular definition includes as a special case the construction Agol used in \cite{Agol-oct} for mapping tori of finite-type surfaces.
\end{remark}

An immediate consequence of \Cref{P:geometric block} is the following result. Recall that $V_\oct$ is the volume of the regular ideal octahedron in hyperbolic $3$-space.
\begin{corollary} \label{C:block decomp volume}
Suppose $L \subset N$ is a link in a compact $3$-manifold $N$ with a block decomposition of $N$ relative to $L$ that consists of $n_T$ blocks of type $\B^T$ and $n_S$ blocks of type $\B^S$. 
Then, $N - L$ admits a complete hyperbolic metric with totally geodesic thrice-punctured sphere boundary components and volume $V_\oct(n_T+2n_S)$.
\end{corollary}

\subsection{A block decomposition from a path in $\mathcal P(S)$} \label{S:blocks from paths} Now assume $f$ is an irreducible, end-periodic homeomorphism of $S$. 
Choose a pants decomposition $P$ in any $f$--invariant component $\Omega\subset \mathcal P_f(S)$ and a path $P = P_0,P_1,\ldots,P_n = f^{-1}(P)$ in $\Omega$. We fix representatives of the pants decompositions and of the isotopy class of $f$ (all denoted by the same names) so that $f(P_n) = P_0$. 
We also assume that our representatives are such that $P_k$ and $P_{k+1}$ agree on the complement of a complexity $1$ (open) subsurface $Z_k$. We further assume that on $Z_k$, $P_k$ and $P_{k+1}$ meet in simple closed curves $\alpha_k^-,\alpha_k^+$, respectively, that intersect each other minimally: once or twice depending on the homeomorphism type of $Z_k$. We let $n_T$ and $n_S$ denote the number of surfaces $Z_k$ homeomorphic to $\Sigma_{1,1}$ and $\Sigma_{0,4}$, respectively, with $n_S+n_T = n$. 
As usual, $\overline Z_k\subset S$ is the compactified subsurface (which may only be embedded on its interior; see \Cref{S:surfaces-and-mcg}).

For each $0 \leq k \leq n$, we also view $P_k,Z_k \subset S$ as a subset of the slice
\[ P_k,Z_k \subset S \times \left\{\tfrac{k}{n} \right\} \subset S \times [0,1] \]
via the inclusion in the first factor.
Gluing $S \times \{1\}$ to $S \times \{0\}$ via $f$, we produce $M_f$. Each of $P_k$ and $Z_k$ are well-defined subsets of $M_f \subset \overline M_f$, where we identify $P_n$ with $P_0$ via the gluing map $f$, and similarly identify $Z_n$ with $Z_0$. Define $L = \cup_k P_k \subset M_f$.

For each $0 \leq k < n$, set
\[ W_k = Z_k \times \left( \tfrac{k}n, \tfrac{k+1}n \right), \quad  \overline W_k = \overline Z_k \times \left[ \tfrac{k}n, \tfrac{k+1}n \right], \quad W = \bigsqcup_{k=0}^{n-1} W_k, \quad \mbox{ and } \quad \overline W = \bigcup_{k=0}^{n-1} \overline W_k.\] 
The set $W$ and each $W_k$ is naturally an open subset of $M_f$, and by another abuse of notation we also view $\overline W$ and each $\overline W_k$ as subsets of $M_f$ (again, only the interior is actually embedded).

Each $\overline W_k$ admits a natural map to a block $\phi_k \colon \overline W_k \to \B_k$ (by definition of a block). Here, $\alpha_k^-$, $\alpha_k^+$, and $\partial \overline Z_k \times \left[ \tfrac{k}n,\tfrac{k+1}n \right]$ map to $\partial_1^- \B_k$, $\partial_1^+ \B_k$, and $\partial_1^v \B_k$, respectively. 
We write $\partial_2\overline W_k$ and $\partial_2^\pm \overline W_k$ for the unions of $3$-punctured spheres mapping to $\partial_2 \B_k$ and $\partial_2^\pm \B_k$, respectively. We identify $\partial_2 \overline W_k$, and $\partial_2^\pm \overline W_k$ with their images in $M_f$ when convenient. We also write $\partial^v \overline W_k = \partial \overline Z_k \times \left[ \tfrac{k}n,\tfrac{k+1}n \right]$ to denote {\em vertical boundary annuli} which project to $\partial_1^v \B_k$.

To state the main proposition of this section, recall from \Cref{C:complexity of S_pm} that $\xi(\partial \overline M_f) = \xi(S_+) +\xi(S_-) = 3|\Phi^*(f)|$.

\begin{proposition} \label{P:block decomposition} 
Given a sequence of elementary moves $P = P_0,\ldots,P_n = f^{-1}(P) \subset \Omega\subset \mathcal P_f(S)$, together with the notation and assumptions above, there is a manifold $\widehat M_f$ and a quotient map $h \colon \overline M_f \to \widehat M_f$ with the following properties:
\begin{enumerate}
    \item \label{I:homeo on bdy} $h$ restricts to a homeomorphism from $\partial \overline M_f \to \partial \widehat M_f$;
    \item \label{I:htpc to homeo} $h$ is homotopic (rel boundary) to a homeomorphism $H \colon \overline M_f \to \widehat M_f$;
    \item \label{I:link compts} $\hat L = h(L)$ is a link with $n + \tfrac32|\Phi^*(f)|$ components;
    \item\label{I:block decomp} $h|_{\overline W_k} \colon \overline W_k \to \widehat M_f$ is the composition of $\phi_k$ and a block map $\B_k \subset \widehat M_f$ relative to $\hat L$. Moreover, $\{\B_k\}_{k=0}^{n-1}$ is a block decomposition of $(\widehat M_f,\hat L)$;
    \item\label{I:boundary pants decomp} the pants decomposition $P_\Omega = h^{-1}(\hat L \cap \partial \widehat M_f) = H^{-1}(\hat L \cap \partial \widehat M_f) \subset \partial \overline M_f$ depends only on the component $\Omega \subset \mathcal P(S)$ containing $P$, up to isotopy; and 
    \item\label{I:acylindrical pared} $(\overline M_f,P_\Omega)$ is a pared acylindrical manifold.
\end{enumerate}
\end{proposition}

The pants decomposition $P_\Omega$ of $\partial \overline M_f$ is the one alluded to in the introduction, and so by \Cref{P:block decomposition} (\ref{I:acylindrical pared}) together with \Cref{Theorem:Thurston.Geometrization}, $\overline M_f - P_\Omega$ admits a complete hyperbolic structure, $\sigma(f,\Omega)$, with totally geodesic, thrice-punctured sphere boundary components. So, by \Cref{T:Storm},
\[ \iVol(\overline M_f-P_\Omega) = \Vol(\overline M_f - P_\Omega) = \Vol(\overline M_f-P_\Omega,\sigma(f,\Omega)).\]

\subsection{Applications of \Cref{P:block decomposition}} We will prove the proposition in the next section, but first we use it to deduce the first two theorems from the introduction and their corollaries. 

\medskip

\noindent{\bf \Cref{T:component volume}.} {\em \Omegaupperboundtheorem}
\begin{proof} Observe that for any $N >0$, $\overline M_{f^N}$ is an $N$--fold cover of $\overline M_f$. We will apply \Cref{P:block decomposition} to $\overline M_{f^N}$ for an appropriate choice of $N$, which we now specify. Fix any $f$--invariant component $\Omega \subset \mathcal P_f(S)$ and let $\epsilon > 0$. Let $N >0$ and $P = P_0,\ldots,P_n = f^{-N}(P) \subset \Omega$ be a sequence of elementary moves so that 
\[ \left|\tau(f,\Omega) - \frac{n_T+2n_S}N \right| < \epsilon,\]
where $n = n_S+n_T$, as above. Recall that the path metric we defined on (the components of) $\mathcal P(S)$ has edges of length $1$ corresponding to an elementary move that occurs on a one-holed torus has length $1$, and edges of length $2$ corresponding to an elementary move that occurs on a four-holed sphere. 

For the path $P_0,\ldots, P_n$ above, let $P_\Omega^N \subset \partial \overline M_{f^N}$ be the pants decomposition from part (\ref{I:boundary pants decomp}). 
Observe that by replacing the sequence $P_0,\ldots,P_n$ with a concatenation of $P = P_0 = P_0',P_1',\ldots,P_k'= f^{-1}(P)$ and its image under $f^{-1},f^{-2},\ldots,f^{-N+1}$, we get the same pants decomposition $P_\Omega^N$ on $\partial \overline M_{f^N}$, by \Cref{P:block decomposition} (\ref{I:boundary pants decomp}), but this necessarily lifts the pants decomposition $P_\Omega$ of $\partial \overline M_f$. 
Since $\Vol(\overline M_{f^N} - P_{\Omega}^N)$ is the volume of the metric with totally geodesic boundary on $\overline M_{f^N} - P_\Omega^N$, it follows that $\Vol(\overline M_{f^N} - P_{\Omega}^N) = N\Vol(\overline M_f - P_\Omega)$. Therefore, if we can prove
\begin{equation}\label{E:Nth power bound} \Vol(\overline M_{f^N} - P_{\Omega}^N) \leq V_\oct(n_T+2n_S)
\end{equation}
we will have 
\[ \Vol(\overline M_f- P_\Omega)= \frac{\Vol(\overline M_{f^N} - P^N_\Omega)}{N} \leq V_\oct \frac{n_T+2n_S}{N} \leq V_\oct(\tau(f,\Omega)+\epsilon),\]
and letting $\epsilon \to 0$ will then complete the proof.

To prove the inequality (\ref{E:Nth power bound}), we apply \Cref{P:block decomposition} \eqref{I:block decomp} to the path $P_0,\ldots,P_n$ and $\overline M_{f^N}$, producing a block decomposition of $\widehat M_{f^N}$ relative to $\hat L$. By \Cref{C:block decomp volume}, the drilled manifold $\widehat M_{f^N} - \hat L$ 
also admits a complete hyperbolic metric $\sigma$ with totally geodesic thrice-punctured sphere boundary components having volume $\Vol(\widehat M_{f^N} - \hat L,\sigma)=V_\oct(n_T+2n_S)$.
We now observe that $\overline M_{f^N} - P_\Omega^N$ is homeomorphic to a Dehn filling on $\widehat M_{f^N} - \widehat L$ of some of the components of $\hat L$. By \Cref{P:block decomposition} \eqref{I:acylindrical pared} and \Cref{Theorem:Thurston.Geometrization}, there is a convex hyperbolic metric $\sigma(f^N, \Omega)$ on $\overline M_{f^N} - P_\Omega^N$ with totally geodesic boundary. 
Appealing to \Cref{T:filling volume}, we see
\[ \Vol(\overline M_{f^N}-P_\Omega^N) \leq \Vol(\widehat M_{f^N}-\hat L) = V_\oct(n_T+2n_S).\]
This proves (\ref{E:Nth power bound}) and so completes the proof of the theorem.
\end{proof}

From \Cref{T:component volume} we obtain the first theorem.

\medskip

\noindent {\bf \Cref{T:upper}.} {\em \upperboundtheorem}

\smallskip

\begin{proof}
Given $\epsilon >0$, suppose $\Omega \subset \mathcal P_f(S)$ is a component such that
\[ \left| \tau(f) - \tau(f,\Omega) \right| <\epsilon\]
and let $N >0$ be such that $f^N$ preserves $\Omega$. Then, 
\[  \frac{\tau(f^N,\Omega)}{N} = \tau(f,\Omega)  \leq \tau(f) + \epsilon. \]
According to \Cref{T:component volume} we have
\[ \Vol(\overline M_{f^N} - P_\Omega) \leq V_\oct \tau(f^N,\Omega), \]
where $P_\Omega$ is the pants decomposition of $\partial \overline M_{f^N}$ from \Cref{P:block decomposition} \eqref{I:boundary pants decomp}.

On the other hand, by \Cref{T:Storm} and \Cref{T:Storm3},
\[ \iVol(\overline M_{f^N}) \leq \iVol(\overline M_{f^N}-P_{\Omega}) = \Vol(\overline M_{f^N}-P_{\Omega}),  \]
and so by \Cref{T:Storm 2} and the inequalities above we have
\[
\iVol(\overline M_f) = \frac{\iVol(\overline M_{f^N})}{N}  \leq   \frac{V_\oct \tau(f^N,\Omega)}N  \leq V_\oct (\tau(f)+\epsilon).\]
Letting $\epsilon \to 0$ completes the proof.
\end{proof}



Next, we prove the lower bound on translation length stated in the introduction. Recall that $V_\tet$ is the volume of a regular ideal tetrahedron in $\mathbb H^3$.

\medskip

\noindent
{\bf \Cref{C:lower bound translation}}
{\em \lowertranslation}

\begin{proof}
First observe that $\tau(f) = \inf \tau(f,\Omega)$, where the infimum is taken over all components $\Omega \subset \mathcal P(S)$, 
so it suffices to prove that for any component $\Omega$, $\tau(f,\Omega)$ is bounded below by the quantity on the right in the corollary. We need only consider components $\Omega \subset \mathcal P_f(S)$, so we fix such a component. Let $N > 0$ be such that $f^N(\Omega) = \Omega$, and let $P_\Omega \subset \partial \overline M_{f^N}$ be the pants decomposition given by \Cref{P:block decomposition} \eqref{I:boundary pants decomp}. By \Cref{P:block decomposition} \eqref{I:acylindrical pared} together with \Cref{Theorem:Thurston.Geometrization}, $\overline M_{f^N} - P_\Omega$ admits a convex hyperbolic structure $\sigma_{min}$ with totally geodesic boundary so that
\[ \iVol(\overline M_{f^N} - P_\Omega) = \Vol(\overline M_{f^N}-P_\Omega).\]
The number of pants curves in $P_\Omega$ is $\xi(\partial \overline M_{f^N}) = N \xi(\partial \overline M_f)$, and so the double of $\overline M_{f^N}-P_\Omega$ admits a complete hyperbolic structure, obtained by doubling $\sigma_{min}$, having \[ \xi(\partial \overline M_{f^N}) = N \xi(\partial \overline M_f)\] 
cusps. According to Adams \cite{Adams-Ncusp}, the volume of this hyperbolic structure on the double of $\overline M_{f^N}$ is at least $V_\tet$ times the number of cusps, and thus
\[ \tau(f,\Omega) = \frac{\tau(f^N,\Omega)}N  \geq \frac{\iVol(\overline M_{f^N} - P_\Omega)}{N \, V_\oct} \geq \frac{V_\tet \, N \xi(\partial \overline M_f)}{2N \, V_\oct} = \frac{V_\tet \, \xi(\partial \overline M_f)}{2 V_\oct}. \]
\end{proof}

The following provides an alternate description of the lower bound in \Cref{C:lower bound translation} in terms of the coarse end behavior $\Phi^*(f)$ of $f$ as discussed in \Cref{sub:endbehavior}.
\begin{corollary} \label{C:intrinsic lower translation}
Given an irreducible, end-periodic homeomorphism $f$ with coarse end behavior $\Phi^*(f)$, we have
\[ \tau(f) \geq \frac{3 V_\tet |\Phi^*(f)|}{2 V_\oct}. \]
\end{corollary}
\begin{proof}
Since $\partial \overline M_f = S_+ \cup S_-$, the result is immediate from \Cref{C:complexity of S_pm} and \Cref{C:lower bound translation}.
\end{proof}

\section{Proof of \Cref{P:block decomposition}}\label{Section: Proof of Pro4.3}

The bulk of the proof of \Cref{P:block decomposition} involves defining the quotient map $h: \overline{M}_f\to \widehat M_f$, analyzing the structure of the sets of the associated decomposition, and then appealing to \Cref{thm:armentrout}. Since the hypotheses of that theorem require a closed manifold, we will double $\overline M_f$ over $\partial \overline M_f$.

We now define a quotient map $h: \overline M_f \to \widehat M_f$ in terms of a decomposition of $\overline M_f$. The elements in the decomposition are either:
\begin{enumerate}
    \item singletons in $W$; or
    \item maximal components of a flow line of the local flow $(\bar \psi_s)$ (defined at the end of \Cref{section:Compactification}) intersected with $\overline M_f  -  W$.
\end{enumerate}
We note that there may, a priori, be a flow line that entirely misses $W$. We will see that this does not happen, and as \Cref{C:flowlines hit W} below shows, every component of intersection of a flow line with $\overline M_f  -  W$ is a compact arc. We first verify this for points in $L$ and then for points outside. Along the way, we verify that the image of $L$ under $h$ is the appropriate union of circles from
\Cref{P:block decomposition} \eqref{I:link compts}.\\

For each component $\alpha \subset L$, we say that $\alpha$ {\em flips forward} if $\psi_{1/n}(\alpha) \not \subset L$. Note that $\alpha$ flips forward if and only if it is equal to $\alpha_k^- \subset Z_k \subset S \times \left\{\tfrac{k}n\right\}$ for some $0 \leq k < n$. Thus, there are precisely $n$ components of $L$ that flip forward.
We say that $\alpha$ {\em flips forward after $k \geq 0$ steps} if either $k =0$ and $\alpha$ flips forward, or if $k > 0$ and $\psi_{j/n}(\alpha)$ does not flip forward for any $0 \leq j < k$, but $\psi_{k/n}(\alpha)$ flips forward.
We similarly say that a component $\alpha \subset L$ {\em flips backward} if $\psi_{-1/n}(\alpha) \not \subset L$ and note that there are exactly $n$ components of $L$ that flip backward. {\em Flipping backward after $k\geq 0$ steps} is defined in the analogous manner.
Briefly, we will also say that $\alpha$ {\em eventually flips forward/backward}, if it does so after $k$ steps, for some $k \geq 0$.

For any component $\alpha \subset L$ and (possibly infinite) interval $I \subseteq \mathbb R$ containing $0$, we let
\[ A(\alpha,I) = \bigcup_{t \in I} \psi_t(\alpha).\] 

\begin{lemma} \label{L:eventually flipped}
Every component $\alpha \subset L$ eventually flips forward or backward.
\end{lemma}

\begin{proof}
We suppose $\alpha$ does not eventually flip forward or backward and arrive at a contradiction. Indeed, under this assumption we have $A(\alpha,\mathbb R)$
is a bi-infinite annulus that meets $S = S \times \{0\}$ in an $f$--invariant subset of $P_0$. This means that there is either a periodic simple closed curve in $P_0$ or else the $f$--orbit of a curve in $P_0$ contains components in both $U_+$ and $U_-$. Either conclusion contradicts the irreducibility of $f$.
\end{proof}

For the next lemma, we let $n = n_S +n_T$, as in \Cref{S:blocks from paths}. 
\begin{lemma} \label{L:annuli through L}
The intersection of $\overline M_f  -  W$ with the union of $(\bar \psi_s)$ flow lines through $L$ consists of $n+\tfrac32|\Phi^*(f)|$ (possibly degenerate) compact annuli. There are $3|\Phi^*(f)|$ boundary components of these annuli on $\partial \overline M_f$ forming a pants decomposition of $\partial \overline M_f$, and $n$ boundary components which are curves of the form $\alpha_k^\pm$ on some $\overline W_k$.
\end{lemma}
By a ``degenerate annulus'', we mean an annulus $S^1 \times [a,b]$, where $a= b$, which is therefore just a circle. In this case, the two boundary components of the annulus consist of the circle, and this has the form $\alpha_k^+ =\alpha_{k+1}^-$.
\begin{proof} Fix a component $\alpha \subset L$.
By \Cref{L:eventually flipped}, $\alpha$ either flips forward after $k \geq 0$ steps, backward after $k'\geq 0$ steps, or both (for some $k$ and $k'$). If it flips forward after $k$ steps and does not eventually flip backward, then $A(\alpha,(-\infty,\tfrac{k}n])$ is a half-open annulus with its boundary component given by $\alpha_j^- = \psi_{k/n}(\alpha) \subset \overline W_j$, for some $j$. The open end limits to a curve in $S_- \subset \partial \overline M_f$.  

Similarly, if $\alpha$ flips backward after $k' \geq 0$ steps, but does not eventually flip forward, then $A(\alpha,[\tfrac{-k'}n,\infty))$ is a half-open annulus with boundary component $\alpha_j^+ = \psi_{-k'/n}(\alpha) \subset \overline W_j$, for some $j$. 
The open end limits to a curve in $S_+$.  

If $\alpha$ both flips forward after $k \geq 0$ steps and backward after $k' \geq 0$ steps, then $A(\alpha,[\tfrac{-k'}n,\tfrac{k}n])$
is a compact annulus with one boundary component $\alpha_j^- = \psi_{k/n}(\alpha) \subset \overline W_j$ and the other $\alpha_{j'}^+ = \psi_{-k'/n}(\alpha) \subset \overline W_{j'}$, for some $j,j'$.

The annuli of the first two types compactify in $\overline M_f$ to have boundary components on $\partial \overline M_f$. Since $(\bar \psi_s)$ is the local flow extending the reparameterization of $(\psi_s)$, the union of all resulting compact annuli (from all three cases) are exactly the union of the flow lines through $L$ intersected with $\overline M_f  -  W$.

We can count the number of annuli by counting the number of boundary components of the annuli and dividing by $2$. There are exactly two boundary components on each $\overline W_k$, and this accounts for $2n$ boundary components. The remaining boundary components form a multicurve in $\partial \overline M_f$. If this is a pants decomposition, then we have $\xi(\partial \overline M_f) = 3 |\Phi^*(f)|$ such boundary components, and hence the number of annuli is $n+\tfrac32|\Phi^*(f)|$, as required.

To see that the boundary components on $S_+$ form a pants decomposition, we argue as follows (the case of $S_-$ is similar). Choose a good nesting neighborhood $U_+$ of the attracting end with the property that $P_j \cap U_+ = P_k \cap U_+$, for all $j,k$, and so that $f(U_+) \subset U_+$. Note that this means that no component of $P_k \cap U_+ \subset S \times \{\tfrac{k}n\}$ eventually flips forward.
Every point of $U_+$ flows forward to $S_+$ (since $U_+ \subset \mathcal U_+$) and this flow defines a surjective local diffeomorphism $U_+ \to S_+$; more precisely, this is one end of the infinite cyclic covering $\mathcal U_+ \to S_+$. Because $P_k$ defines a pants decomposition of the end $U_+$, this projects by the flow to a pants decomposition of $S_+$, completing the proof.
\end{proof}

Let $A_1,\ldots,A_m$ be the annuli from \Cref{L:annuli through L}, with $m = n + \tfrac32|\Phi^*(f)|$, and let $\mathcal A$ be the union of these annuli. 
Note that the product structure on $A_j \cong S^1 \times I$ obtained from the flow has the property that $\{\ast\} \times I$ is an element of the decomposition defining the quotient $\widehat M_f$. 
Furthermore, every component of $L$ has the form $S^1 \times \{\ast\} \subset A_j$, for some $j$. Therefore, we have the following.
\begin{corollary} \label{C:image of L}
We have $\hat L = h(L) = h(\mathcal A)$ and the restriction $h|_{\mathcal A} \colon \mathcal A \to \widehat M_f$ factors as
\[ \mathcal A \to \bigsqcup_{j=1}^{m} S^1 \to \widehat M_f,\]
where the first map projects out the interval directions of the annuli, and the second map is injective.
\end{corollary}
Thus far we have described the restriction of $h$ to both $W$ (on which $h$ is injective) and $\mathcal A$. The next lemma describes the structure of the rest of $\overline M_f$.

\begin{lemma}\label{L:3-punc} 
Every component $X$ of $\overline M_f  -  (\mathcal A \cup W)$ is homeomorphic to a product of a $3$-punctured sphere with a (possibly degenerate) interval $I \subset [-\infty,\infty]$ with $0 \in \partial I$. The $3$-punctured sphere $Y$ is a component of $\bigcup \partial_2 \overline W_j$, and the homeomorphism $Y \times I \to X$ is given by $(x,t) \mapsto \psi_t(x)$. Furthermore, $Y \times \partial I$ consists of two (possibly equal) $3$-punctured sphere components of
\[\bigcup_{j=0}^{n-1} \partial_2 \overline W_j \cup (\partial \overline M_f  -  \mathcal A).\]
\end{lemma}
Note that in the statement of this lemma, we may have to make sense of $\psi_\infty(x)$ for some $x$, which is defined to be the point $\displaystyle{\lim_{s \to \infty}\psi_s(x)}$ in $S_+ \subset \partial \overline M_f$, when this limit exists.  We make a similar convention for $\psi_{-\infty}(x)$.  It might seem more natural to use the local flow, $(\bar \psi_s)$, but the parameterization of $(\bar \psi_s)$ is not well-behaved with respect to the $S$--fibers. In particular, we cannot produce the nice homeomorphism described in the statement.
\begin{proof} Consider any $3$-punctured sphere component $Y \subset \partial_2^+ \overline W_j$ for any $j$. This is the interior of a (compact) pair of pants we denote $\overline Y \subset \overline Z_j \times \{\tfrac{j+1}n\}$, and the boundary consists of three (not necessarily distinct) pants curves $\alpha_1,\alpha_2,\alpha_3$ in $P_{j+1}$ (with $j+1$ taken modulo $n$).
If one of these pants curves flips forward after $k$ steps, then taking $I = [0,\tfrac{k}n]$, the map $Y \times I \to M_f$, given by the formula in the lemma maps to $\overline M_f  -  (\mathcal A \cup W)$, with $Y \times \{\tfrac{k}n\}$ mapping to a $3$-punctured sphere component of $\partial_2^- W_i$, for some $i$.  If none of these curves eventually flip forward, then we take $I = [0,\infty]$ and the formula provides the homeomorphism, this time sending $Y \times \{\infty\}$ to a $3$-punctured sphere component of $S_+  -  \mathcal A$.

We can argue similarly for a $3$-punctured sphere component $Y \subset \partial_2^- \overline W_j$ for any $j$, this time producing an interval $I$ of the form $[-\tfrac{k}n,0]$ if one of the boundary pants curves flips backward after $k$ steps or $[-\infty,0]$ if it does not eventually flip backward.

Each of the products produced above is a component of $\overline M_f  -  (\mathcal A \cup W)$, and thus it suffices to show that this accounts for all components of this complement. To see this, we simply note that any point $x$ in such a component $X$ flows forward (remaining inside $X$) for time less than $\tfrac1n$ until it reaches some fiber $S \times \{\tfrac{j}n\}$. It meets this in a $3$-punctured sphere component of the complement of $P_j$, and (each of) the boundary pants curves must eventually flip forward or backward. The first time one of these flips forward or backward specifies a length of time we need to flow $x$ (inside $X$) until it hits $\bigcup \partial_2 \overline W_i$. Therefore, $X$ is one of the components already discussed, completing the proof.
\end{proof} 

\begin{corollary} \label{C:flowlines hit W} 
Every flow line of $(\bar \psi_s)$ has nontrivial intersection with $W$. Consequently, every set in the decomposition is either a point or compact arc of a flow line with at most one endpoint on $\partial \overline M_f$.
\end{corollary}

We are now ready for the proof of \Cref{P:block decomposition}. Most of the proof involves analyzing the decomposition defining the quotient $h \colon \overline M_f \to \widehat M_f$ using the description of the decomposition elements from the lemmas/corollary above. From this analysis, we will show that $\widehat M_f$ is a $3$-manifold with boundary and that $h$ restricts to a homeomorphism $\partial \overline M_f \to \partial \widehat M_f$. Proving that $h$ is homotopic to a homeomorphism involves analyzing the associated quotient of the double $D\overline M_f$ and applying \Cref{thm:armentrout}.  The remaining items follow fairly quickly after this is done. 

\begin{proof}[Proof of \Cref{P:block decomposition}.]
We begin by proving that $\widehat M_f$ is a $3$-manifold with boundary. The first step is to find a neighborhood of each point in $\widehat M_f$ homeomorphic to a ball or half-ball in $\mathbb R^3$. For this, observe that the restriction of $h$ to $W$ is a homeomorphism onto its image, since $W$ is an open set on which $h$ is injective. So, all points in $h(W)$ have neighborhoods homeomorphic to a ball in $\mathbb R^3$. Next, observe that by \Cref{L:3-punc}, each component $X$ of $\overline M_f  -  (\mathcal A \cup W)$ is homeomorphic to a product $X \cong Y \times I$, with $\{y\} \times I = \bigcup_{s \in I} \psi_s(y)$ an arc of a flow line.  The restriction of $h$ to $X$ factors as the composition of the projection $X \to Y$ and an inclusion $Y \to \widehat M_f$,
\[ X \to Y \hookrightarrow \widehat M_f.\] 
If the interval $I$ does not contain either $\infty$ or $-\infty$, then any point in the image of $Y$ in $\widehat M_f$ has a neighborhood homeomorphic to a ball in $\mathbb R^3$, constructed from half-balls in $\overline W_j$ and $\overline W_i$ where the bottom and top $3$-punctured spheres $Y \times \partial I$ of $X$ are components of $\partial_2^+ \overline W_j$ and $\partial_2^- \overline W_i$, respectively. If $I$ contains $\infty$, then it cannot contain $-\infty$, and $Y \times \{0\} \subset X$ is a component of some $\partial_2^+ \overline W_j$ while $Y \times \{\infty\}$ is a component of $S_+  -  \mathcal A$, and in this case every point in the image of $Y$ in $\widehat M_f$ has a neighborhood homeomorphic to a half-ball in $\mathbb R^3$, coming from a half-ball in $\overline W_j$. If $I$ contains $-\infty$, we similarly find a half-ball neighborhood.

The only remaining points of $\widehat M_f$ are those in the image of $\mathcal A$. For these, we consider each component annulus $A_j \subset \mathcal A$ and build neighborhoods from all the $\overline W_k$ it meets. There are different cases depending on whether the annulus $A_j$ meets $\partial \overline M_f$ or not.

First, suppose $A_j \cap \partial \overline M_f = \emptyset$. See \Cref{F:collapsing near annuli} for an illustration of the construction of the neighborhood of a point in the image of $A_j$.  Then we can write
\[ A_j = \bigcup_{s \in [0,t]} \psi_s(\alpha_k^+), \]
for some some $t \geq 0$ and integer $k$ so that $\psi_t(\alpha_{k}^+)=\alpha_{\ell}^-$, for some integer $\ell$. Observe that $k$ is such that $\alpha_k^+ \subset \overline{Z}_k \times \{\tfrac{k+1}n\}$ and does not flip backward (c.f.~proof of \Cref{L:annuli through L}). The annulus may also meet various pieces $\overline W_{k_1},\ldots,\overline W_{k_j}$ along the vertical annuli $\partial^v \overline W_{k_1},\ldots,\partial^v \overline W_{k_j}$.
Any point in the image of $A_j$ is given by the decomposition element of the form $\{\ast\} \times I \subset S^1 \times I \cong A_j$, and a neighborhood of this point is obtained from neighborhoods in $\overline W_k$ and $\overline W_\ell$ on the ``top and bottom", together with neighborhoods ``along the sides" from the images of $\overline W_{k_1},\ldots,\overline W_{k_j}$.

When the annulus $A_j$ meets $S_+$ or $S_-$, we have a similar picture, except in the former case the top reaches $S_+$ instead of some $\overline W_\ell$, and in the latter, the bottom reaches $S_-$ instead of ``starting" on some $\overline W_k$. In this setting, we get half-ball neighborhoods (since the top/bottom are missing). 

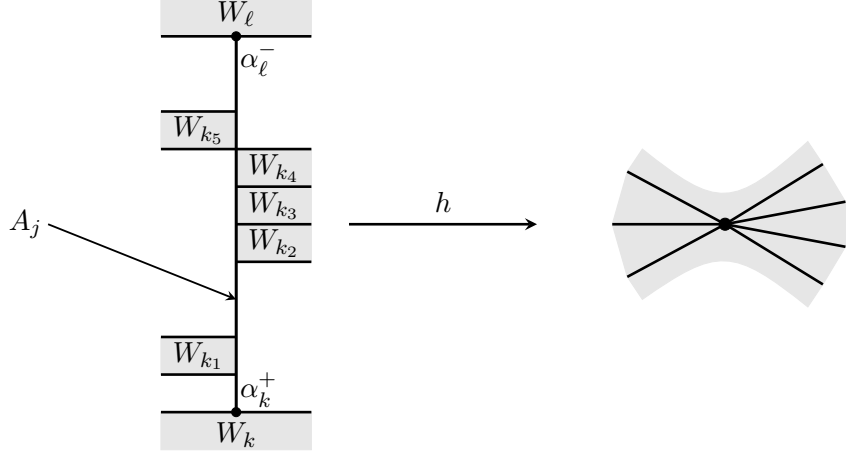
\begin{figure}
\centering
\begin{tikzpicture}
\filldraw[opacity=.1] (-1,5)--(1,5)--(1,5.5) -- (-1,5.5) -- (-1,5);
\node at (0,5.3) {$W_\ell$};
\draw[line width=1pt] (-1,5)--(1,5);
\filldraw[opacity=.1] (-1,0)--(1,0)--(1,-.5) -- (-1,-.5) -- (-1,0);\draw[line width=1pt] (0,0) -- (0,5);
\node at (0,-.3) {$W_k$};
\draw[line width=1pt] (-1,0) -- (1,0);
\filldraw[opacity=.1] (-1,1) -- (0,1) -- (0,.5) -- (-1,.5) -- (-1,1);
\node at (-.5,.75) {$W_{k_1}$};
\draw[line width=1pt] (-1,1) -- (0,1);
\draw[line width=1pt] (-1,.5) -- (0,.5);
\filldraw[opacity=.1] (-1,4) -- (0,4) -- (0,3.5) -- (-1,3.5) -- (-1,4);
\node at (-.5,3.75) {$W_{k_5}$};
\draw[line width=1pt] (-1,4) -- (0,4);
\draw[line width=1pt] (-1,3.5) -- (0,3.5);
\filldraw[opacity=.1] (1,2) -- (0,2) -- (0,3.5) -- (1,3.5) -- (1,2);
\node at (.5,2.25) {$W_{k_2}$};
\node at (.5,2.75) {$W_{k_3}$};
\node at (.5,3.25) {$W_{k_4}$};
\draw[line width=1pt] (0,3.5) -- (1,3.5);
\draw[line width=1pt] (0,3) -- (1,3);
\draw[line width=1pt] (0,2.5) -- (1,2.5);
\draw[line width=1pt] (0,2) -- (1,2);
\draw[line width=1pt] (0,0) -- (0,5);
\draw[fill=black] (0,0) circle (.06cm);
\node at (.3,.3) {$\alpha_k^+$};
\draw[fill=black] (0,5) circle (.06cm);
\node at (.3,4.7) {$\alpha_\ell^-$};
\draw[line width=1pt,->,>=stealth] (1.5,2.5) -- (4,2.5);
\node at (2.75,2.8) {$h$};
\filldraw[opacity=.1] (5.2,3.2)-- (5,2.5) -- (5.2,1.8) -- (6.5,2.5) -- (5.2,3.2);
\filldraw[opacity=.1] (7.8,1.7) -- (8.1,2.2) -- (8.1,2.8) -- (7.8,3.3) -- (6.5,2.5) -- (7.8,1.7); 
\filldraw[opacity=.1] (5.2,1.8) -- (5.4,1.5)  .. controls (6.4,2.2) and (6.6,2.2) .. (7.6,1.4) -- (7.8,1.7) -- (6.5,2.5) -- (5.2,1.8);
\filldraw[opacity=.1] (5.2,3.2) -- (5.4,3.5)  .. controls (6.4,2.7) and (6.6,2.7) .. (7.6,3.6) -- (7.8,3.3) -- (6.5,2.5) -- (5.2,3.2);
\draw[fill=black] (6.5,2.5) circle (.08cm);
\draw[line width=1pt] (5.2,3.2) -- (6.5,2.5) -- (7.8,3.3);
\draw[line width=1pt] (5,2.5) -- (6.5,2.5) -- (8.1,2.8);
\draw[line width=1pt] (5.2,1.8) -- (6.5,2.5) -- (8.1,2.2);
\draw[line width=1pt] (6.5,2.5) -- (7.8,1.7);
\draw[line width=.7pt,->,>=stealth] (-2.5,2.5) -- (0,1.5);
\node at (-2.8,2.5) {$A_j$};
\end{tikzpicture}
\caption{The left-hand side of the figure shows a slice of a neighborhood of some annulus $A_j \cong S^1 \times I$ (specifically, the slice $\{\ast\} \times I$). The figure illustrates the situation that the boundary components $\alpha_k^+$ and $\alpha_\ell^-$ of $A_j$ are on a pair of regions $\overline W_k$ and $\overline W_\ell$. The annulus meets some (possibly zero) number of other regions $\overline W_{k_1},\ldots,\overline W_{k_j}$ along the ``vertical" boundary annuli. The right-hand side of the figure shows the result of collapsing. Neighborhoods in the $\overline W_i$ pieces are glued together in the quotient to produce a neighborhood of points in $h(A_j)$.} \label{F:collapsing near annuli}
\end{figure}

We have found neighborhoods of every point in $\widehat M_f$ homeomorphic to either an open ball or half-ball in $\mathbb R^3$. From the construction of these neighborhoods, it is straightforward to show that any two distinct points in $\widehat M_f$ have disjoint neighborhoods, and that there is a countable basis for the topology. Since $\widehat M_f$ is the image of the compact space $\overline M_f$, it is also compact. Thus, $\widehat M_f$ is a compact $3$-manifold with boundary.  

Since the points with neighborhoods homeomorphic to half-balls are precisely those in the image of $\partial \overline M_f$, we have $\partial \widehat M_f = h(\partial \overline M_f)$.
To see that the restriction of $h$ to $\partial \overline M_f$ is injective, observe that by \Cref{C:flowlines hit W}, any element of the decomposition meets $\partial \overline M_f$ in at most one point. This proves item \eqref{I:homeo on bdy}. According to \Cref{C:image of L}, $\hat L = h(L)= h(\mathcal A)$ is the continuous image of a disjoint union of $n+\tfrac32 |\Phi^*(f)|$ circles. Since $\widehat M_f$ is Hausdorff, the injection is an embedding, proving item \eqref{I:link compts}.
Moreover, since $\mathcal A$ meets $\partial \overline M_f$ in a pants decomposition, $\hat L$ meets $\partial \widehat M_f$ in a pants decomposition which we denote by $P_\Omega$, where $\Omega$ is the component of $\mathcal P(S)$ containing~$P$.

Next, we want to see that $h|_{\overline W_k} \colon \overline W_k \to \widehat M_f$ factors through a block map relative to $\hat L = h(L)$, and that the maps $h|_{\overline{W}_k}$ define a block decomposition. This follows from (a) the fact that $h|_{W_k}$ is injective on the interior of product (by definition of the decomposition), (b) $h$ collapses $\partial^v \overline W_k$ to circles contained in $\hat L$ since these annuli are sub-annuli of $\mathcal A$, (c) $h$ is injective on each pair of pants $Y \subset \partial_2 \overline W_k$ by \Cref{L:3-punc}, and (d) $h$ is injective on each curve $\alpha_k^\pm$ since this is either the bottom or top boundary component of an annulus in $\mathcal A$.

In fact, \Cref{C:image of L} and \Cref{L:3-punc} tell us that $\widehat M_f$ is built by gluing the blocks together along pairs of pants and unions of circle components of $1$--boundaries $\partial_1 \B_k$ contained in $\hat L$ (by definition of the decomposition space, the interiors of the blocks are disjoint). This gives a block decomposition relative to $\hat L$, proving item \eqref{I:block decomp}.

Next we prove item \eqref{I:htpc to homeo}, namely that $h$ is homotopic to a homeomorphism $H \colon \overline M_f \to \widehat M_f$.  For this, we consider the doubles $D\overline M_f$ and $D\widehat M_f$, and let $Dh \colon D\overline M_f \to D\widehat M_f$ be the associated quotient map, which, restricted to each copy of $\overline M_f$, maps via $h$ to one of the copies of $\widehat M_f$. Note that because $D\widehat M_f$ is a $3$-manifold, hence Hausdorff, and $Dh$ is a quotient map from the compact space $D \overline M_f$, it follows that $Dh$ is a closed map. By \Cref{P:cellular characterization}, the decomposition is upper semicontinuous. Furthermore, by \Cref{C:flowlines hit W}, each element of the decomposition defining the quotient map $h$ is an arc (or point) meeting $\partial \overline M_f$ in at most one point, hence each element of the decomposition defining the quotient map $Dh$ is also an arc or point. These are in fact arcs of flow lines of the extension of $(\bar \psi_s)$ to $D \overline M_f$ and so it is easy to find a sequence of neighborhoods, each homeomorphic to a cell, whose intersection is the arc/point. Therefore, the decomposition defining $Dh$ is cellular, hence $Dh$ is cellular, and so by \Cref{thm:armentrout}, $Dh$ is homotopic to a homeomorphism.

Now we note that the inclusion of $\overline M_f$ to $D\overline M_f$ induces an injection $\pi_1 \overline M_f$, and since $Dh_*$ is an isomorphism, it follows that $h_* \colon \pi_1 \overline M_f \to \pi_1 \widehat M_f$ is also injective. By a result of Waldhausen \cite[Theorem~6.1]{Waldhausen}, $h$ is homotopic to a homeomorphism; in fact, this is through a homotopy $h_t$ with $h_t|_{\partial \overline M_f} = h|_{\partial \overline M_f}$ for all $t$. This completes the proof of item \eqref{I:htpc to homeo}.

By construction, $\hat L \cap \partial \widehat M_f = h(P_\Omega)$ and this is equal to $H(P_\Omega)$, since $h|_{\partial \overline M_f} = H|_{\partial \overline M_f}$.
Suppose we replace the path $P= P_0,\ldots,P_n = f^{-1}(P)$ in the construction with a path between $P'$ and $f^{-1}(P')$ in the same component $\Omega \subset \mathcal P(S)$. Since $P = P'$ outside a compact set of $S$ (and the same is true for all pants in the path from $P'$ to $f^{-1}(P')$), one sees that the annuli $\mathcal A'$ constructed analogously to $\mathcal A$ meet $\partial \overline M_f$ in the same pants decomposition. Therefore, $P_\Omega$ does indeed depend only on the component $\Omega$. This proves item \eqref{I:boundary pants decomp}.

All that remains is to prove item \eqref{I:acylindrical pared}, stating that $(\overline M_f,P_\Omega)$ is a pared acylindrical manifold. All the conditions in the definition of pared manifold are immediate from the preceding results, except condition \eqref{I:annuli}, stating that an incompressible annulus $(A,\partial A) \subset (\overline M_f,N(P_\Omega))$ with boundary in a neighborhood $N(P_\Omega)$ of $P_\Omega$ must be homotopic into $N(P_\Omega)$, via a homotopy of pairs. Consider such an annulus $(A,\partial A) \subset (\overline M_f,N(P_\Omega))$.

Since the preimage $\widetilde P_\Omega$ of $P_\Omega$ in $\widetilde M_\infty \subset S \times [-\infty,\infty]$ is a pants decomposition of $\partial \widetilde M_\infty = \mathcal U_\pm \times \{\pm \infty\}$, $A$ lifts to an annulus $\widetilde M_\infty$ with boundary in $N(\widetilde P_\Omega)$. By \Cref{L:annuli analysis}, $A$ is not essential, and hence must be boundary parallel.  Since  $N(P_\Omega)$ is a neighborhood of a pants decomposition, being boundary parallel gives a homotopy of pairs taking $A$ into $N(P_\Omega)$, proving that $(\overline M_f,P_\Omega)$ is a pared manifold.

Now observe that $\overline M_f$ has incompressible boundary, and hence so does $(\overline M_f,P_\Omega)$. 
Since the paring locus $P_\Omega$ is a pants decomposition of the boundary, $(\overline M_f,P_\Omega)$ is acylindrical, as explained in the discussion after the definition of pared manifolds. This completes the proof of item \eqref{I:acylindrical pared} of the proposition, and we are done.
\end{proof}

\section{Sharpness of the Upper Bound}\label{sharpness}

In this section, we describe a construction of end-periodic homeomorphisms and provide sufficient conditions which ensure that they are either strongly irreducible or irreducible (but not strongly irreducible).  The construction allows us to prescribe the end behavior (\Cref{T:blah}), and also produce sequences of examples showing that the bound on volume given in \Cref{T:upper} is asymptotically sharp (\Cref{T:blah blah}). The proof of \Cref{T:blah} will follow immediately from \Cref{prop:irreducible-examples} together with \Cref{C:coarse conjugates to fine}, while the proof of \Cref{T:upper} will appear at the end of \Cref{S:large distance fixed f}.

\subsection{Examples of irreducible, end-periodic homeomorphisms}\label{S:example}

Recall that $S$ is a boundaryless infinite-type surface with $n$ ends, each accumulated by genus, with $2 \leq n < \infty$. We will construct a general family of examples of irreducible and strongly irreducible end-periodic homeomorphisms of $S$ with arbitrary coarse end behavior. All examples will be of the form $f = \rho h$, where $\rho$ is as in \Cref{construction} and $h$ is supported on a compact subsurface $C$.  
Whether the resulting homeomorphism $f$ is irreducible or strongly irreducible will depend on the choice of subsurface $C$.

\begin{definition}[Fully separating/partially separating]
We say a subsurface $C$ is \textit{fully separating} if $C$ separates every end of $S$, in the sense that the end space of each component of $S - C$ is a singleton. We say a subsurface $C$ is \textit{partially separating} (with respect to a given end-periodic homeomorphism) if $C$ separates the collection of attracting ends from the collection of repelling ends.  
\end{definition}

Note that a fully separating subsurface is necessarily partially separating with respect to any end-periodic homeomorphism, but the converse is not necessarily true. For any partially separating subsurface $C$, the boundary $\partial C$ decomposes as a disjoint union $\partial C = \partial_+C \cup \partial_-C$ according whether the ends cut off are attracting or repelling.  Examples and a non-example are illustrated in \Cref{fig:fully-partially-separating}.

\begin{figure}[htb!]
    \centering
    \includegraphics[width = 5in]{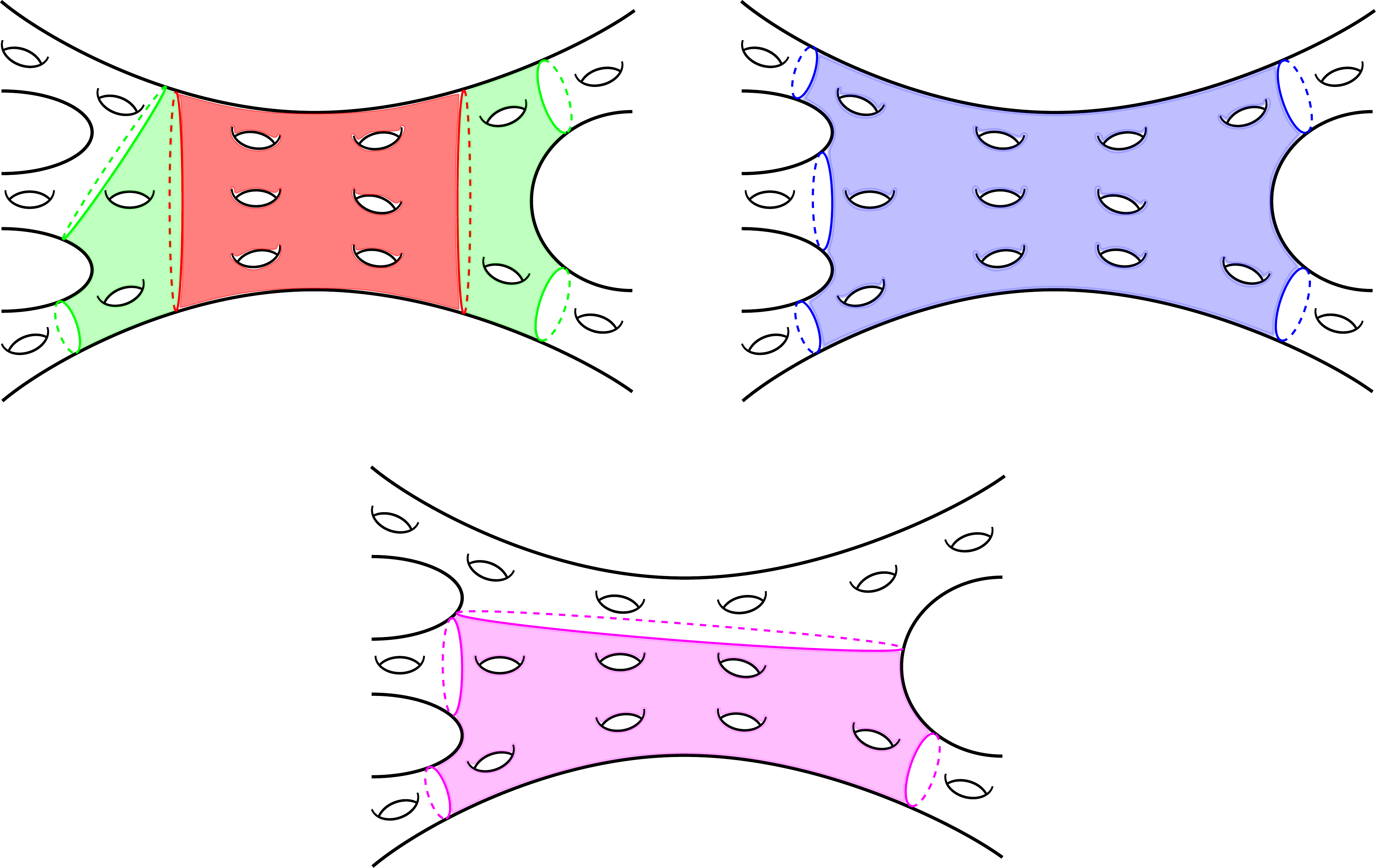}
    \caption{Consider an end-periodic homeomorphism of the $5$-ended surface shown here with $3$-repelling ends on the left and $2$-attracted ends on the right. Both the finite-type subsurface bounded by the red curves and the finite-type subsurface bounded by the green curves are partially separating (upper left). The finite-type subsurface bounded by the blue curves is fully separating (upper right). The finite-type surface bounded by the pink curves is neither fully nor partially separating (bottom).} 
    \label{fig:fully-partially-separating}
\end{figure}

\begin{convention} \label{Convention-6.1}
As in \Cref{construction}, we fix any coarse end behavior of $S$ and realize it by an end-periodic homeomorphism (isotopic to) $\rho = \Pi_1^m \rho_i$, where $\{\rho_i\}_1^m$ are pairwise-commuting handle shifts. We continue with the notation and assumptions from that construction, letting $v_1,\ldots,v_n$ be separating curves cutting off good nesting neighborhoods of the ends, and let $U_+$ and $U_-$ denote the resulting union of good nesting neighborhoods of the attracting and repelling ends, respectively. In addition, we now let $C \subset S$ be a compact subsurface which is partially separating with respect to $\rho$ and disjoint from $U_+$ and $U_-$.  We further assume, as we may, that
\begin{enumerate}
    \item \label{I:planar complement} $S  -  (U_+ \sqcup U_- \sqcup C)$ is either empty or a union of planar subsurfaces;
    \item \label{I:boundary arcs} $\partial C$ meets each of the handle strips in arcs isotopic to the arcs of intersection with the curves $v_1,\ldots,v_n$; and
    \item \label{I:genus at least 2} $C$ intersects each handle strip in a subsurface with one boundary component and genus at least $2$.
\end{enumerate} 
See \Cref{fig:handle-strip-intersection-section-6} for more explanation. 
When $C$ is fully separating, we observe that item \eqref{I:planar complement} implies that $S  -  C = U_+ \sqcup U_-$. Fix a component $\eta$ of $\partial_- C$ and fix a component $\alpha$ of $\partial_+ C$. By items \eqref{I:boundary arcs} and \eqref{I:genus at least 2}, we have $i(\rho(\eta),\rho^{-1}(\alpha)) = 0$. For this fixed $\eta$, we choose $h \in \Map(S)$ which is supported on $C$ and for which $d_{C}(\rho(\eta), h(\rho(\eta)))\geq 9$ (see \Cref{S:associated-graphs} for discussion of subsurface distance). We also let $B \subset B' \subset \mathcal{AC}(C)$ denote the balls of radius two and three about $\rho(\eta)$, respectively.  Since $i(\rho^{-1}(\alpha),\rho(\eta)) = 0$, we have $d_C(\rho^{-1}(\alpha),\rho(\eta)) = 1$ and so $\rho^{-1}(\alpha) \in B$. We will consider the end-periodic homeomorphism $f = \rho h$.  By construction, $U_+$ and $U_-$ are still good nesting neighborhoods of the attracting and repelling ends of $f$, respectively.
\end{convention}

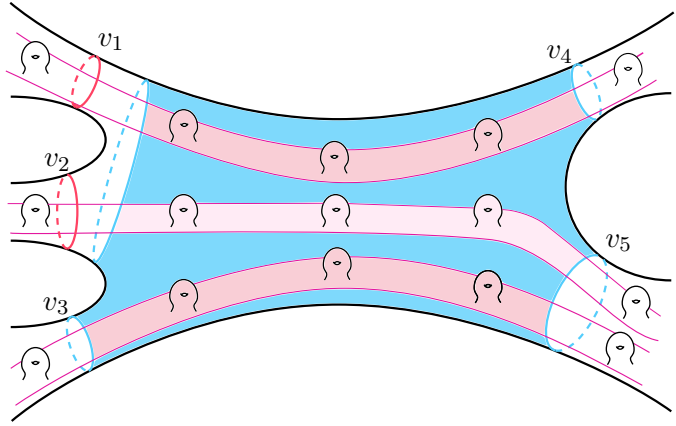
\begin{figure}[htb!]
    \centering
    \def\svgwidth{3.5in}
\begingroup%
  \makeatletter%
  \providecommand\color[2][]{%
    \errmessage{(Inkscape) Color is used for the text in Inkscape, but the package 'color.sty' is not loaded}%
    \renewcommand\color[2][]{}%
  }%
  \providecommand\transparent[1]{%
    \errmessage{(Inkscape) Transparency is used (non-zero) for the text in Inkscape, but the package 'transparent.sty' is not loaded}%
    \renewcommand\transparent[1]{}%
  }%
  \providecommand\rotatebox[2]{#2}%
  \newcommand*\fsize{\dimexpr\f@size pt\relax}%
  \newcommand*\lineheight[1]{\fontsize{\fsize}{#1\fsize}\selectfont}%
  \ifx\svgwidth\undefined%
    \setlength{\unitlength}{596.56860212bp}%
    \ifx\svgscale\undefined%
      \relax%
    \else%
      \setlength{\unitlength}{\unitlength * \real{\svgscale}}%
    \fi%
  \else%
    \setlength{\unitlength}{\svgwidth}%
  \fi%
  \global\let\svgwidth\undefined%
  \global\let\svgscale\undefined%
  \makeatother%
  \begin{picture}(1,0.62846669)%
    \lineheight{1}%
    \setlength\tabcolsep{0pt}%
    \put(0,0){\includegraphics[width=\unitlength,page=1]{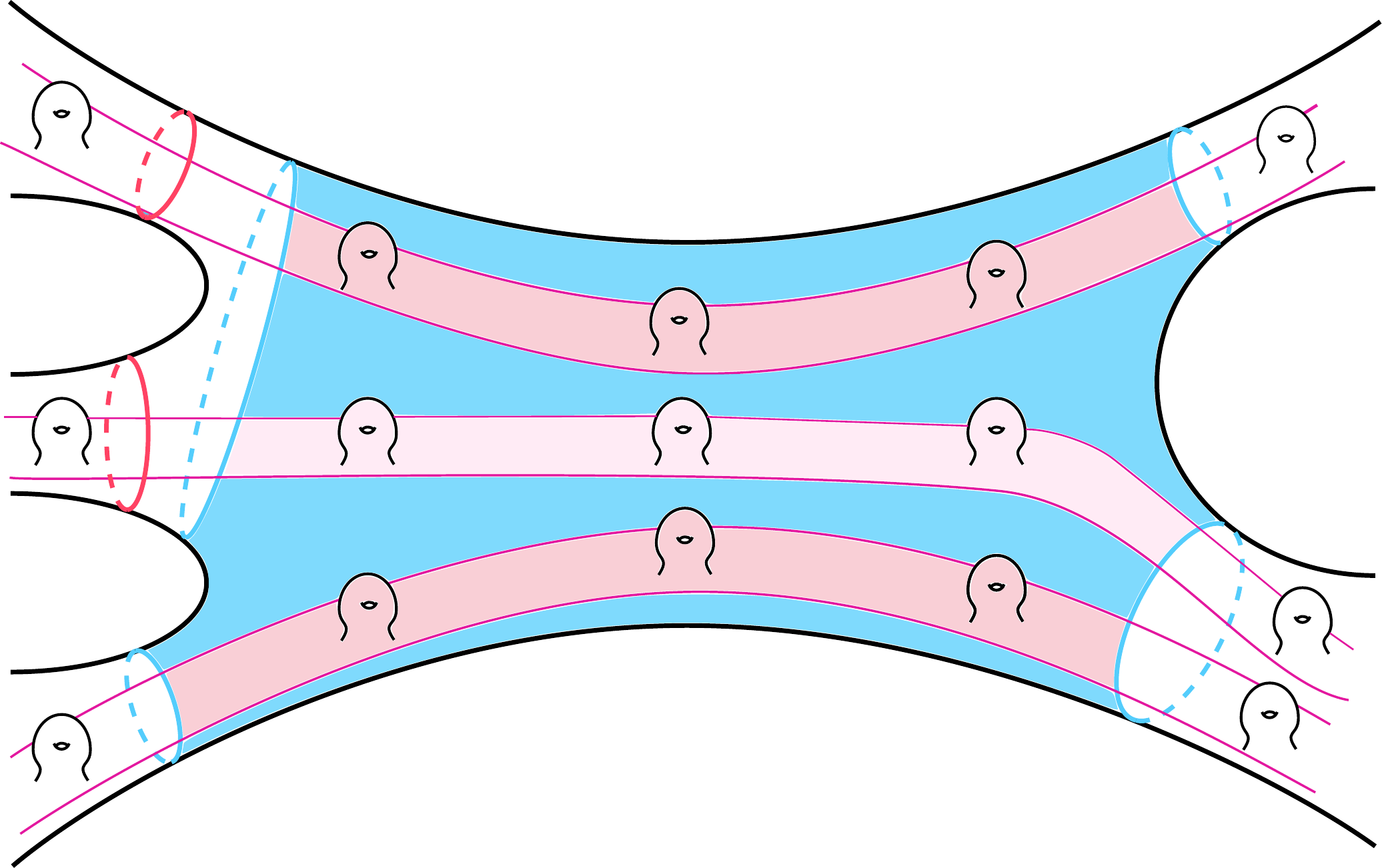}}%
    \put(0.1376617,0.56267074){\color[rgb]{0,0,0}\makebox(0,0)[lt]{\lineheight{1.25}\smash{\begin{tabular}[t]{l}$v_1$\end{tabular}}}}%
    \put(0.0971521,0.38229515){\color[rgb]{0,0,0}\makebox(0,0)[rt]{\lineheight{1.25}\smash{\begin{tabular}[t]{r}$v_2$\end{tabular}}}}%
    \put(0.0955836,0.16741287){\color[rgb]{0,0,0}\makebox(0,0)[rt]{\lineheight{1.25}\smash{\begin{tabular}[t]{r}$v_3$\end{tabular}}}}%
    \put(0.84845597,0.54384893){\color[rgb]{0,0,0}\makebox(0,0)[rt]{\lineheight{1.25}\smash{\begin{tabular}[t]{r}$v_4$\end{tabular}}}}%
    \put(0.89680799,0.26152187){\color[rgb]{0,0,0}\makebox(0,0)[lt]{\lineheight{1.25}\smash{\begin{tabular}[t]{l}$v_5$\end{tabular}}}}%
  \end{picture}%
\endgroup%

    \caption{A partially separating subsurface $C$ (shown in blue) for an end-periodic homeomorphism $\rho$ with repelling ends on the left and attracting ends on the right. The intersection of each handle strip with $C$ is a subsurface of genus $3$ with a single boundary component (shown in pink).  The boundary of $C$ intersects the strips in arcs that represent the same isotopy class as the intersection with the curves $v_1,\ldots,v_5$ defining good nesting neighborhoods of the ends.}
    \label{fig:handle-strip-intersection-section-6}
\end{figure}

\begin{proposition}\label{prop:irreducible-examples}
If $C$ is fully separating, then $f = \rho h$ is strongly irreducible. If $C$ is partially separating with respect to $\rho$, but not fully separating, then $f = \rho h$ is irreducible but not strongly irreducible. 
\end{proposition}

Assuming the lemmas and corollary proved in the remainder of \Cref{S:example}, we now provide a straightforward proof of \Cref{prop:irreducible-examples}. 

\begin{proof} Let $f = \rho h$ and $C = \supp(h)$ be as in \Cref{Convention-6.1}. Suppose that $C$ is partially separating. Then by \Cref{corollary:backwardforward-iteration}, \Cref{lemma:periodic}, and \Cref{lemma:irreducible-arcs}, $f$ is irreducible. Furthermore, if $C$ is fully separating then there are no periodic lines by \Cref{lemma:irreducible-arcs}. Hence $f$ is strongly irreducible in this case. 

Note that in the case that $C$ is partially, but not fully separating, there exist at least two ends of the same type (that is, either both positive or both negative) that are contained in the same complementary component of $C$. This implies that there is a bi-infinite line $\ell$ passing between these two ends, which is disjoint from $C$ and which lives in the complement of the handle strips that comprise $\rho$. Since $\ell$ does not intersect the support of $h$ (namely $C$) or any handle strips, then it is an AR-periodic line passing between two ends of the same type and so $f$ is irreducible, but not strongly irreducible.\end{proof}

The following lemmas will be used to show that there are no periodic lines or reducing curves.

\begin{lemma}\label{lemma:NonemptyProjection}
Let $\eta$ and $\alpha$ be as in \Cref{Convention-6.1}, and let $\gamma$ be any curve with $\pi_C(\gamma) \neq \emptyset$.
If $d_{C}(\gamma, \rho(\eta))\geq 2$, then $\pi_C (\rho^{-1}(\gamma))\neq \emptyset$. 
If $d_{C}(\gamma, \rho^{-1}(\alpha))\geq 2$, then $\pi_C (\rho(\gamma)) \neq \emptyset$. 
\end{lemma}

\begin{proof} If $d_C(\gamma,\rho(\eta)) \geq 2$, then $i(\gamma,\rho(\eta)) \neq 0$; see \eqref{E:at least 2}.  Therefore  $i(\rho^{-1}(\gamma),\eta) \neq 0$, and hence $\pi_C(\rho^{-1}(\gamma)) \neq \emptyset$.  A similar argument proves the second claim.
\end{proof}

\begin{lemma}\label{FillingPair}
Let $\alpha,\eta$, and $B$ be as in \Cref{Convention-6.1}.  Then for all $k\geq 1$, we have
\begin{enumerate}
    \item[(i)] $\pi_C (f^{-k}(\alpha))\subset h^{-1}B$
    \item[(ii)] $\pi_C (f^k(\eta))\subset B$
\end{enumerate} 
Consequently, $f^k(\eta)$ and $f^{-k}(\alpha)$ fill $C$. \end{lemma}

\begin{proof} We will split the proof into two pieces, first showing that (i) holds and then~(ii).  In each case, the proof is by induction.
 
\begin{proof}[Proof of (i)]
As noted in \Cref{Convention-6.1},  $\rho^{-1}(\alpha) \subset B$, and since $C = \supp(h)$ we have $f^{-1}(\alpha) = h^{-1}\rho^{-1}(\alpha) \subset h^{-1}B$ and our base case holds.

Assume, by induction, that there exists $m \geq 1$ such that $\pi_{C}( f^{-m}(\alpha)) \subset h^{-1}B$. Since $d_C(h^{-1}\rho(\eta),\rho(\eta)) \geq 9$, we have $d_C(f^{-m}(\alpha),\rho(\eta)) \geq 7$, and so \Cref{lemma:NonemptyProjection} implies that 
$\pi_C (\rho^{-1}f^{-m}(\alpha)) \neq \emptyset$. Since $i(f^{-m}(\alpha), \alpha) = 0$ (because $i(\alpha, f^m(\alpha)) = i(\alpha, \rho^m(\alpha)) = 0$ by construction), we also have that $i(\rho^{-1}f^{-m}(\alpha), \rho^{-1}(\alpha)) = 0$. This implies that $d_{C}( \rho^{-1}f^{-m}(\alpha),  \rho^{-1}(\alpha)) = 1$. Since $d_C(\rho(\eta),\rho^{-1}(\alpha))=1$, the triangle inequality guarantees that $\pi_C (\rho^{-1}f^{-m}(\alpha)) \subset B$. Thus, as $C = \supp(h)$, we have 
\[ h^{-1}\pi_C(\rho^{-1}f^{-m}(\alpha)) = \pi_C (h^{-1} \rho^{-1} f^{-m}(\alpha)) =\pi_C (f^{-m-1}(\alpha)) \subset h^{-1}B,\] 
as desired.
\end{proof}

\begin{proof}[Proof of (ii):] Note that $f(\eta) = \rho(\eta)$, since $\eta \subset \partial_- C$. So we immediately have that $\pi_C (f(\eta))\subset B$ and our base case holds.

Now, assume that for some $m\geq 1$, $\pi_C (f^m(\eta))\subset B$.  Therefore, we have that $h\pi_C (f^m(\eta))= \pi_C (hf^m(\eta))\subset hB$.  By \Cref{Convention-6.1}, $d_C(\rho(\eta),\rho^{-1}(\alpha)) = 1$
and $d_C(h(\rho(\eta)),\rho(\eta)) \geq 9$, and thus
$d_C(hf^m(\eta),\rho^{-1}(\alpha)) \geq 6$.
Therefore, \Cref{lemma:NonemptyProjection} implies that $\pi_C (f^{m+1}(\eta)) =\pi_C (\rho hf^m(\eta)) \neq \emptyset$. 
Since $i(\eta, f^m(\eta)) = 0$, we must also have that $i(f(\eta), f^{m+1}(\eta)) = 0$. Thus, $\pi_C( f^{m+1}(\eta))\subset B$. \end{proof}

So, we have shown by induction that for all $k\geq 1$, $\pi_C( f^{-k}(\alpha))\subset h^{-1}B$ and $\pi_C (f^k(\eta))\subset B.$ Since $d_C(h\rho(\eta),\rho(\eta))\geq 9$, we have $d_C(f^k(\eta),f^{-k}(\alpha)) \geq 5$, and thus $f^k(\eta)$ and $f^{-k}(\alpha)$ fill $C$ for all $k\geq 1$. \end{proof}

\begin{lemma}\label{ArbitraryFillingPair}
Let $\gamma^+\subset \overline U_+$ and $\gamma^- \subset \overline U_-$ be curves such that $\pi_C (\rho^{-1}(\gamma^+)) \neq \emptyset$ and  $\pi_C (\rho(\gamma^-)) \neq \emptyset$. Let $\alpha,\eta$, and $B'$ be as in \Cref{Convention-6.1}. Then for all $k\geq 1$, we have 

\begin{enumerate}
    \item [(i)] $\pi_C (f^{-k}(\gamma^+)) \subset h^{-1} B'$; and 
    \item[(ii)] $\pi_C( f^k(\gamma^-))\subset B'$.
\end{enumerate}
Consequently, $f^k(\gamma^-)$ and $f^{-k'}(\gamma^+)$ fill $C$ for all $k, k'\geq 1$.
\end{lemma}

\begin{proof}
As in \Cref{FillingPair}, we will first show that (i) holds, and then prove (ii).  In both cases, the proof is by induction.

\begin{proof}[Proof of (i)]
As noted in \Cref{Convention-6.1}, we have $d_C(\rho(\eta), \rho^{-1}(\alpha)) = 1$, and since $i(\rho^{-1}(\gamma^+),\rho^{-1}(\alpha)) = 0$, it follows that $\pi_C (\rho^{-1}(\gamma^+)) \subset B\subset B'$. Therefore, we have $h^{-1}\pi_C( \rho^{-1}(\gamma^+)) = \pi_C(f^{-1}(\gamma^+))\subset h^{-1} B'$, proving the base case. 

Assume by induction that for some $m \geq 1$, $\pi_C (f^{-m}(\gamma^+))\subset h^{-1}B'$. Since $d_C(h^{-1}(\rho(\eta)),\rho(\eta)) \geq 9$, we have $d_C(f^{-m}(\gamma^+),\rho(\eta)) \geq 7$, and \Cref{lemma:NonemptyProjection} implies that
$\pi_C( \rho^{-1}f^{-m}(\gamma^+))\neq \emptyset$. As in that proof we have  \[ h^{-1}\pi_C (\rho^{-1} f^{-m}(\gamma^+)) = \pi_C (h^{-1}\rho^{-1}f^{-m}(\gamma^+)) = \pi_C (f^{-(m+1)}(\gamma^+)) \neq \emptyset.\]
Since $i(\gamma^+, \alpha) = 0$, this implies that $i(f^{-(m+1)}(\gamma^+), f^{-(m+1)}(\alpha)) = 0$. Then as $\pi_C (f^{-(m+1)}(\alpha))\subset h^{-1}B$, \Cref{FillingPair} implies $\pi_C( f^{-(m+1)}(\gamma^+))\subset h^{-1} B'$, as desired.
\end{proof}

\begin{proof}[Proof of (ii)]
Since $i(\gamma^-, \eta) = 0$, we have $i(f(\gamma^-), f(\eta)) = 0$. Because $h$ has no effect on $\partial C$, we have $f(\eta) = \rho h(\eta) = \rho(\eta)$, and  \Cref{FillingPair} implies that $h(\pi_C(\rho(\gamma^-)) = \pi_C (f(\gamma^-)) \subset B\subset B'$, proving the base case.

Assume by induction that for some $m\geq 1$, $\pi_C (f^m(\gamma^-) )\subset B'$. Then since $d_C(\rho^{-1}(\alpha),\rho(\eta))=1$, we have
 $d_C(f^m(\gamma^-),\rho^{-1}(\alpha)) \leq 4$. The assumption that $d_C(h\rho(\eta),\rho(\eta))\geq 9$, then implies $d_C(hf^m(\gamma^-),\rho^{-1}(\alpha)) \geq 5$.  From \Cref{lemma:NonemptyProjection} we deduce that 
$\pi_C (\rho h f^m(\gamma^-))\neq \emptyset$. Now observe that since $i(\gamma^-, \eta) = 0$, we have $i(f^{m+1}(\gamma^-), f^{m+1}(\eta)) = 0$. As \Cref{FillingPair} gives us that $\pi_C (f^{m+1}(\eta))\subset B$, this implies that $\pi_C(f^{m+1}(\gamma^-))\subset B'$, as desired.
\end{proof}
Hence, we have shown by induction that for all $k\geq 1$, $\pi_C( f^{-k}(\gamma^+))\subset h^{-1} B'$ and $\pi_C (f^k(\gamma^-))\subset B'$, for any curves $\gamma^+\subset U_+$ and $\gamma^-\subset U_-$ which project nontrivially to $C$ after one handle shift. Since $d_C(h\rho(\eta),\rho(\eta))\geq 9$, it follows that $d_C(h^{-k}(\gamma^+),f^{k'}(\gamma^-))\geq 3$, and we can conclude that $f^k(\gamma^-)$ and $f^{-k'}(\gamma^+)$ fill $C$ for all $k, k'\geq 1$.
\end{proof}

\begin{corollary}\label{corollary:backwardforward-iteration}
There is no curve in $U_+$ which escapes into $U_-$ under backward iteration of the map $f$, and there is no curve in $U_-$ which escapes into $U_+$ under forward iteration of the map $f$.
\end{corollary}

\begin{proof} Let $\beta^-\subset  U_-$ and $\beta^+\subset  U_+$ be any curves. Then there exists $q,r\geq 0$ so that $f^q(\beta^-)\subset \overline U_-$ but $\pi_C(\rho f^q(\beta^-)) \neq \emptyset$, and  $f^{-r}(\beta^+)\subset \overline U_+$ but $\pi_C(\rho^{-1}f^{-r}(\beta^+)) \neq \emptyset$. 
Note that the curves $\gamma^- = f^q(\beta^-)$ and $\gamma^+ = f^{-r}(\beta^+)$ satisfy the hypotheses of \Cref{ArbitraryFillingPair}. Thus, for all $k> \max\{q, r\}$, $f^k(\beta^-)$ and $f^{-k}(\beta^+)$ project nontrivially to $C$. Consequently, $\beta^-$ cannot escape into $U_+$ under forward iteration of $f$ and $\beta^+$ cannot escape into $U_-$ under backward iteration of $f$. \end{proof}

\begin{lemma}\label{lemma:periodic}
Let $C$ be either partially or fully separating and let $f = \rho h$ be as in \Cref{Convention-6.1}. 
Then, there are no curves that are periodic under the homeomorphism $f$. That is, for curves $\delta$ on $S$, $f^m(\delta) \neq \delta$ for all $m \neq 0$.
\end{lemma}

\begin{proof} Suppose $\delta$ is a simple closed curve in $S$ which is periodic under $f$. 

Note that if any part of $\delta$ essentially intersects $U_+$ or $U_-$, then $\delta$ cannot be periodic because it will always move further into $U_+$ under positive powers of $f$ or will always move further into $U_-$ under negative powers of $f$. Hence, if $\delta$ is periodic under $f$, then $\delta\subset S  -  (U_+\sqcup U_-)$. 

There are now two cases to consider. Either $\delta$ essentially intersects $C$, or $\delta$ is contained entirely in $S  -  (U_+ \sqcup U_- \sqcup C )$. If $\delta$ is contained entirely in $S  -  (U_+ \sqcup U_- \sqcup C )$, then by the assumptions on $U_+$, $U_-$, and $C$, this means that $\delta$ must essentially intersects some handle strip. Since the components of the complement of the handle strips in $S  -  (U_+ \sqcup U_- \sqcup C )$ are contractible (being homeomorphic to closed disks, minus a finite set of points on the boundary). However, if $\delta$ essentially intersects a handle strip outside of $C$, then this portion of $\delta$ will also be carried further out into an end along this handle strip under either forward or backward iteration of $f$.

So, we assume that $\delta$ essentially intersects $C$ but is disjoint from $U_+$ and $U_-$, and that  $m > 0$ is such that $f^m(\delta) = \delta$. Let $\gamma^+\subset  U_+$ and $\gamma^-\subset  U_-$ be such that $\pi_C (\rho^{-1}(\gamma^+)) \neq \emptyset$ and $\pi_C (\rho(\gamma^-))\neq \emptyset$, as in \Cref{ArbitraryFillingPair}. 
Note that since $i(\delta, \gamma^\pm) = 0$, this implies that $i(f^k(\delta), f^k(\gamma^\pm)) = 0$ for all $k \in \mathbb{Z}$. But, \Cref{ArbitraryFillingPair} states that $f^k(\gamma^-)$ and $f^{-k}(\gamma^+)$ fill $C$ for all $k> 0$.  In particular, since $f^m(\delta) = \delta = f^{-m}(\delta)$ essentially intersects $C$, we must have $i(f^{m}(\delta),f^m(\gamma^+)) \neq 0$ or $i(f^{-m}(\delta),f^{-m}(\gamma^-)) \neq 0$, which is a contradiction.
\end{proof}

\begin{lemma}\label{lemma:irreducible-arcs}
If $C$ is partially separating, then $f$ has no AR-periodic lines. If $C$ is fully separating, then $f$ has no periodic lines. 
\end{lemma}

\begin{proof}
We will first consider the case where $C$ is partially separating. Suppose $f$ has an AR-periodic line $\ell$ with period $m$, so that $f^{mp}(\ell) = \ell$ for all $p\in \mathbb{Z}$. As $\ell$ cannot fill either $U_+$ or $U_-$, we can find curves $\beta^+\subset U_+$ and $\beta^-\subset U_-$ which are each disjoint from $\ell$. There exist integers $q, r\geq 0$ such that $f^q(\beta^-) = \rho^q(\beta^-)\subset \overline U_-$ but $\pi_C(\rho f^q(\beta^-)) \neq \emptyset$ and $f^{-r}(\beta^+) = \rho^{-r}(\beta^+)\subset \overline U_+$ but $\pi_C(\rho^{-1} f^{-r}(\beta^+)) \neq \emptyset$. 

The curves $\gamma^- = f^q(\beta^-)$ and  $\gamma^+ = f^{-r}(\beta^+)$ satisfy the hypotheses of \Cref{ArbitraryFillingPair}, and thus $f^k(\gamma^-) = f^{q+k}(\beta^-)$ and $f^{-k'}(\gamma^+) = f^{-r -k'}(\beta^+)$ fill $C$ for all $k, k'\geq 1$. Note that since $i(\ell, \beta^\pm) = 0$, this implies that $i(f^k(\ell), f^k(\beta^\pm)) = 0$ for all $k\in \mathbb{Z}$. Since $f^{mp}(\ell) = \ell$ for all $p\in \mathbb{Z}$, this implies that $i(\ell, f^{mp}(\beta^\pm)) = 0$ for all $p\in \mathbb{Z}$. But, for $p$ large enough so that $mp > \max\{q, r\}$, \Cref{ArbitraryFillingPair} implies that $f^{mp}(\beta^-)$ and $f^{-mp}(\beta^+)$ fill $C$, contradicting the fact that $\ell$ was disjoint from $\beta^-$ and $\beta^+$. 

We now consider the case where $C$ is fully-separating and suppose that $f$ has a periodic line $\ell$ with period $m$. If $\ell$ passes between two ends of $S$ of the same type (i.e. both attracting or both repelling), then in order to be invariant, $\ell$ must intersect $C$ (or else it would get shifted farther out the end under forward or backward iteration of $f$). If $\ell$ passes between two ends of $S$ of distinct types (i.e. one is attracting and one is repelling), then $\ell$ must also intersect $C$ since $C$ is fully separating. As $\ell$ can only intersect $U_-$ and $U_+$ in a finite number of arcs and rays, we can find curves $\beta^-\subset U_-$ and $\beta^+\subset U_+$ as before which are disjoint from $\ell$.
But, this gives rise to the same contradiction as in the case when $C$ was partially separating, since $f^{pm}(\beta^-)$ and $f^{-pm}(\beta^+)$ must fill $C$ for sufficiently large $p$ by \Cref{ArbitraryFillingPair}.
\end{proof}

\subsection{Sharpness of the Upper Bound}
\label{sub:sharpness}

In this subsection, we will use some of the examples of irreducible end-periodic homeomorphisms we have constructed to show that \Cref{T:upper} is asymptotically sharp, in the sense made precise by the following theorem from the introduction.  
Recall that $V_\oct$ is the volume of a regular ideal octahedron.

\medskip

\noindent{\bf \Cref{T:blah blah}}
{\em \tblahblah }

\smallskip

We will only be considering strongly irreducible end-periodic homeomorphisms $f = \rho h$, as constructed in \Cref{S:example}. Thus, by \Cref{prop:irreducible-examples}, we will require our subsurface $C = \supp(h)$ to be fully separating. 

\begin{convention}\label{convention:subsurface}
We make the same assumptions on $S,\alpha,\eta,\rho,h,C,U_\pm,B,$ and $B'$ from \Cref{Convention-6.1}, in addition to the following. We assume that $C$ is fully separating, that there is a $4$--times punctured sphere $\Sigma = \Sigma_{0,4}$ embedded in the subsurface $C \cap \rho(C) \subset C \subset S$, and that each of the four boundary curves in $\overline \Sigma$ as well as some curve $\gamma_0 \subset \Sigma$ have all distinct topological types in $S$: that is, there is no homeomorphism of $S$ taking any one of these curves to any other one. For example, the curves can be chosen to all be separating curves which each cut off a subsurface with different topological type. In particular, no two boundary curves of $\overline \Sigma$ or $\gamma_0$ are homotopic. 
For example, the conditions on $\Sigma$ hold when $\Sigma$ is embedded in $C \cap \rho(C)$ so that $S  -  \Sigma$ consists of four pairwise non-homeomorphic components and $\gamma_0$ cuts off a compact subsurface of $S$ not homeomorphic to any of the components of $S  -  \Sigma$.

We also fix a pants decomposition $P$ of $S$ which is $\rho$-invariant on $U_+$ and $U_-$, and which contains the four boundary curves of $\Sigma$, as well as $\partial C_+$, $\partial C_-$ and $\rho (\partial C_-)$.
\end{convention}

\begin{lemma}\label{lemma:TrappedCurves} With the assumptions in \Cref{convention:subsurface} (and hence also \Cref{Convention-6.1}),
for all curves $\gamma\subset C\cap \rho(C)$, we have that for all $k \geq 1$: 
\begin{enumerate}
    \item[(i)]  $\pi_C (f^{-k}(\gamma))\subset h^{-1}B'$; and 
    \item[(ii)] $\pi_C (f^k(\gamma))\subset B'$.
\end{enumerate}
\end{lemma}

\begin{proof} The proof will proceed in a similar manner as that of \Cref{ArbitraryFillingPair}, with proofs of both (i) and (ii) by induction.
\begin{proof}[Proof of (i)] Since $i(\alpha, \gamma) = 0$, we have that $i(\rho^{-1}(\alpha), \rho^{-1}(\gamma)) = 0$.  In addition, $\rho^{-1}(\gamma) \subset C$ since $\rho^{-1}(C \cap \rho(C)) \subset C$.  As $d_C(\rho^{-1}(\alpha),\rho(\eta)) = 1$, we have $h^{-1} \pi_C (\rho^{-1}(\gamma)) = \pi_C (f^{-1}(\gamma))\subset h^{-1} B'$, proving the base case.

So, assume by induction that for some $m\geq 1$, $\pi_C(f^{-m}(\gamma)) \subset h^{-1} B'$. As $d_C(h^{-1}\rho(\eta),\rho(\eta)) \geq 9$, we have $d_C(f^{-m}(\gamma),\rho(\eta)) \geq 6$, and hence \Cref{lemma:NonemptyProjection} implies that $\pi_C (\rho^{-1}f^{-m}(\gamma))\neq \emptyset$. As before, we have that $h^{-1}\pi_C (\rho^{-1} f^{-m}(\gamma)) = \pi_C (f^{-m - 1}(\gamma)) \neq \emptyset$. Since $i(\gamma, \alpha) = 0$, we have that $i(f^{-m -1}(\alpha), f^{-m-1}(\gamma)) = 0$. By \Cref{FillingPair}, we know that $\pi_C (f^{-m - 1}(\alpha))\subset h^{-1} B$, and so we must have that $\pi_C (f^{-m -1}(\gamma))\subset h^{-1}B'$, as desired.\end{proof}

\begin{proof}[Proof of (ii)]
Since $\gamma\subset C\cap \rho(C)$, we have $i(\gamma, \eta) = 0$, and hence $i(f(\gamma), f(\eta)) = 0$. As $f(\eta) = \rho h(\eta) = \rho(\eta)$,  we have $\pi_C( f(\gamma))\subset B\subset B'$, proving the base case.

Now, assume by induction that for some $m\geq 1$, $\pi_C (f^m(\gamma)) \subset B'$. Then since  $d_C(h\rho(\eta),\rho(\eta)) \geq 9$, we have $d_C(hf^m(\gamma),\rho(\eta)) \geq 6$, and hence \Cref{lemma:NonemptyProjection} implies $\pi_C(f^{m+1}(\gamma)) = \pi_C (\rho h f^m(\gamma))\neq \emptyset$. Since $i(\gamma, \eta) = 0$, it follows that we have $i(f^{m+1}(\gamma), f^{m+1}(\eta)) = 0$. As $\pi_C (f^{m+1}(\eta)) \subset B$ by \Cref{FillingPair}, we therefore have that $\pi_C (f^{m+1}(\gamma))\subset B'$, as desired. 
\end{proof}
Having proved (i) and (ii), this completes the proof.
\end{proof}

Let $L \subset \overline M_f$ be the link defined by
\[ L = \partial \overline \Sigma \cup \gamma_0 \subset S \times \{ 1 \} \subset M_f \subset \overline M_f.\] 

Write $\overline M_0 = \overline M_f  -  L$ for the complement of $L$ in $\overline M_f$.  Note that $\partial \overline M_0 = \partial \overline M_f$, but  $\overline M_0$ has five ends with neighborhoods homeomorphic to $T^2 \times (0,\infty)$. 

\begin{lemma}\label{DrilledOutHyperbolic}
The manifold $\overline M_0$ admits a convex hyperbolic structure such that $\partial \overline M_0$ is totally geodesic. 
\end{lemma}

\begin{proof} Instead of simply removing $L$ from $\overline M_f$, we instead remove an open tubular neighborhood of $L$ to produce a compact $3$-manifold whose boundary consists of $\partial \overline M_f$ together with five tori. Declaring the union of these five tori to be the paring locus, we have that the resulting manifold, together with the tori, satisfies all conditions of being a pared $3$-manifold with incompressible boundary, except possibly conditions \eqref{I:annuli} and \eqref{I:non cyclic}.  We can prove these, and simultaneously prove that the pared manifold is acylindrical, by proving that the double, $D\overline M_0$, is atoroidal (see \Cref{S:pared}).  If we do this, then we may apply \Cref{Theorem:Thurston.Geometrization} to deduce that $\overline M_0$ admits a convex hyperbolic structure with totally geodesic boundary.

So, suppose to the contrary that there is an embedded,
torus $T$ in $D\overline M_0$. There are two cases to consider: (a) $T$ is contained in (one copy of) $\overline M_0$; and (b) $T$ essentially intersects $\partial \overline M_0$ in $D \overline M_0$. 

First, suppose that $T$ is contained entirely inside of (one copy of) $\overline M_0$, and let $S' = S- L \subset \overline M_0$.  We assume that $T \pitchfork S'$ (this can always be done by a small perturbation) such that $\Gamma = T \cap S'$ has the minimal number of components among all tori isotopic to $T$ which are transverse to $S'$. We claim that $T  -  \Gamma$ is a union of annuli. Suppose to the contrary, that $T -  \Gamma$ is not a union of annuli. Then, since the Euler characteristic of $T$ is 0, there must be some disk $D$ in $T  -  \Gamma$. Note that $\partial D$ is an inessential curve in $S'$. Otherwise, $\partial D$ bounds an essential curve in $S'$ and $D$ is a compressing disk contradicting the incompressibility of $S'$. Since $\partial D$ is not essential in $S'$, then it bounds a disk in $S'$ which together with $D$ bounds a $3$-ball in $\overline M_0$. The $3$-ball can be used to define an isotopy removing the intersection of $T$ with $S'$ forming $\partial D$, which contradicts the minimality of $\Gamma$.  Therefore, $T  -  \Gamma$ consists entirely of annuli.

Let $M_0 \subset \overline M_0$ denote the interior, so that $M_0 = M_f - (L \times \{1\})$. Setting 
\[ N = S \times [0,1]  -  (L \times \{1\} \cup f(L) \times \{0\}),\]
then $M_0$ is obtained from $N$ by gluing
\[ \partial_1 N = (S -  L) \times \{1\} \subset S \times \{1\}\]
via the homeomorphism $f$ to 
\[ \partial_0 N = (S  -  f(L)) \times \{0\}  \subset S \times \{0\}.\]
Since we may assume that $T$ is contained in $M_0$, the annuli in $T  -  \Gamma$ compactify to embedded, incompressible annuli $(A,\partial A) \to (N,\partial N)$. We let $\mathcal A$ denote the union of these annuli, whose boundary is the union of $\Gamma \subset \partial_1 N \subset S \times \{1\}$ and $f(\Gamma) \subset \partial_0 N \subset S \times \{0\}$.

Since each component of $S' \subset S$ is $\pi_1$--injective, the annuli in $\mathcal A$ are in fact incompressible in $S \times [0,1]$.
Let $A \subset N$ be a component of $\mathcal A$. By work of Waldhausen (see Proposition 3.1 and Lemma 3.4 of \cite{Waldhausen}), $A$ is either boundary parallel in $S \times [0,1]$ or {\em vertical}, meaning it is isotopic to an annulus of the form $\gamma \times [0,1] \subset S \times [0,1]$ for some curve $\gamma \subset S$. These two options are illustrated in  \Cref{fig:parallel-vertical-annuli}. If all annuli are vertical, then $T$ is in fact an incompressible torus in $\overline M_f$, a contradiction.  Therefore, we assume that there is at least one component $A \subset \mathcal A$ which is boundary parallel in $S \times [0,1]$.  Because the number of components of $\Gamma$ in $\partial_1 N$ is equal to the number of components of $f(\Gamma)$ in $\partial_0 N$, we may assume that both boundary components of $A$ are in $\partial_1 N$.

\begin{figure}[htb!]
    \raggedleft
    \def\svgwidth{7.5in}
    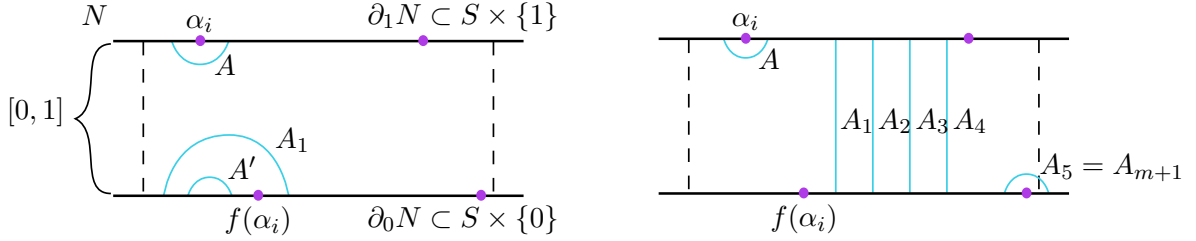
    \caption{In this schematic of $(N, \partial N)$, the curves of $L$ and $f(L)$ are represented by purple dots on $S \times \{1\}$ and $S \times \{0\}$, respectively. All of the annuli shown in the diagram on the left are \emph{boundary parallel} and the annuli $A_1, A_2, A_3,$ and $A_4$ in the diagram on the right are \emph{vertical}}
    \label{fig:parallel-vertical-annuli}
\end{figure}

Note that $A$ cannot be boundary parallel in $N$ since this could be used to define an isotopy in $M_0$ to reduce the number of components of $\Gamma = T \cap S'$.   Therefore, we may assume that $A$ has both boundary curves parallel to some $\alpha_i \subset L$ in $S \times \{1\}$ and on opposite sides of some small neighborhood of $\alpha_i$. Choosing a different component of $\mathcal A$ if necessary, we may assume that $A$ is an innermost annulus of this sort, and let $\partial_+ A,\partial_- A \subset \Gamma \subset S'$ be the two components of $\partial A$.  These curves $\partial_+ A,\partial_- A$ thus bound an annulus $G \subset S \times \{1\}$ containing $\alpha_i$ as a core curve and no other components of $\Gamma$ in its interior. 

Let $A_1 \subset \mathcal A$ be the component with one boundary component $\partial_- A_1 = f(\partial_+ A) \subset \partial_0 N$.  The other boundary component $\partial_+ A_1$ is not equal to $f(\partial_- A)$, for if it were, then $T$ would be a peripheral torus, contradicting the assumption that $T$ is essential.  If $\partial_+ A_1 \subset \partial_0 N$, then $A_1$ is boundary parallel in $S \times [0,1]$ and so $\partial A_1$ is the boundary of an annulus $G' \subset S \times \{0\}$.  This must also contain a component $f(\alpha_j) \subset f(L)$ (or else we could have reduced the number of components of $\Gamma$).  Note that both $f(\alpha_j)$ and $f(\alpha_i)$ are isotopic in $S \times \{0\}$ to $f(\partial_+ A)$ and hence are isotopic to each other.  It follows that $i=j$ by our choice of $L$, and hence $f(G)$ is properly contained in $G'$. Since $f(\partial_- A)$ is also the boundary component of some annulus $A' \subset \mathcal A$ which is disjoint from $A_1$, it must be that the other boundary component of $A'$ is contained in $G' - f(G)$.  This annulus $A'$ is also boundary parallel in $S \times [0,1]$, but note that the annulus in $S \times \{0\}$ with the same boundary curves as $A'$ is entirely contained in $G'-f(G)$, which can contain no components of $f(L)$. Thus $A'$ is boundary parallel in $N$, providing an isotopy of $T$ to reduce the number of components of intersection with $S'$, again a contradiction. This case is illustrated in the diagram on the left of \Cref{fig:parallel-vertical-annuli}. 

Therefore, the annulus $A_1$ is necessarily vertical and $\partial_+ A_1 \subset \partial_1 N$. There is thus an annulus $A_2 \subset \mathcal A$ with $\partial_- A_2 = f(\partial_+ A_1) \subset \partial_0 N$.  Either $A_2$ is boundary parallel in $S \times [0,1]$, or $\partial_+ A_2 \subset \partial_1 N$ and there is an annulus $A_3 \subset \mathcal A$ with $\partial_- A_3 = f(\partial_+ A_2)$.  Continuing in this way, ``tracing around the torus $T$" we obtain a sequence of annuli $A_1,A_2,\ldots,A_m$ which are all vertical, but for which $\partial A_{m+1} \subset \partial_0 N$.  Then $A_{m+1}$ is boundary parallel in $S \times [0,1]$, and as before, there must be some curve $f(\alpha_j) \subset f(L)$ isotopic in $S \times \{0\}$ to $\partial_- A_{m+1}$.  This means that $\alpha_j \subset L$ is isotopic to $\partial_+ A_m$ in $S \times \{1\}$.  

The annuli $A_1,\ldots,A_m$ glue together in $M_0$ to give an annulus from $\partial_+ A$ to $\partial_+ A_m$.  Since $\partial_+ A$ is isotopic to $\alpha_i$ in $S$ and $\partial_+ A_m$ is isotopic to $\alpha_j$ in $S$, inside $M_f$ we can flow this glued-up annulus forward to produce an isotopy from $f^m(\alpha_i)$ to $\alpha_j$.  Since $f$ is irreducible, $f^m(\alpha_i)$ cannot be isotopic to $\alpha_i$, and hence $i \neq j$.  On the other hand, $f^m$ is a homeomorphism from $S$ to itself sending $\alpha_i$ to $\alpha_j$ (up to isotopy), but this contradicts the fact that all components of $L$ in $S$ have different homeomorphism types. This contradiction shows that there is no essential torus in $\overline M_0$. An illustration of this case is shown in the diagram on the right of \Cref{fig:parallel-vertical-annuli}.

So, suppose now there is an embedded, non-peripheral torus $T$ in $D\overline M_0$ which essentially intersects a higher genus boundary component of $\overline M_0$ in $D \overline M_0$. An argument similar to the first case implies that we may assume that $T$ meets $\partial \overline M_0$ minimally and transversely, so that each component of $T  -  \partial \overline M_0$ compactifies to an essential annulus in one of the copies of $\overline M_0$ in $D \overline M_0$.  Let $A \subset \overline M_0$ be any such annulus.

Note that $A$ meets $\partial \overline M_0 = \partial \overline M_f$ in two essential curves, and so $A$ is incompressible in $\overline M_f$. Since $\overline M_f$ is acylindrical and since $\partial A$ are essential curves in $\overline M_f$, $A$ must be boundary parallel in $\overline M_f$. So there must be some annulus $A' \subset \partial \overline M_f$, so that $A \cup A'$ must bound a solid torus in $\overline M_f$.  Since $A$ is essential in $\overline M_0$, there must be some curve $\alpha_i \in L$, which is contained in this solid torus, but this solid torus gives a homotopy of $\alpha_i$ into $\partial \overline M_f$. This contradicts \Cref{lemma:TrappedCurves}. Therefore, $\overline M_0$ admits a complete hyperbolic metric such that the boundary is totally geodesic. \end{proof}

We will use the following result of Adams \cite[Theorem~3.1]{adams1985thrice} to conclude that the convex hyperbolic metric on $\overline M_0$ from the previous lemma makes the thrice-punctured spheres coming from $\Sigma  -  L$ totally geodesic in $\overline M_0$.
 
\begin{theorem}\label{AdamsSphere}
A properly embedded incompressible thrice-punctured sphere in a hyperbolic 3-manifold is isotopic to a totally geodesic properly embedded thrice-punctured sphere.
\end{theorem}
 
After an isotopy, we can assume that $\Sigma  -  L$ consists of pairwise disjoint, totally geodesic thrice-punctured spheres in $\overline M_0$. Next, we cut open along these surfaces to obtain four totally geodesic thrice-punctured sphere boundary components. We will make use of them in the following proof of \Cref{T:blah blah} which concludes this section.

\begin{proof}[Proof of \Cref{T:blah blah}.]

Fix a $4$-holed sphere $\Sigma$ and a pants decomposition $P$ as in \Cref{convention:subsurface}. As before, $L$ denotes the union of the four boundary curves of $\overline \Sigma$ together with the pants curve $\gamma_0$ in $\Sigma$, 
and $\overline M_0 := \overline M_f - (L \times \{1\})$. By \Cref{DrilledOutHyperbolic}, $\overline M_f$ admits a convex hyperbolic metric with totally geodesic boundary. 
By \Cref{AdamsSphere}, we can cut along the two totally geodesic thrice-punctured spheres inside $\overline M_0$ which yields a new 3-manifold, $\overline M_0'$ with four more boundary components, each of which is a thrice-punctured sphere.

Now isometrically glue $2k$ copies of the four-holed sphere block, $\mathring{\mathcal{B}}$, (with its complete hyperbolic metric with totally geodesic boundary) into $\overline M_0'$, stacked vertically so that the bottom of the first block is glued to the bottom new boundary component and so that the top of the last block is glued to the top new boundary component.  We call this new manifold $\overline M_k$, and note that \[ \Vol(\overline M_k) = \Vol(\overline M_0)+4kV_{\oct}.\]

Next observe that we may alternatively view $\overline M_k$ as being obtained from $\overline M_f$ by removing a link $L_k \subset \overline M_f$, each component of which is contained in fiber of $M_f$. 
More precisely, we can take $L_k$ to consist of $5+2k$ curves, consisting of $(\partial \overline \Sigma \cup \gamma_0)  \times \left\{\tfrac14\right\}$ together with $2k$ curves of the form $\beta_i \times \left\{\tfrac14 + \tfrac{i}{4k}\right\}$, where $\beta_i \subset \Sigma$ (the choice of interval $[\tfrac14,\tfrac34]$ is arbitrary: up to isotopy, we could of course take any interval in $[0,1]$). For any integer $s$, we let $\overline M_{k,s}$ denote the result of performing $(1,s)$--Dehn filling on each torus cusp of $\overline M_k$. For $s=0$, this gives us back the original mapping torus, $\overline M_{k,0} \cong \overline M_f$. 

Suppose that $P$ and $f^{-1}(P)$ differ by $r$ pants moves, and suppose that the pants move which occurs on $\Sigma$ occurs at the $m$-th vertex in the path between $P$ and $f^{-1}(P)$ in $\mathcal{P}(S)$. As in \Cref{Section:quotient construction}, we may view this move as occurring in the fiber $S\times \left\{\tfrac{m}{r}\right\}$. Gluing in the $2k$ pants blocks and performing $(1,0)$--Dehn filling to get $\overline M_{k,0} \cong \overline M_f$ corresponds to choosing a new path between $P$ and $f^{-1}(P)$ which now has length $r + 2k$ (we have essentially inserted $2k$ redundant pants moves that occur inside $\Sigma$). So, we may perform an isotopy so that the initial pants move on $\Sigma$ now occurs in the fiber $S\times \left\{\tfrac{m}{N + 2k} \right\}$, and so that the $i$-th pants block (after filling) has one boundary component in the fiber $S\times \left\{\tfrac{m + i - 1}{N + 2k}\right\}$ and the other boundary component in the fiber $S \times \left\{\tfrac{m + i}{N + 2k}\right\}$.

By \Cref{P:Stallings}, $\overline M_{k,s} \cong \overline{M}_{fD_{k,s}}$, where $D_{k,s}$ is a product of powers of Dehn twists about the curves in $L_k$ (projected into $\Sigma \subset S$). Letting $D \overline M_k$ denote the double of $\overline M_k$ over the boundary, for each $s >0$ we can view the double $D \overline M_{k,s}$ of $\overline M_{k,s}$ over its boundary as distinct Dehn fillings of $D \overline M_k$.  By \Cref{T:filling volume}, for $s$ sufficiently large, $D \overline M_{k,s}$ is hyperbolic, and as $s \to \infty$, $\Vol(D \overline M_{k,s}) \to \Vol(D \overline M_k)$.  By Mostow rigidity, the involution interchanging the two sides of the double is isotopic to an isometry with respect to the hyperbolic metric on $D \overline M_{k,s}$ and fixes $\partial \overline M_{k,s}$.  Consequently, the hyperbolic metric on $\overline M_{k,s}$ also has totally geodesic boundary and \[ \Vol(\overline M_{k,s}) \to \Vol(\overline M_k) = \Vol(\overline M_0) + 4kV_\oct.\]
For each $k >0$, let $s_k$ be such that
\[\Vol(\overline M_{k,s_k}) \geq 4k V_\oct.\]

We note that the end-periodic monodromy $f_k = fD_{k,s_k} = \rho h D_{k,s_k}$ for $\overline M_{k,s_k}$ is strongly irreducible. This is because our only requirement on the map $h$ in \Cref{S:example} was that $d_C(h\rho(\eta),\rho(\eta)) \geq 9$ in $\mathcal{AC}(C)$. As $L_k$ is disjoint from $\rho(\eta)$, this implies $D_{k,s_k}$ must preserve $\rho(\eta)$. Hence, $h D_{k,s_k} \rho(\eta) = h \rho (\eta)$, and therefore we have $d_C(h D_{k,s_k} \rho(\eta),\rho(\eta))\geq 9$.

By \Cref{T:upper}, we have that $\Vol(\overline M_{f_k}) \leq V_\oct \tau(f_k) \leq V_\oct(n_T+2n_S+4k)$, where $n_S$ and $n_T$ are the number of pants moves on four-punctured spheres and tori, respectively, in the original path between $P$ and $f^{-1}(P)$ (note that all $2k$ pants move from the blocks occur on four-punctured spheres).  On the other hand, $\Vol(\overline M_{f_k}) \geq 4k V_\oct$, and therefore
\[ |\Vol(\overline M_{f_k}) - V_\oct \tau(f_k)| \leq V_\oct(n_T+2n_S). \]
Since $\Vol(\overline M_{f_k}) \geq 4k V_\oct \to \infty$ as $k \to \infty$, this provides the required sequence $\{f_k\}$.\end{proof}

\subsection{Large component translation} \label{S:large distance fixed f}
Here we prove the following corollary from the introduction.\\

\noindent
{\bf \Cref{C:large distance}}
{\em \largedistance}

\smallskip

\begin{proof}
By \Cref{T:component volume}, it suffices to find a sequence of $f$--invariant components $\Omega_k \subset \mathcal P_f(S)$ so that
\[ \Vol(\overline M_f - P_{\Omega_k}) \to \infty \]
as $k \to \infty$. For this, it is useful to observe that any pants decomposition $P$ of $\partial \overline M_f$ arises as $P_\Omega$ for some $f$--invariant component $\Omega \subset \mathcal P_f(S)$ if and only if the preimage of $P$ in $\mathcal U_\pm \times \{\pm \infty\}$ are pants decompositions. This is because $\Omega$ being $f$--invariant is equivalent to requiring that any $P_0 \in \Omega$ has the property that $P_0$ and $f(P_0)$ differ by a finite number of pants moves. This is equivalent to $P_0$ being $f$--invariant when restricted to some nesting neighborhoods, and since these nesting neighborhoods form ``half" the infinite cyclic covers $\mathcal U_\pm \to S_\pm$, this is equivalent to $P \subset S_\pm$ lifting to pants decompositions on $\mathcal U_\pm$.

Now choose any pants decomposition $P$ in an $f$--invariant component $\Omega \subset \mathcal P_f(S)$ and consider $P_\Omega \subset \partial \overline M_f$.  Let $\gamma_0 \subset P_\Omega$ be any component such that $\partial \overline M_f  -  (P_\Omega -  \gamma_0)$ contains a four-punctured sphere component $\Sigma$.  Consider the convex hyperbolic metric on $\overline M_f  -  P_\Omega$ with totally geodesic boundary consisting of thrice-punctured spheres, and as in the proof of \Cref{T:blah blah}, isometrically glue on $2k$ copies of $\mathring{\mathcal B}$ stacked vertically so that the bottom of the first block is glued to $\Sigma - \gamma_0$.  The resulting manifold $\overline M_k$ is homeomorphic to $\overline M_f  -  (P_\Omega \cup L_k)$, where $L_k \subset M_f \subset \overline M_f  -  P_\Omega$ is a link in the interior.  After an isotopy, we can assume that the curves in $L_k$ are of the form $\beta \times \{x_\beta\}$ with respect to some product structure on a collar neighborhood $N \cong \partial \overline M_f \times [0,1]$ of $\partial \overline M_f$, with $\beta \subset \Sigma$ and $x_\beta \in (0,1)$ for all $\beta$.

Similar to \Cref{P:Stallings}, performing $(1,s)$--Dehn filling on each torus cusp of $\overline M_k$ simply changes the product structure on the collar neighborhood of $\partial \overline M_f$, and consequently changes the pants decomposition on the boundary by a composition of powers of Dehn twists in the curves $L_k$.  We can therefore denote the result of this Dehn filling as $\overline M_f  -  P_{k,s}$ for some pants decomposition $P_{k,s}$ of $\partial \overline M_f$.  On the other hand, since all the Dehn twists occur on curves in $\Sigma$, and the entire subsurface $\Sigma$ lifts to $\mathcal U_\pm$, it follows that $P_{k,s}$ lifts to pants decompsitions on $\mathcal U_\pm$.  Therefore, $P_{k,s} = P_{\Omega_{k,s}}$ for some $f$--invariant component $\Omega_{k,s} \subset \mathcal P_f(S)$.  As in the proof of \Cref{T:blah blah}, appealing to \Cref{T:filling volume}, we have
\[\Vol(\overline M_f - P_{\Omega_{k,s}}) \to \Vol(\overline M_f - P_\Omega) + 4k V_\oct. \]
We may therefore choose $s_k$ sufficiently large so that
\[ \Vol(\overline M_f - P_{\Omega_{k,s_k}}) \geq 4k V_\oct.\]
The term on the right tends to $\infty$, and thus setting $\Omega_k = \Omega_{k,s_k}$ completes the proof.
\end{proof}

\bibliographystyle{alpha}
  \bibliography{main}

\end{document}